\documentclass[12pt]{amsart}

\usepackage{latexsym,fancyhdr,amssymb,color,amsmath,amsthm,graphicx,listings,comment}
\usepackage[section]{placeins} 
\newtheorem{thm}{Theorem} \newtheorem{lemma}{Lemma} \newtheorem{propo}{Proposition} 
 \newtheorem{coro}{Corollary} 
\let\paragraph\subsection
 \newcommand{\ZZ}{\mathbb{Z}}
  \def\G{\mathcal{G}}
\def\osquare{\mbox{\ooalign{$\times$\cr\hidewidth$\square$\hidewidth\cr}} }

\title{On the arithmetic of graphs}
\author{Oliver Knill}
\date{Jun 18, 2017}
\address{Department of Mathematics \\ Harvard University \\ Cambridge, MA, 02138 }

\begin{document}
\maketitle

\begin{abstract}
The Zykov ring of signed finite simple graphs with topological join as addition and compatible 
multiplication is an integral domain but not a unique factorization domain. We know
that because by the graph complement operation it is isomorphic to the strong 
Sabidussi ring with disjoint union as addition. We prove that the Euler characteristic
is a ring homomorphism from the strong ring to the integers by demonstrating that the strong ring
is homotopic to a Stanley-Reisner Cartesian ring. More generally, the Kuenneth formula holds on the strong ring
so that the Poincar\'e polynomial is compatible with the ring structure. 
The Zykov ring has the clique number as a ring homomorphism. Furthermore, the Cartesian ring has the 
property that the functor which attaches to a graph the spectrum of its connection Laplacian is
multiplicative. The reason is that the connection Laplacians do tensor under multiplication, similarly to what the adjacency 
matrix does for the weak ring. The strong ring product of two graphs contains both the weak and direct product graphs
as subgraphs. The Zykov, Sabidussi or Stanley-Reisner rings are so manifestations of a 
network arithmetic which has remarkable cohomological properties, dimension and spectral compatibility
but where arithmetic questions like the complexity of detecting primes or factoring are not yet
studied well. We illustrate the Zykov arithmetic with examples, especially from the subring generated 
by point graphs which contains spheres, stars or complete bipartite graphs. While things 
are formulated in the language of graph theory, all constructions generalize to the larger category of 
finite abstract simplicial complexes.
\end{abstract}

\section{Extended summary}

\paragraph{}
Finite simple graphs extend to a class $\G$ of signed finite simple graphs which carry 
three important commutative ring structures: the {\bf weak ring} $(\G,\oplus,\square,0,1)$, 
the {\bf direct ring} $(\G,\oplus,\otimes,0,1)$ and the {\bf strong ring}
$(\G,\oplus,\osquare,0,1)$ \cite{Sabidussi,ImrichKlavzar,HammackImrichKlavzar}. 
In all three cases, the disjoint union $\oplus$ is the addition, the empty graph the zero 
element and the one point graph $K_1$ the one element. The weak ring product $\square$
is the Cartesian product for graphs which corresponds to the tensor product of
adjacency matrices, the direct ring product $\otimes$ is also known as the {\bf tensor product} of graphs.
The strong product $\osquare$ of Sabidussi combines edge sets of the other two. In each case, taking graph complements 
produces dual rings in which the addition is the Zykov join $+$ which corresponds to the join in topology, 
and which preserves the class of spheres. 
We first observe that the dual to the Zykov ring introduced in \cite{Spheregeometry} is the strong ring so
that also the Zykov ring was already known. The {\bf Sabidussi unique prime factorization theorem for connected graphs} and the 
Imrich-Klavzar examples of non-unique prime factorization in the disconnected case or in general for the direct product
and are so inherited and especially hold for the Zykov ring which is therefore, like the Sabidussi ring,
an {\bf integral domain} but {\bf not a unique factorization domain}.

\paragraph{}
The {\bf clique number} from $\G \to \mathbb{Z}$ assigning to a graph $G$ the number ${\rm dim}(G)+1$, where ${\rm dim}$ is
the maximal dimension of a simplex in $G$, extends to a ring homomorphism from the Zykov ring to the integers. 
Inherited from its dual, also the Zykov addition, the join, has a unique additive prime factorization, 
where the primes are the graphs for which the dual is connected. 
We observe that the {\bf Euler characteristic} $\chi$  is a ring homomorphism from the strong ring to the integers 
and that the {\bf Kuenneth formula} holds: the map from the graph to 
its {\bf Poincar\'e polynomial} $p(G)$ is a ring homomorphism $p(G+H)=p(G)+p(H)$ and $p(G \osquare H) = p(G) p(H)$. 
To do so, we observe that the strong product $\osquare$ is homotopic to the graph product $\times$ 
treated in \cite{KnillKuenneth} and which is essentially the Stanley-Reisner
ring product when written down algebraically. 

\paragraph{}
We also note that the tensor product of the connection 
Laplacians of two graphs is the connection Laplacian $L$ of this Stanley-Reisner product $\times$, 
confirming so that the energy theorem equating the Euler characteristic $\chi(G)$ with the sum $\sum_{x,y} g(x,x)$ 
of the matrix entries of $g=L^{-1}$ in the case of simplicial complexes, extends to the full 
Stanley-Reisner ring generated by complexes. 
It follows that the spectra of the connection Laplacians $L(G)$ of a complex $G$ satisfy 
$\sigma(L(G) + L(H)) = \sigma(L(G)) \cup \sigma(L(H))$ and $\sigma(L(G \times H)) = \sigma(L(G)) \sigma(L(H))$. 
So, not only the potential theoretical energy, but also the individual energy spectral values are compatible
with the arithmetic. The Zykov ring relates so with other rings sharing so extraordinary 
topological, homological, potential theoretical and spectral properties modulo duality or homotopy. 

\section{Graph arithmetic}

\paragraph{}
Let $\G_0$ denote the category of finite simple graphs $G=(V,E)$. With the disjoint union $\oplus$ as addition,
the {\bf weak product} $\square$, the {\bf direct product} $\otimes$ 
and the {\bf strong product} $\square$ produce new graphs with vertex set $V(G) \times V(H$) \cite{Sabidussi}. 
Let $(\G,\oplus)$ be the group generated by the monoid $(\G_0,\oplus)$. The tensor product is also called direct product and
the weak product the graph product. The definitions of the products are specified by giving the edge sets. We have
$E(G \square H)=\{ ((a,b),(c,d)) | a=c,(b,d) \in E(H) \} \cup \{ ((a,b),(c,d)), b=d,(a,c) \in E(G) \}$
and $E(G \otimes H)=\{ ((a,b),(c,d) | (a,c) \in E(G), (b,d) \in E(H) \}$ and
$E(G \osquare H) = E(G \square H) \cup E(G \otimes H)$. The unit element in all three monoids
is the one point graph $K_1$ called $1$. A graph $G \in \G$ is {\bf prime} in a ring if they 
are only divisible by $1$ and itself. All three products have a prime factorization but 
no unique prime factorization \cite{ImrichKlavzar}. All three products belong to a ring if the addition $\oplus$
is the disjoint union Grothendieck augmented to a group by 
identifying $A \ominus B \sim C \ominus D$ if $A \oplus C \oplus K=B \oplus D \oplus K$ 
for some $K$. Since we will see that unique additive prime factorization holds in $(\G,\oplus)$, the 
additive primes being the connected graphs, one can simplify this and write every element in the ring as an 
ordered pair $(A,B)$ of graphs and simply write also $A \ominus B$ or $A-B$ and multiply the usual way
like for example $(A-B) \osquare (C-D)$ $= (A \osquare C) \oplus (B \osquare D)$ 
$\ominus (A \osquare D) \ominus (B \osquare C)$.

\paragraph{}
The weak product seems first have appeared in the Principia Mathematica of Whitehead and Russell.
The three products appeared together in the pioneering paper \cite{Sabidussi}, where they were first mathematically 
studied. In \cite{Spheregeometry}, we took the Zykov join \cite{Zykov} $(V,E) + (W,F)=(V \cup W,E \cup F \cup VW$,
where VW denotes all edges connecting any $V$ to any $W$ as a ``sum" and constructed a compatible product
$G \cdot H =$ $(V \times W,$  $\{ ((a,b),(c,d)) | (a,c) \in E(G)$ or $(b,d) \in E(H) \})$. 
Also $(\G,+,\cdot)$ produces a ring. Finally, 
there is a Cartesian product on graphs which satisfies the Kuenneth formula \cite{KnillKuenneth}:
in that product, the vertex set is $V(G_1) \times V(H_1)$, the product of the vertex sets of the
Barycentric refinements $G_1,H_1$ of $G,H$ and two vertices are connected if one product simplex is contained in the other.
It is a ring when extended to CW-complexes or interpreted as the product in the Stanley-Reisner ring.
All three rings $(\G,\oplus,\square,0,1),(\G,\oplus,\otimes,0,1),(\G,\oplus,\osquare,0,1)$ 
have the property that they are defined in $\G$ and that the addition of two graphs $(V,E)$ and $(W,F)$ 
has the vertex set $V \cup W$ and that the product in the ring has the vertex set $V \times W$ 
$=\{ (v,w) \; | \; v \in V(G), w \in V(H) \}$. 

\paragraph{}
The topic of graph product is rich as the main sources for this topic \cite{ImrichKlavzar,HammackImrichKlavzar}
show. The additive operations in these works are usually the disjoint union. The dualization
operation appears in the handbook \cite{HammackImrichKlavzar} who mention that there are exactly 20 
associative products on graphs. In this sense, the ring considered in \cite{Spheregeometry} is not new. 
Maybe because it is just the dual to the disjoint union, the Zykov product $\cdot$ appears not have been
studied much in graph theory. \cite{HararyGraphTheory} attribute the construction to a paper of 1949 \cite{Zykov}. It has
the same properties than the join in topology covered in textbooks like \cite{RourkeSanderson,Hatcher}.
The join especially preserves spheres, graphs which have the property that all unit spheres are spheres
and where removing one vertex produces a contractible graph. Examples are $P_2+P_2=S_4$ or $S_4+S_4$ being
a 3-sphere. 

\paragraph{}
We will look in more detail at the join monoid $(\G,+)$ which is the additive Zykov monoid. 
About spectra of the Kirchhoff Laplacian, we know already from \cite{Spheregeometry} that 
$|V(G)| + |V(H)|$ is an eigenvalue of $G+H$ so that $K_n \cdot G=G+G+ \cdots + G$ has an eigenvalue $n |V(G)|$ of
multiplicity $n-1$. We also proved about the {\bf ground state}, the smallest non-zero eigenvalue $\lambda_2$
of a graph that $\lambda_2(G+H) = {\rm min}(|V(H)|,|V(K)|) + {\rm min}(\lambda_2(G),\lambda_2(H)$.
Furthermore, the eigenvalues of the Volume Laplacian $L_d$ if $L_1 \oplus ... \oplus L_d = (d+d^*)^2$
is the {\bf Hodge Laplacian} defined by the incidence matrices $d$ satisfy $\sigma_V(G) + \sigma_V(H) = \sigma_V(G+H)$.
Already well known is $\sigma(G) + \sigma(H) = \sigma(G \osquare H)$ \cite{HammackImrichKlavzar}. In some sense,
this volume Laplacian result about the join dual to the fact that the Kirchhoff Laplacian $L_0$ (which can 
be seen as the Poincar\'e dual to the Volume Laplacian) the disjoint union $\oplus$ satisfies
$\sigma(G) + \sigma(H) = \sigma(G \oplus H)$. The tensor and strong products have no obvious nice
spectral properties. But we will see below that if we look at the spectrum of the {\bf connection Laplacian}, an 
operator naturally appearing for simplicial complexes, then the spectrum behaves nicely for the Cartesian 
product of finite abstract simplicial complexes. This product 
corresponds to the multiplication in the Stanley-Reisner ring but which is not an 
abstract simplicial complex any more. But as the product is homotopic to the strong product, we can stay 
within the category of graphs or simplicial complexes. 

\paragraph{} Here is an
overview over the definitions of the three rings with $\oplus$ as addition: \\

\begin{tiny} \begin{tabular}{|l|l|l|} \hline
Ring   & Addition                          &Multiplication                                                                    \\ \hline
Weak   & $G \oplus H=(V \cup W,E \cup F)$  &$G \square  H=(V \times W,\{ ((a,b),(a,d)) \} \cup \{ ((a,b),(c,b)) \})$            \\ \hline
Tensor & $G \oplus H=(V \cup W,E \cup F)$  &$G \otimes  H=(V \times W,\{((a,b),(c,d))|(a,c) \in E \; {\rm and} \; (b,d) \in F \})$ \\ \hline
Strong & $G \oplus H=(V \cup W,E \cup F)$  &$G \osquare H=(V \times W, E(G \square H) \cup E(G \otimes H)$                       \\ \hline
\end{tabular} \end{tiny}

We will relate the strong product $\osquare$ to its dual product, 
the Zykov product $\cdot$ in which the Zykov join $+$ is the addition 
and then deform the multiplication from the strong to the Cartesian product (which however
does not define a ring on $\G$ as associativity got removed by pushing the product back to
a graph using the Barycentric refinement).

\begin{tiny} \begin{tabular}{|l|l|l|} \hline
Ring     & Addition                            &Multiplication                                                                    \\ \hline
Zykov    & $G + H=(V \cup W,E \cup F \cup VW)$ &$G \cdot  H=(V \times W,\{((a,b),(c,d))|(a,c) \in E \; {\rm or} \;  (b,d) \in F \})$ \\ \hline
Cartesian& $G \oplus H=(V \cup W,E \cup F)$    &$(G \times H)_1=(c(G) \times c(H), \{(a,b) | a \subset b \; {\rm or} \;  b \subset a \})$  \\ \hline 
\end{tabular} \end{tiny}

\paragraph{}
The next two figures illustrating the five mentioned products: 

\begin{figure}[!htpb]
\scalebox{0.13}{\includegraphics{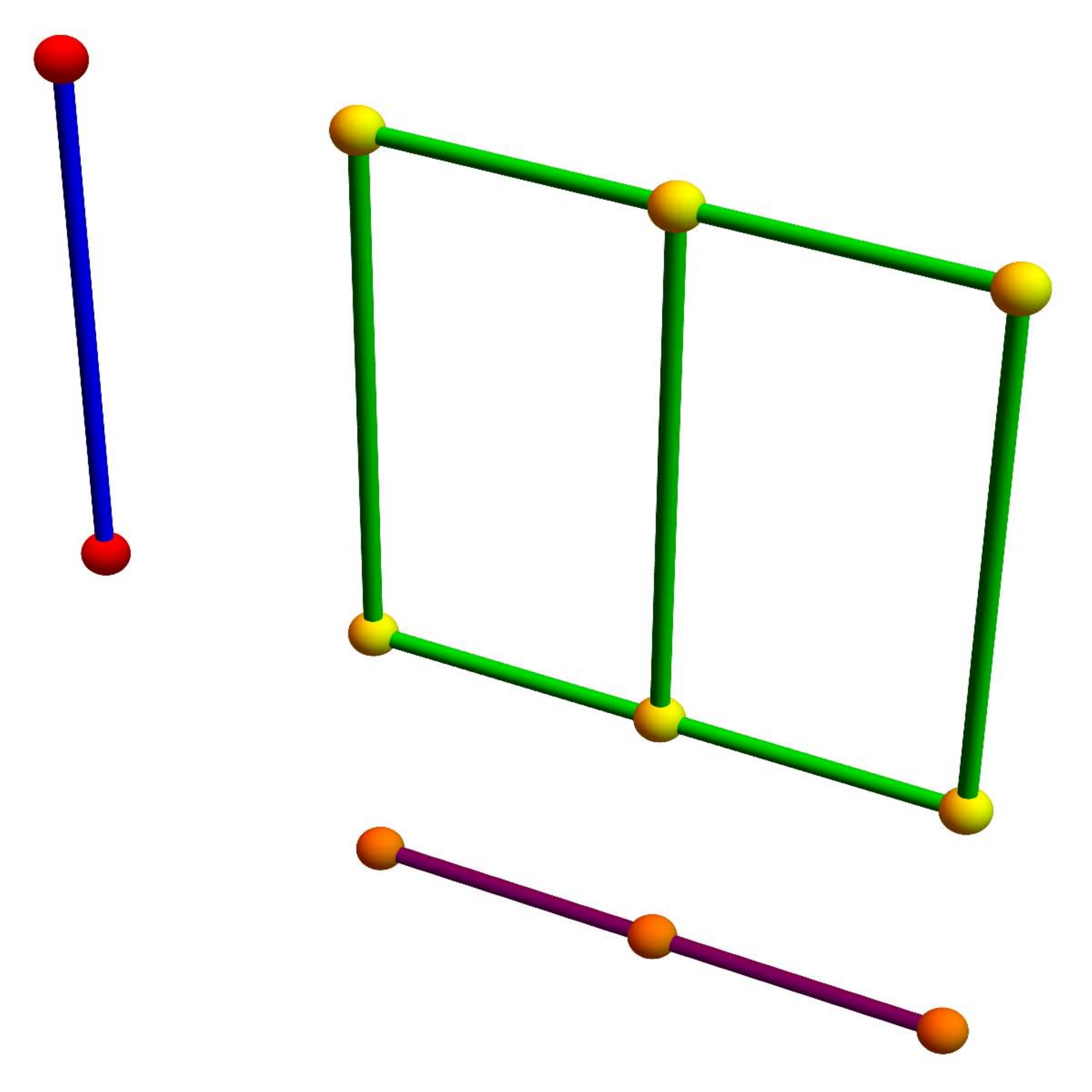}}
\scalebox{0.13}{\includegraphics{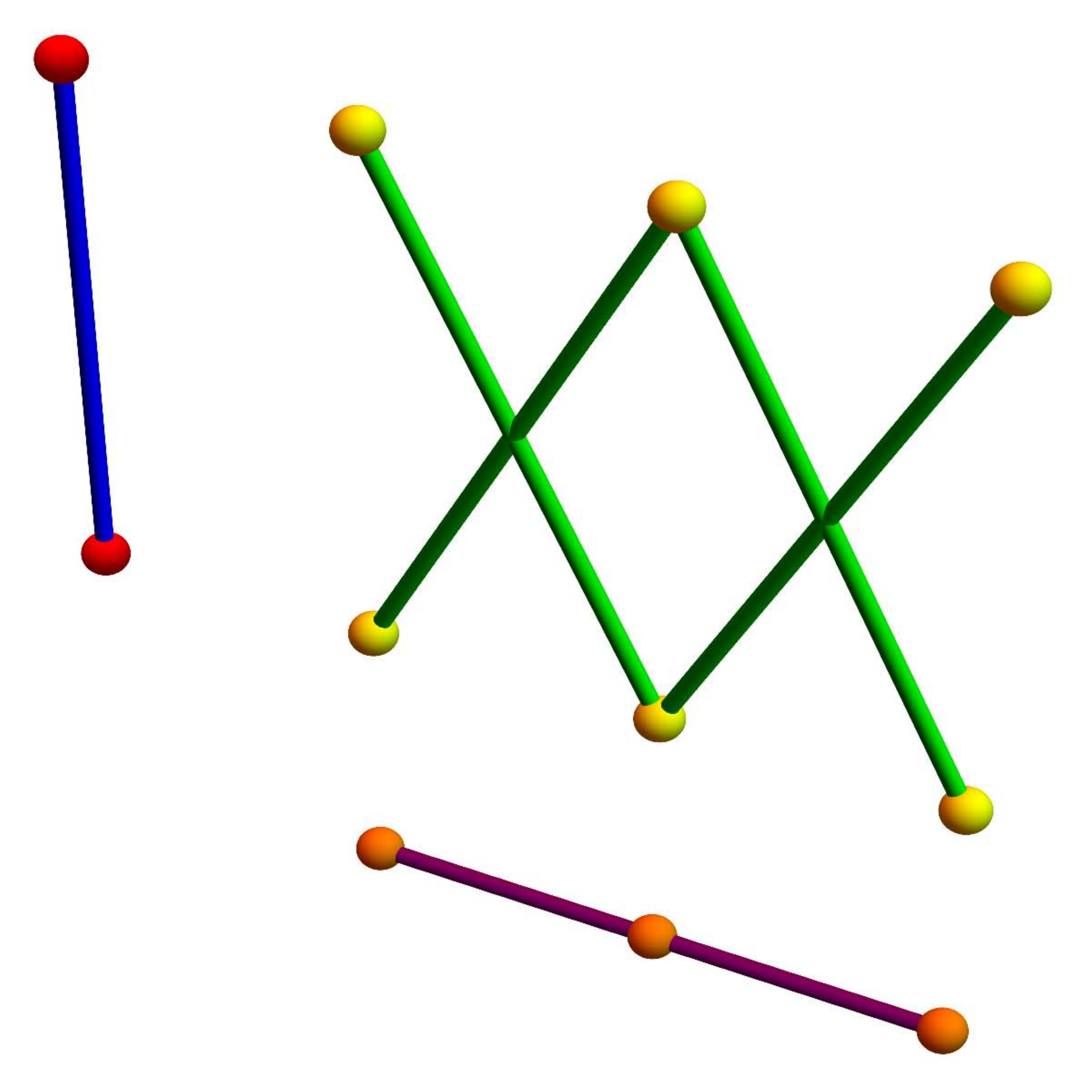}}
\scalebox{0.13}{\includegraphics{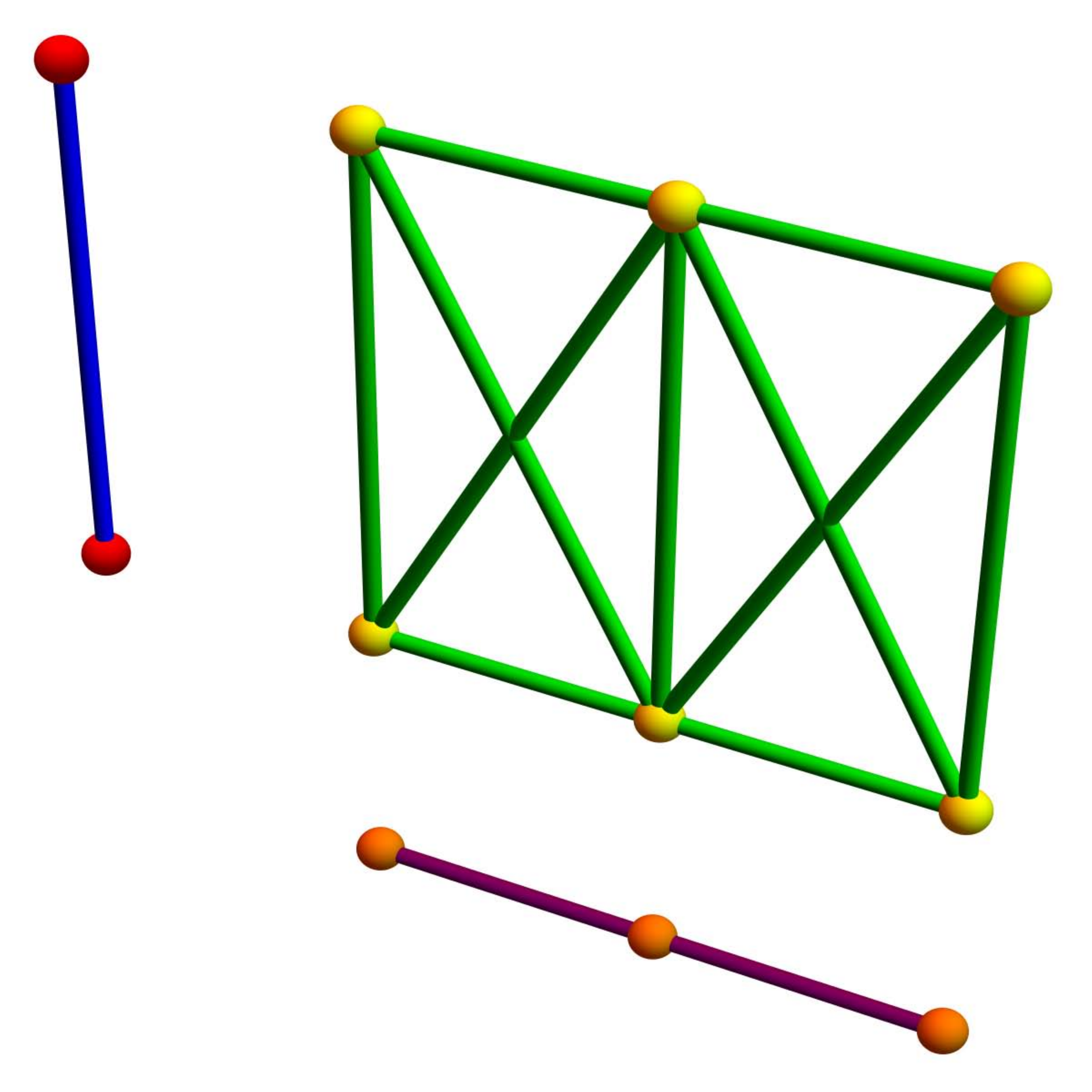}}
\caption{
The multiplication of $K_2$ with $K_2$ in the weak, direct and strong rings. 
}
\end{figure}

\begin{figure}[!htpb]
\scalebox{0.14}{\includegraphics{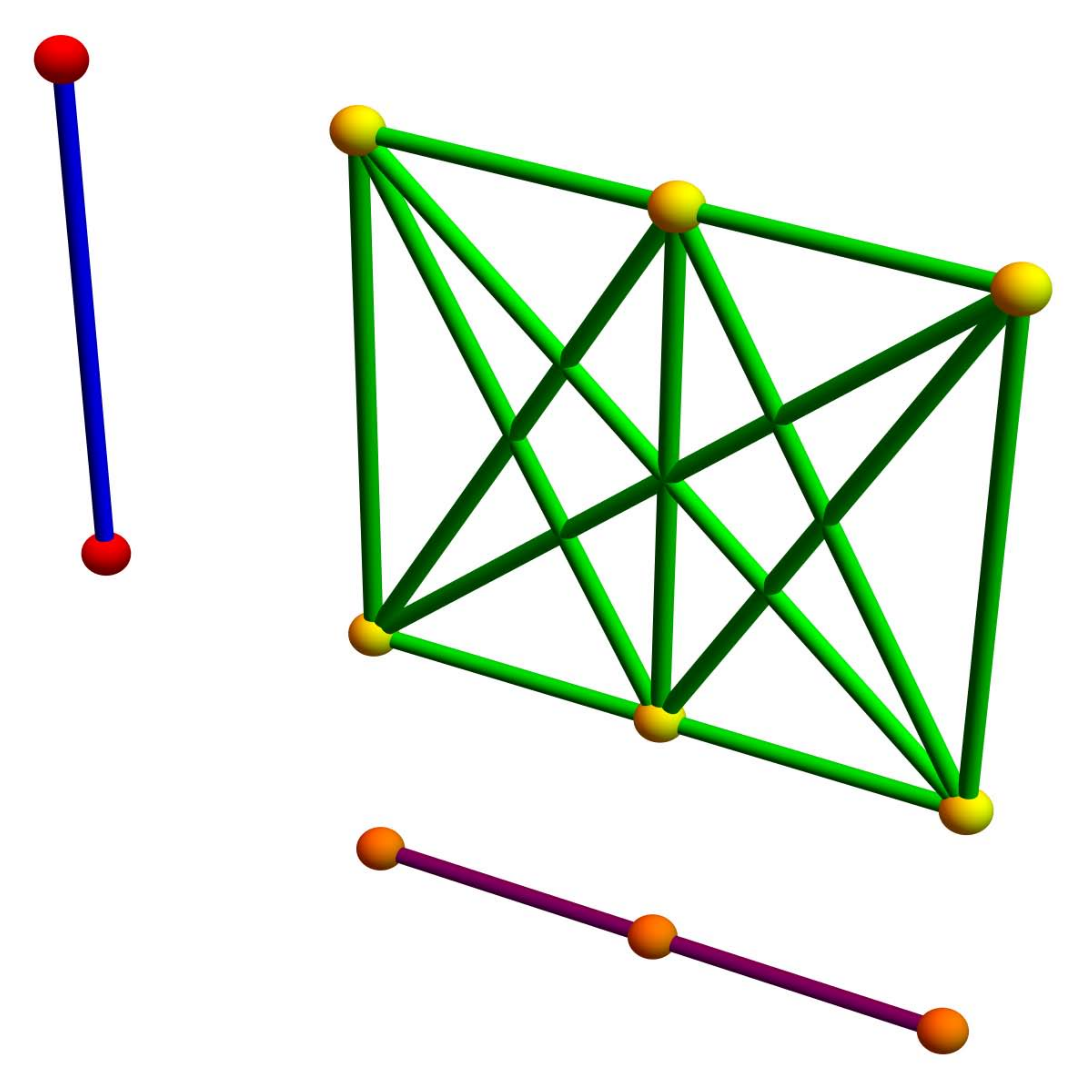}}
\scalebox{0.14}{\includegraphics{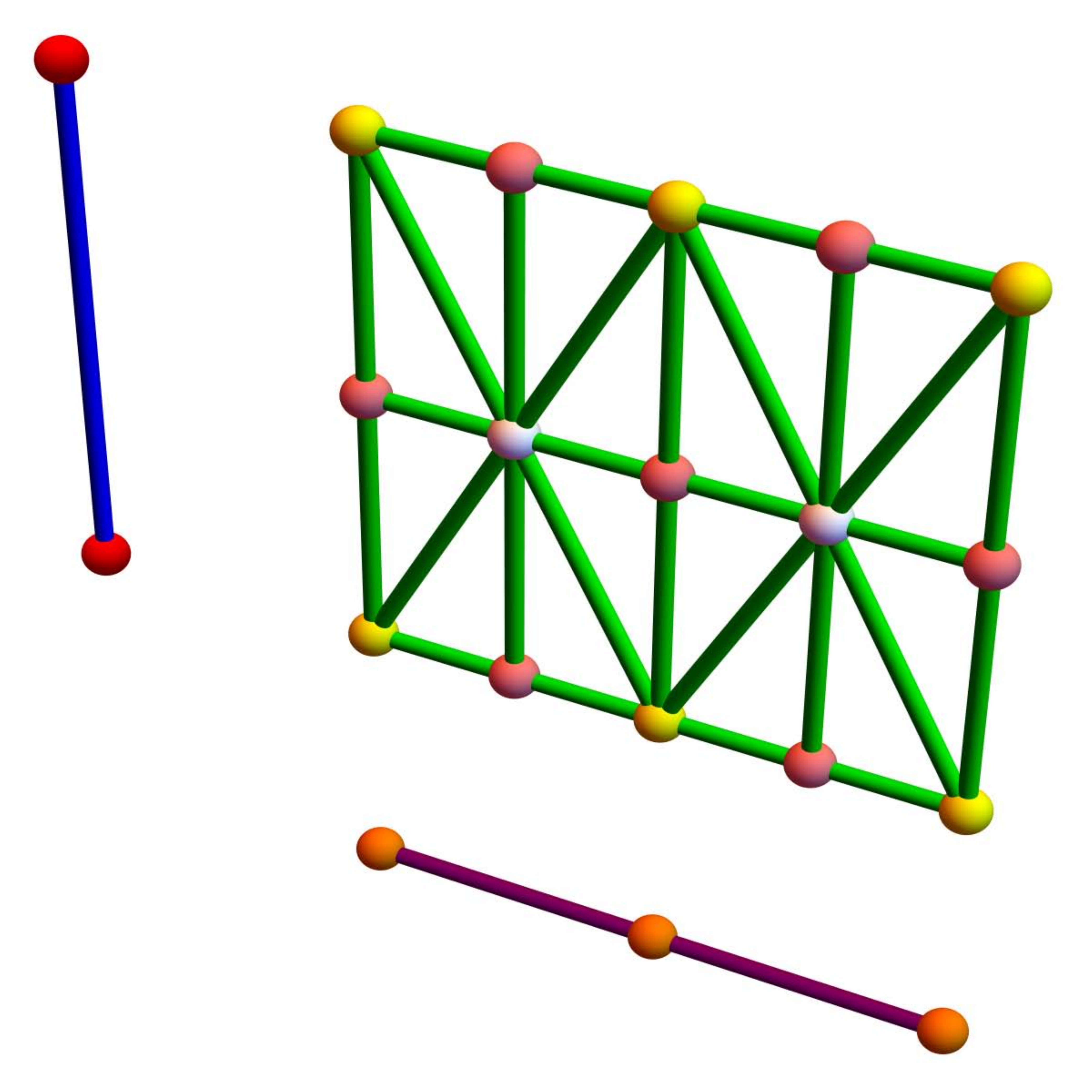}}
\caption{
The multiplication of $K_2$ with $K_2$ in the Zykov and Cartesian simplex product. 
}
\end{figure}

\begin{figure}[!htpb]
\scalebox{0.13}{\includegraphics{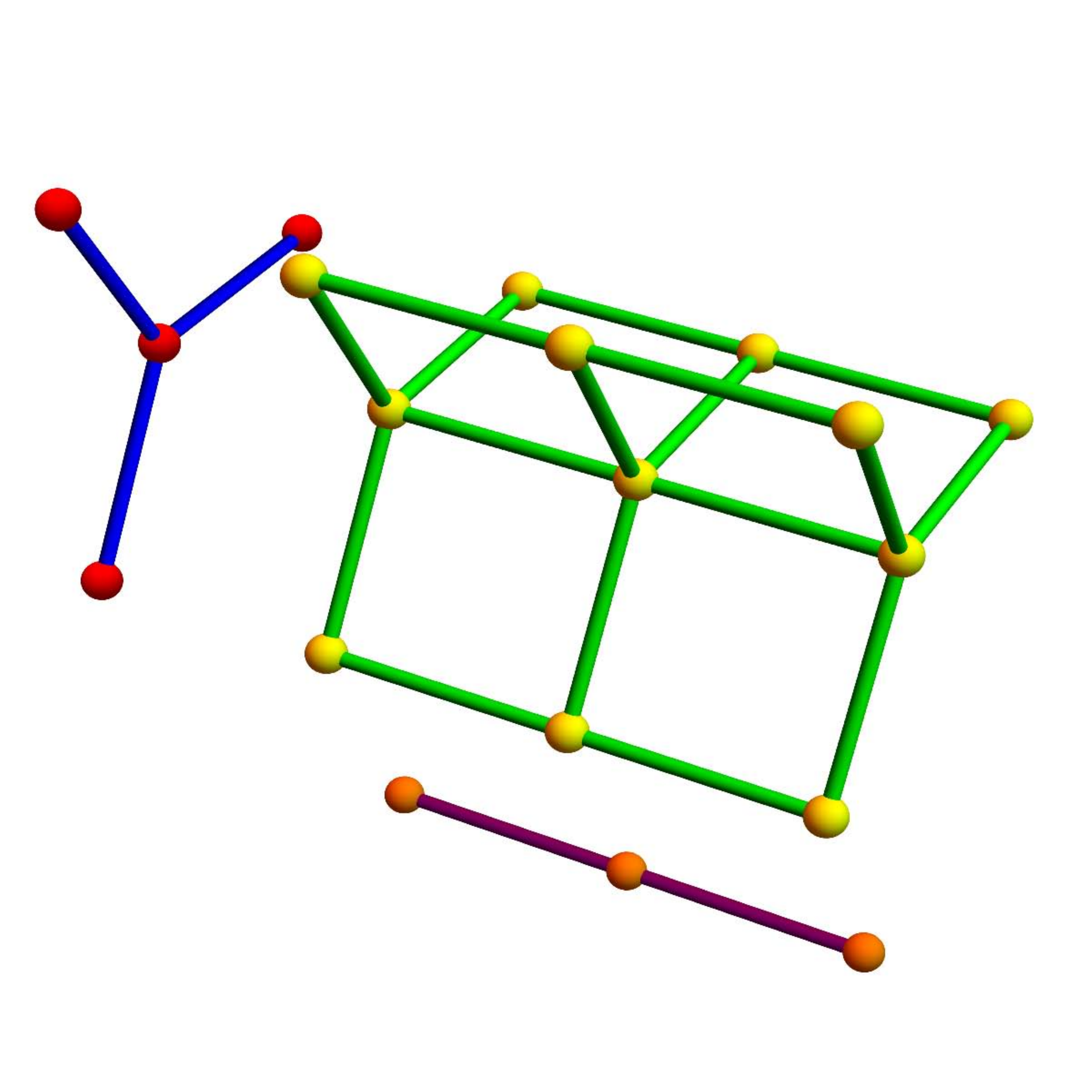}}
\scalebox{0.13}{\includegraphics{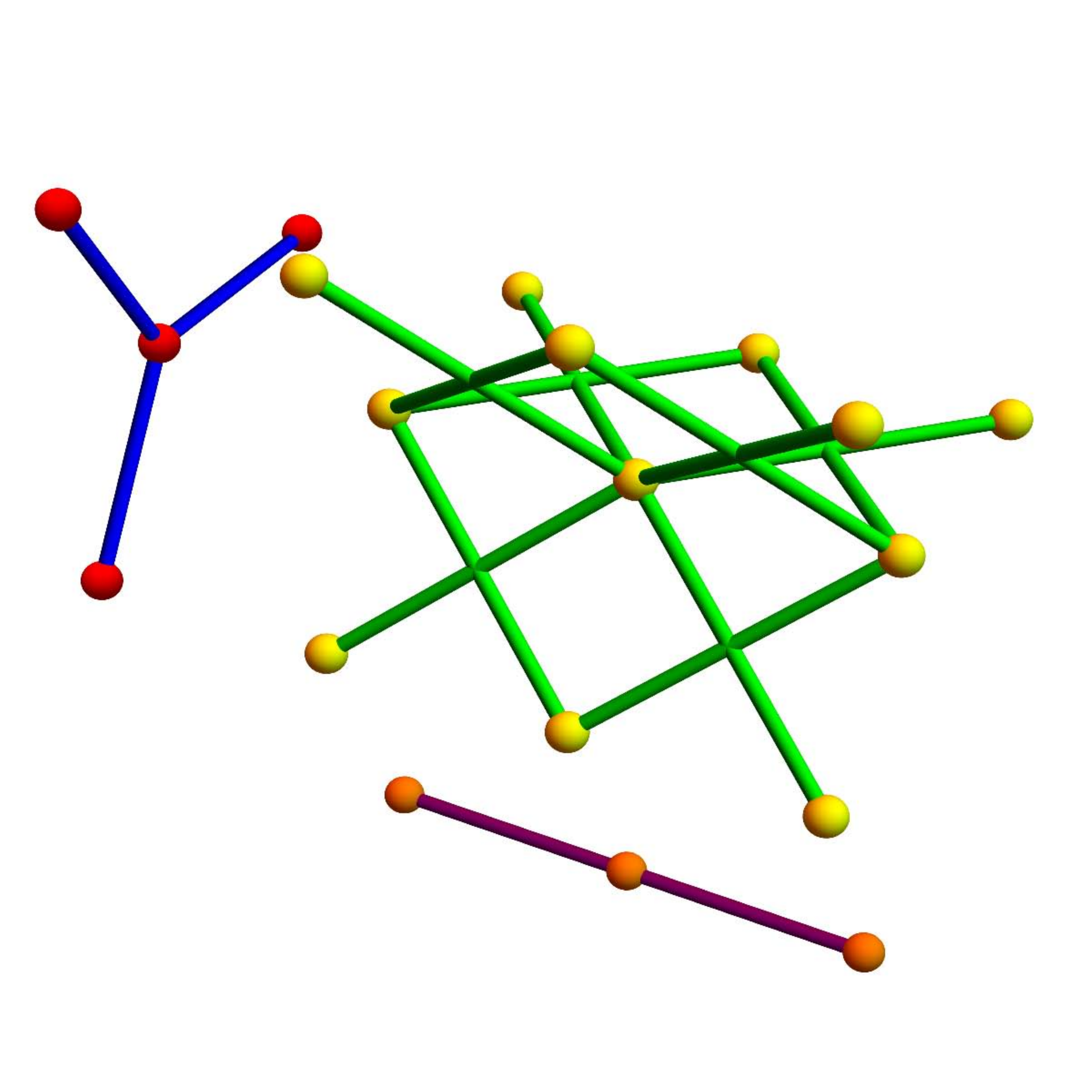}}
\scalebox{0.13}{\includegraphics{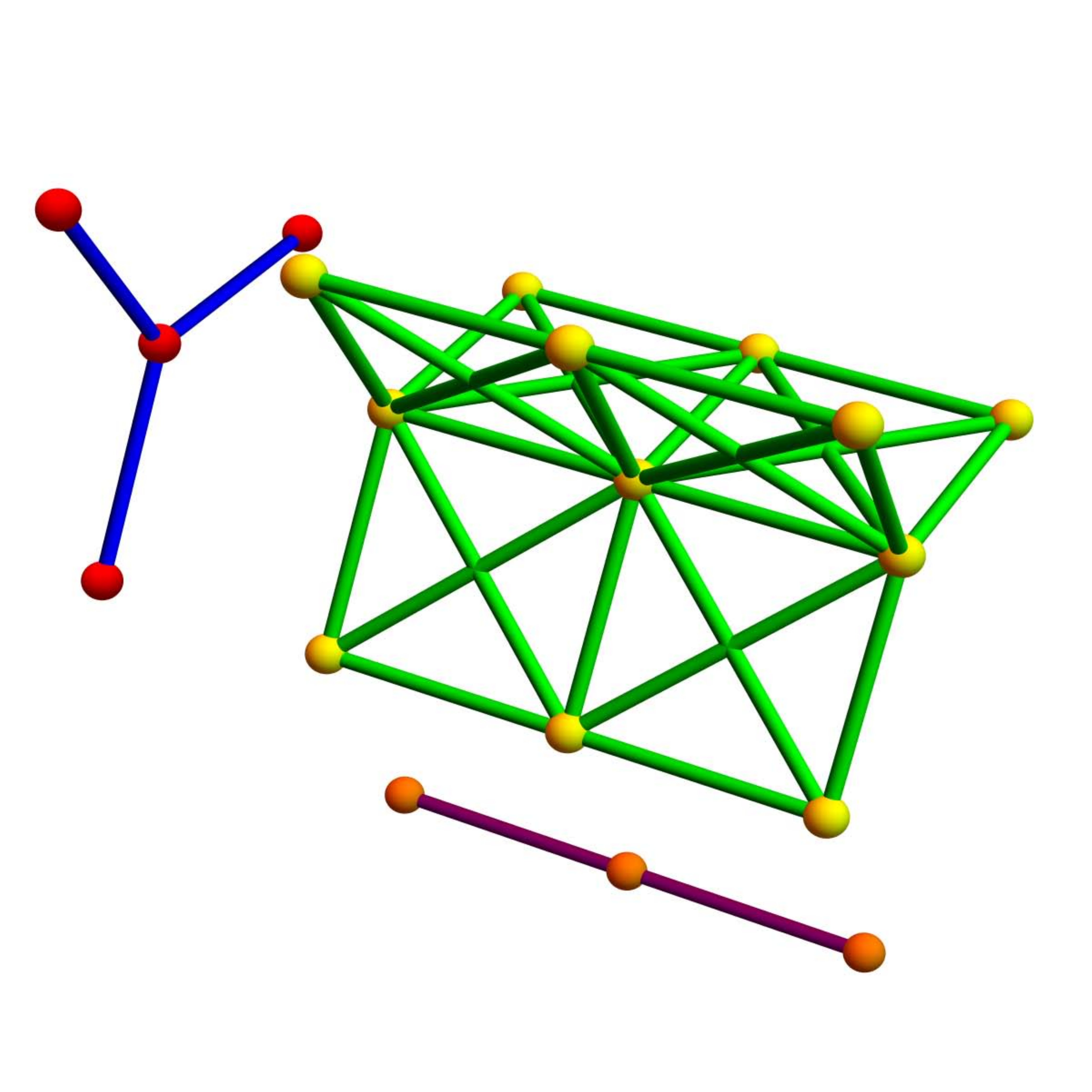}}
\caption{
The multiplication of $K_2$ with $S_3$ is shown in the graph product, the
tensor product and then the strong product. 
}
\end{figure}

\begin{figure}[!htpb]
\scalebox{0.14}{\includegraphics{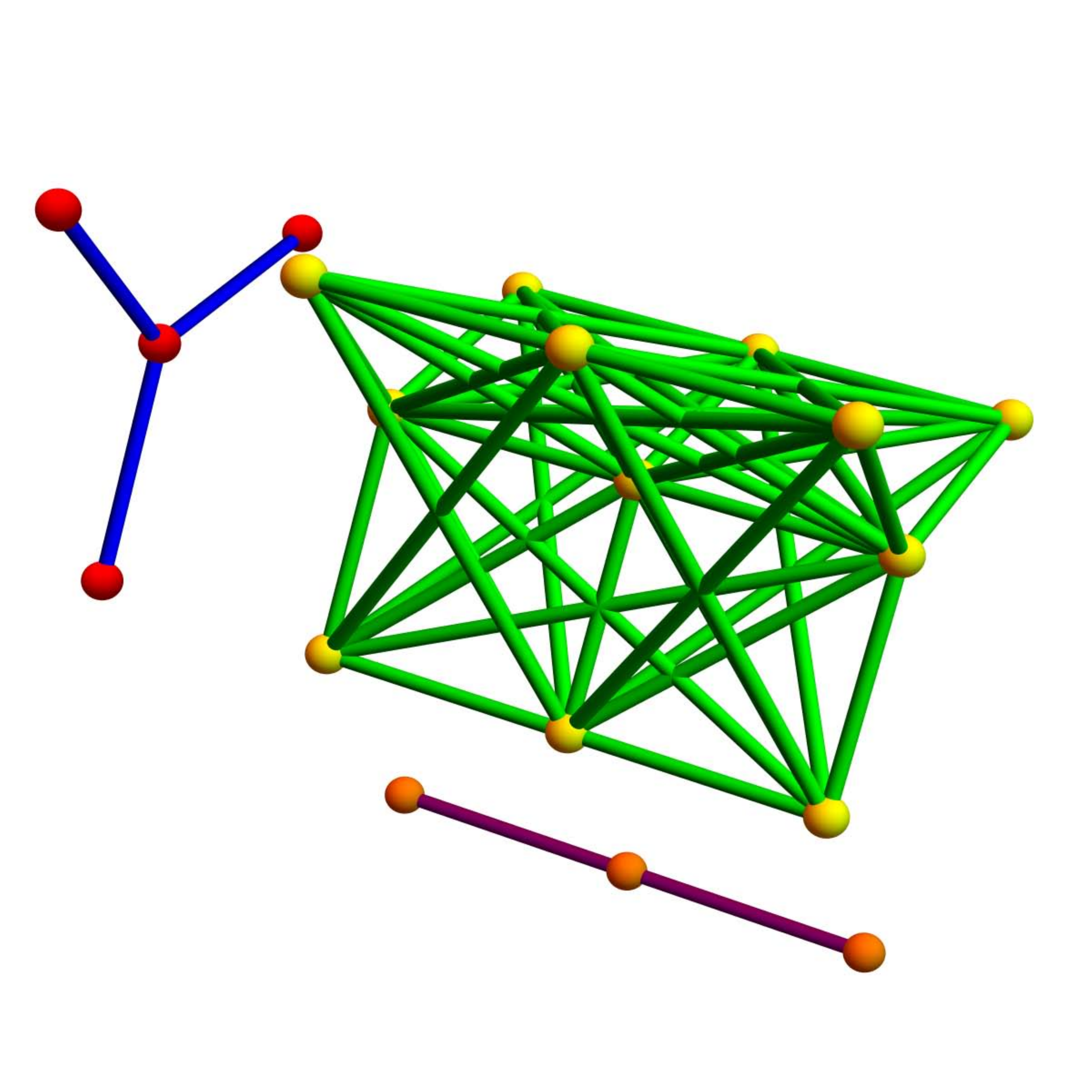}}
\scalebox{0.14}{\includegraphics{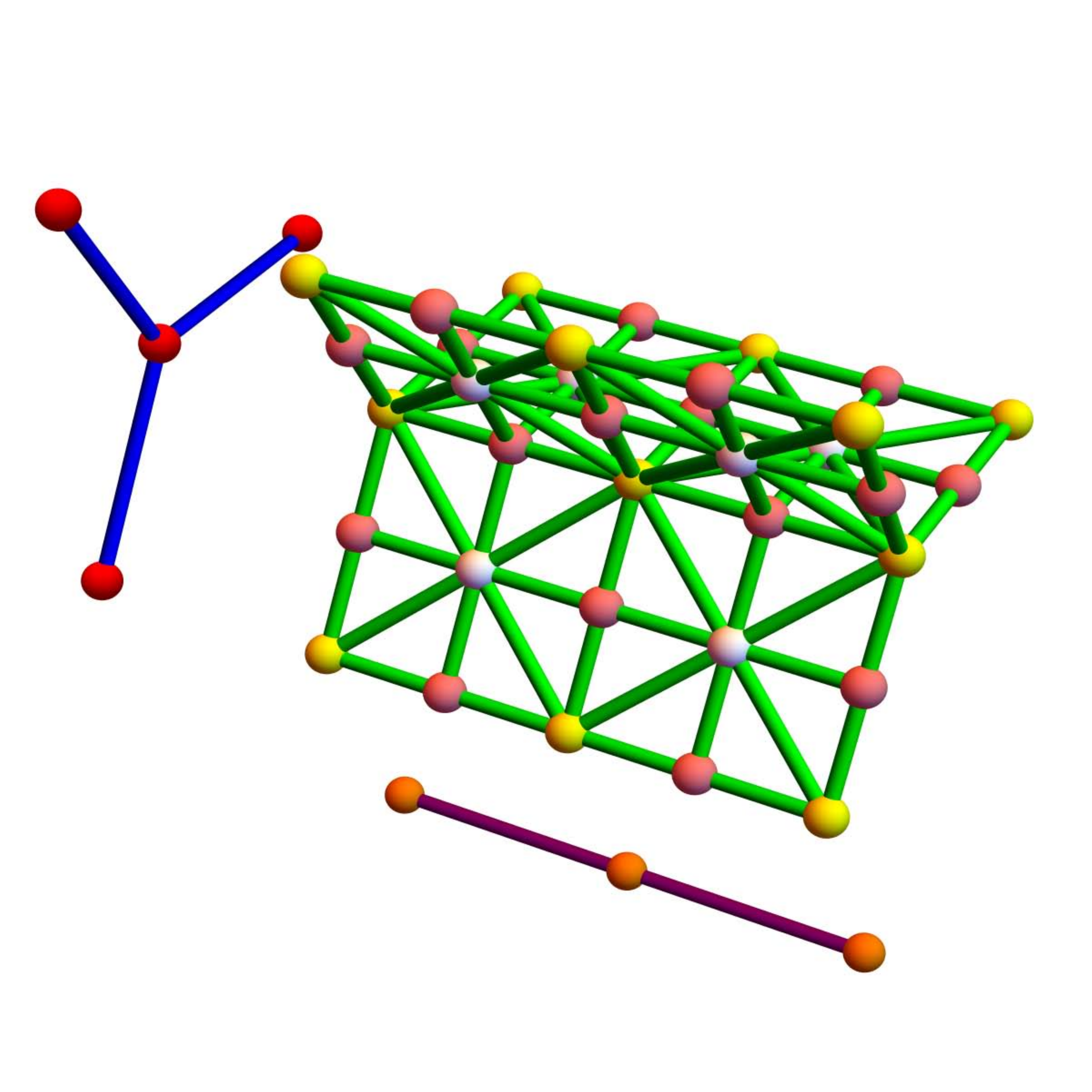}}
\caption{
The multiplication of $K_2$ with $S_3$ is shown in the Zykov and Cartesian simplex product. 
}
\end{figure}

\begin{figure}[!htpb]
\scalebox{0.1}{\includegraphics{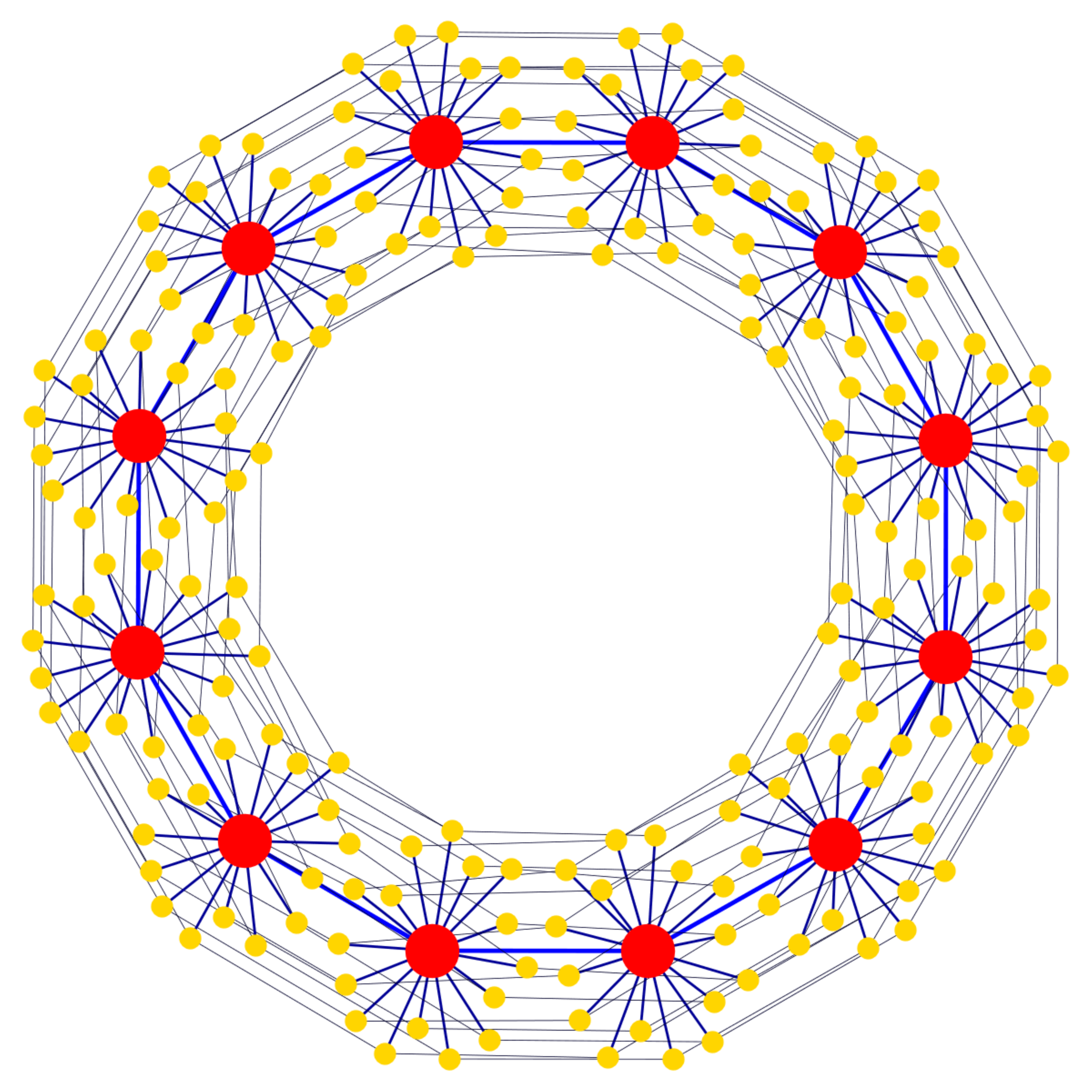}}
\scalebox{0.1}{\includegraphics{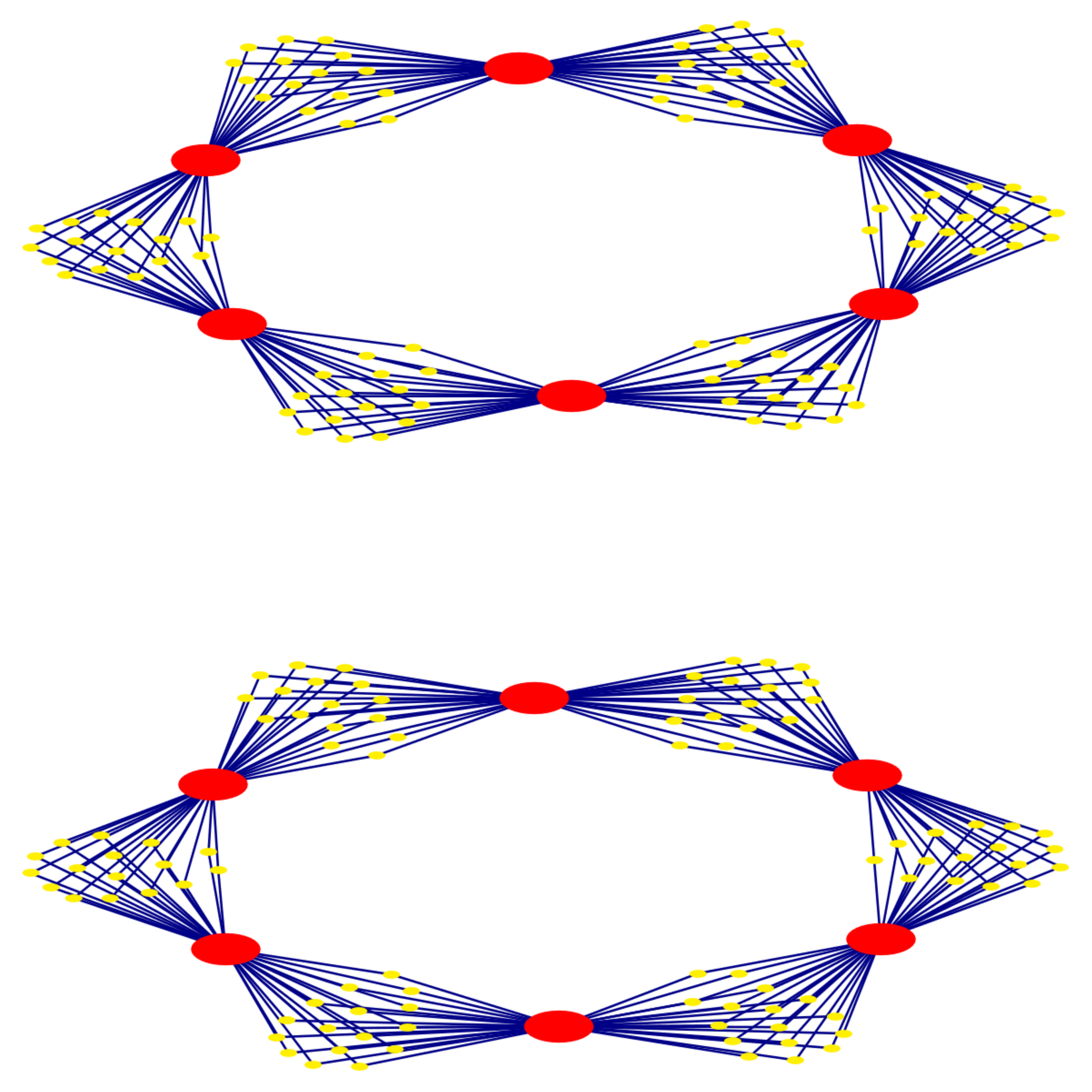}}
\scalebox{0.12}{\includegraphics{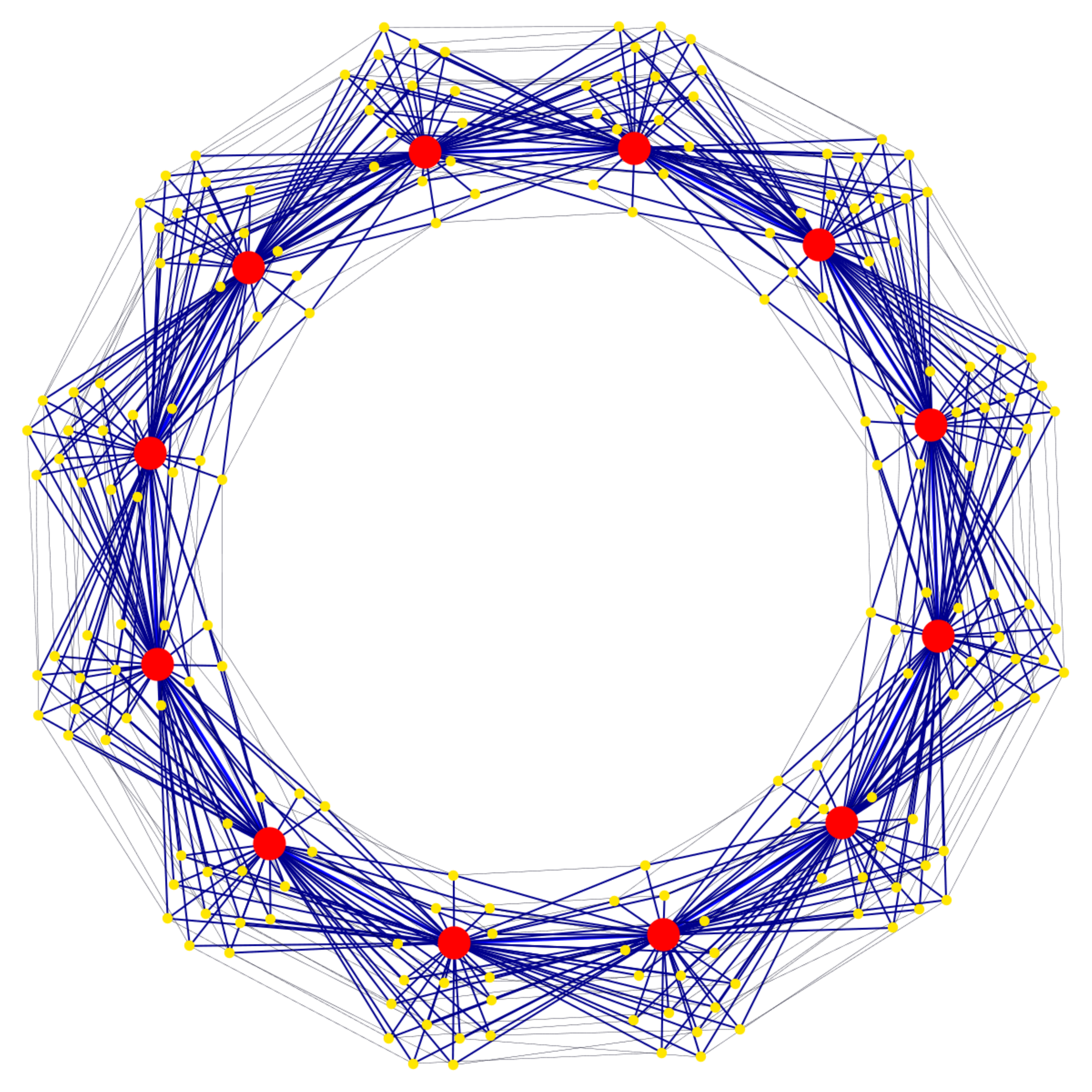}}
\caption{
The weak, direct and strong products of the circular graph $C_{12}$ and 
the star graph $S_{17}$. The direct product (tensor product) is not connected.
It is merged with the first weak product to become the strong product. 
For more sophisticated visualizations, see \cite{VisualizationGraphProducts}.
}
\end{figure}

\section{Properties}

\paragraph{}
In the case of the disjoint union $\oplus$, the graphs $P_n =\{ (1,2, \dots, n), \emptyset \}$ form a
sub-monoid which extends to a subgroup which is isomorphic to the integers $\mathbb{Z}$. 
This is the stone age {\bf pebble addition} which was used before
cuneiforms appeared. The fact that the zero element $0 = \{ \emptyset , \emptyset \}$ and 
the negative numbers $-G$ have been introduced much later in mathematics is paralleled in graph theory:
while the empty graph is used quite frequently as a zero element, negative graphs are rarely discussed. 
{\bf Valuations}, maps $X$ from $\G$ to the real numbers satisfying $X(G \cup H) + X(G \cap H) = X(G) + X(H)$ 
obviously have to be extended in a compatible way. By the discrete Hadwiger theorem \cite{KlainRota}
it is enough 
to look $v_k$ which form a basis. One defines $v_k(G-H) = v_k(H)-v_k(H)$ 
where $v_k(G)$ is the number of $k$-dimensional simplices in $G$. The Euler characteristic is then
still $\sum_{k=0}^{\infty} (-1)^k v_k(G)$. This extension can only be done in the case of the
addition $\oplus$. In the Zykov addition $+$, the generating functions 
$f(x) = 1 + \sum_{k=1}^{\infty} v_{k-1} x^k$ has the property that $f_{G+H} = f_G f_H$ 
(a product and not a sum) which means we would formally would have to require $f_{G-H} = f_G/f_H$
to extend the generating function to a group homomorphism from the Grothendieck group to the
rational functions. Still we can still represent group elements in the additive Zykov groups either
as ordered pairs $(G,H)$ or $G-H$ using some equivalence relation $A-B=C-D$ if $A+D=B+D$. 

\paragraph{}
In the case of the Zykov join operation $+$, the
set of {\bf complete graphs} $K_n$ plays the role of the integers, where the negative numbers are just
written as $-K_n$.  We will see in a moment that one can also write any element in the additive
Zykov group uniquely as $G-H$, where $G,H$ are graphs. From the fact that $K_n + K_m = K_{n+m}$ and
$K_n \cdot K_m = K_{nm}$ (especially postulating $K_{-1} \cdot K_{-1}=K_1$), we see immediately that 
the {\bf clique number} $n=d+1$ giving the largest $n$ for which $K_n$ is a subgraph of $G$ is a ring 
homomorphism: 

\begin{propo}[Clique number as ring homomorphism of Zykov ring]
For the Zykov ring, the clique number $c(G)={\rm dim}(G)+1$ is a ring homomorphisms.  
\end{propo}
\begin{proof}
We have $c(0)=c(\emptyset) = 0$ as the dimension of the empty graph is $-1$. Also $c(1)=c(K_1)=1$.
It follows from $K_n + K_m = K_{n+m}$ that $c(G + H) = c(G) + c(H)$ 
and from $K_n \cdot K_m$ that $c(G \times H) = c(G) c(H)$. To extend this to 
the entire ring, we have to postulate $c(-G)=-c(G)$ but the definition of the group
from the monoid assures that this extends to the additive group and so also to the ring. 
\end{proof} 

\paragraph{}
The {\bf Euler characteristic} of a graph $G$ is defined as $\sum_{k=0}^{\infty} (-1)^k v_k$, where $v_k$
is the number of $k$-dimensional complete subgraphs $K_{k+1}$ in $G$. It can also be written as 
the sum $\sum_{x} (-1)^{{\rm dim}(x)}$ over all simplices $x$ (complete subgraphs) in $G$. The Euler-Poincar\'e
identity tells that $\chi(G) = \sum_{k=0}^{\infty} (-1)^k b_k$, where $b_k={\rm dim}H^k(G)$ are the Betti 
numbers. They can be easily computed as the nullity ${\rm dim}({\rm ker}(L_k)$, where $L_k$ is the $k$'th block
in the Hodge Laplacian $H=(d+d^*)^2$. 
It follows from the K\"unneth formula $H^k(G \times H) = \oplus_{i+j=k} H^i(G) \otimes H^j(G)$ that
the Poincar\'e polynomial $p_G(x)=\sum_{k=0} b_k x^k$ satisfies $p_{G \times H}(x) = p_G(x) p_H(y)$ so that
$\chi(G)=p_G(-1)$ satisfies  $\chi(G \times H) = \chi(G) \chi(H)$. One can also give a direct inductive proof
of the product property of $\chi$ without invoking cohomology using Poincar\'e-Hopf \cite{KnillKuenneth}.  \\

\paragraph{}
The homotopy theory of graphs and finite abstract simplicial
complexes is parallel to the homotopy of geometric realizations but is entirely combinatorial. The adaptation of the Whitehead
definition to the discrete have been done in the 70ies, notably by Evako and Fiske. First define
inductively what a collapsible graph is: a graph $G$ is collapsible if there exists a vertex $x$ for which the unit sphere
$S(x)$  and $G \setminus x$ is collapsible. A homotopy step is the process of removing a vertex with contractible unit sphere
or then making a cone extension over a collapsible subgraph. A graph is contractible if it is homotopic to $K_1$. 
The homotopy of abstract simplicial complexes can then be defined through the homotopy of its Barycentric refinement,
which is the Whitney complex of a graph and therefore part of the graph theoretical contraction definition.
The discrete description has the advantage that it can be implemented easier on a computer.

\begin{lemma}[Homotopy lemma]
For any finite $H,G \in \G$, the strong product graph $H \osquare G$ is homotopic to the Cartesian
simplex product $H \times G$.
\end{lemma}
\begin{proof}
The graph $H \osquare G$ is homotopic to its Barycentric refinement $(H \osquare G)_1$.
Now deform each maximal simplex $x \osquare y$ to a simplex $x \times y$ starting with one
dimensional simplices, then turning to triangles etc. These Whitehead deformation moves
can best be seen in an Euclidean embedding but they can be done entirely combinatorially:
first add a new vertex $m$ in the center of $x \osquare y$ connecting with all vertices of
the week product $x \square y$ and all interior vertices as well as all vertices connecdted
to those interior vertices. Now remove all interior vertices together with their connections.
We end up with $x \times y$. After doing this for all $x \times y$ of dimension $d$, continue
with dimension $d+1$ etc until everything is deformed.
\end{proof}

\begin{propo}[Euler characteristic as ring homomorphism from the strong product]
Euler characteristic is a ring homomorphisms from the 
Zykov ring $(\G,\oplus,\otimes)$ to $\mathbb{Z}$. 
\end{propo}
\begin{proof}
The Euler characteristic is a homomorphism for the Cartesian simplex ring to 
the integers \cite{KnillKuenneth}. The argument there was to factor
$\chi(G \times H) = \sum_{ (x,y) \subset G \times H} (-1)^{{\rm dim}(x) + {\rm dim}(y)}$
as $(\sum_{x \in G}  (-1)^{dim(x)}) (\sum_{y \in H}  (-1)^{dim(y)})$ so
so that $\chi(G \times H) = \chi(G_1) \chi(H_1)$,
where $G_1$ and $H_1$ are the Barycentric refinements of $G$ and $H$. 
But the formula $\sum_x (-1)^{{\rm dim}(x)}$ defining
the Euler characteristic of $G$ is a Poincar\'e-Hopf formula for the 
Morse function $f(x) = {\rm dim}(x)$ on the vertex set 
of the Barycentric refinement $G_1$. 
\end{proof}

{\bf Remarks.} \\
{\bf 1)} Already small examples show that the other products, the weak and direct 
products have no compatibility whatsoever with Euler characteristic.
An example is $K_2 \osquare K_2$. \\
{\bf 2)} The {\bf Wu characteristic} \cite{Wu1953,valuation},
$$ \omega(G) = \sum_{x \sim y} \omega(x) \omega(y) $$ 
with $\omega(x) = (-1)^{{\rm dim}(x)}$, where the sum is over all 
intersecting simplices $x,y$ is not a homotopy invariant. 
We know that $\omega(G \times H) = \omega(G) \omega(H)$ for the Cartesian product,
but the multiplicativity fails in general for all other products. 

\begin{figure}[!htpb]
\scalebox{0.2}{\includegraphics{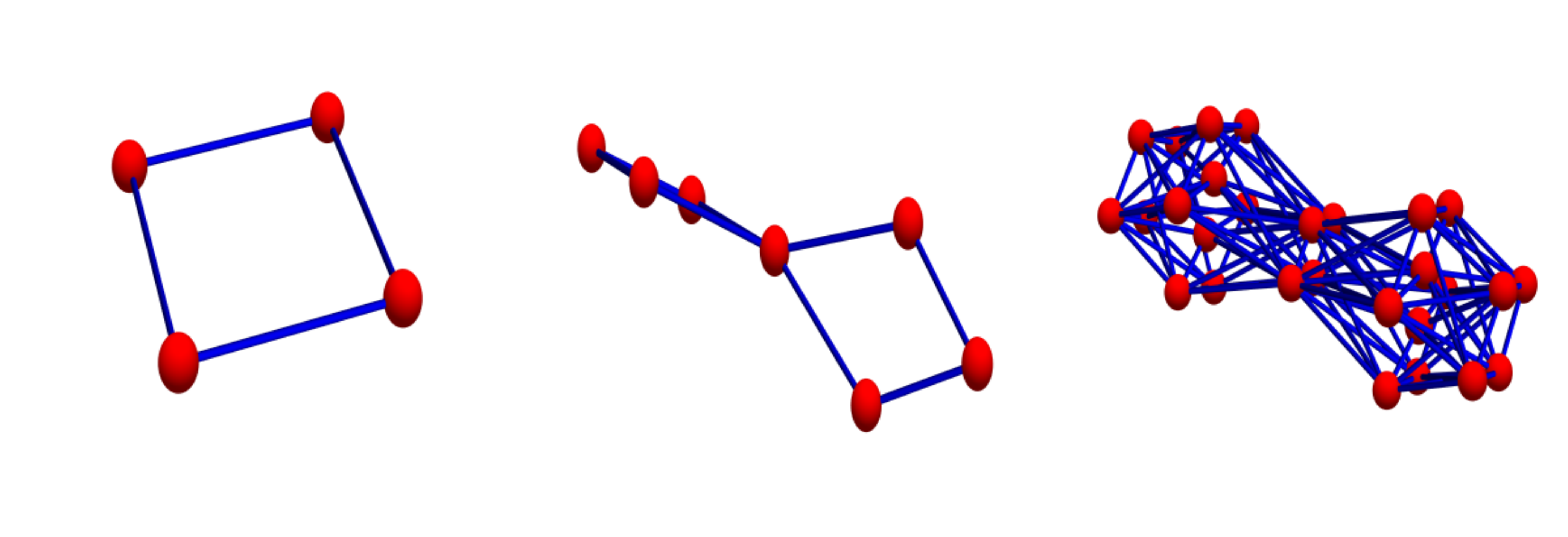}} 
\scalebox{0.2}{\includegraphics{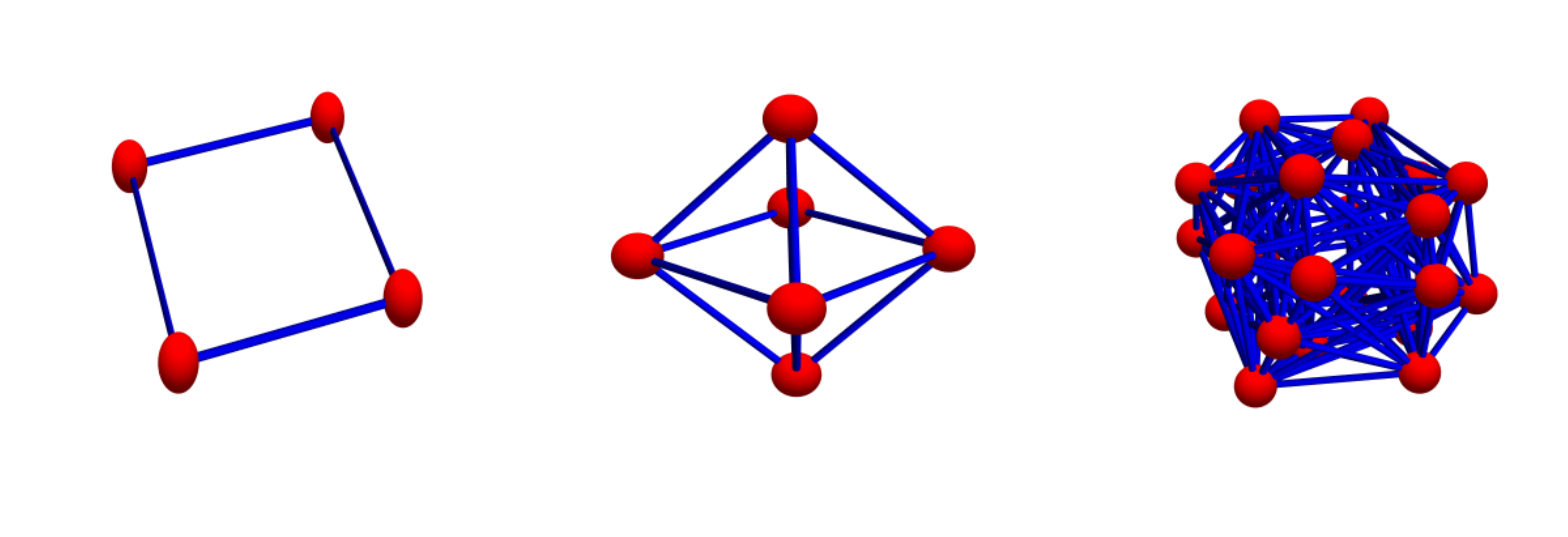}} 
\scalebox{0.2}{\includegraphics{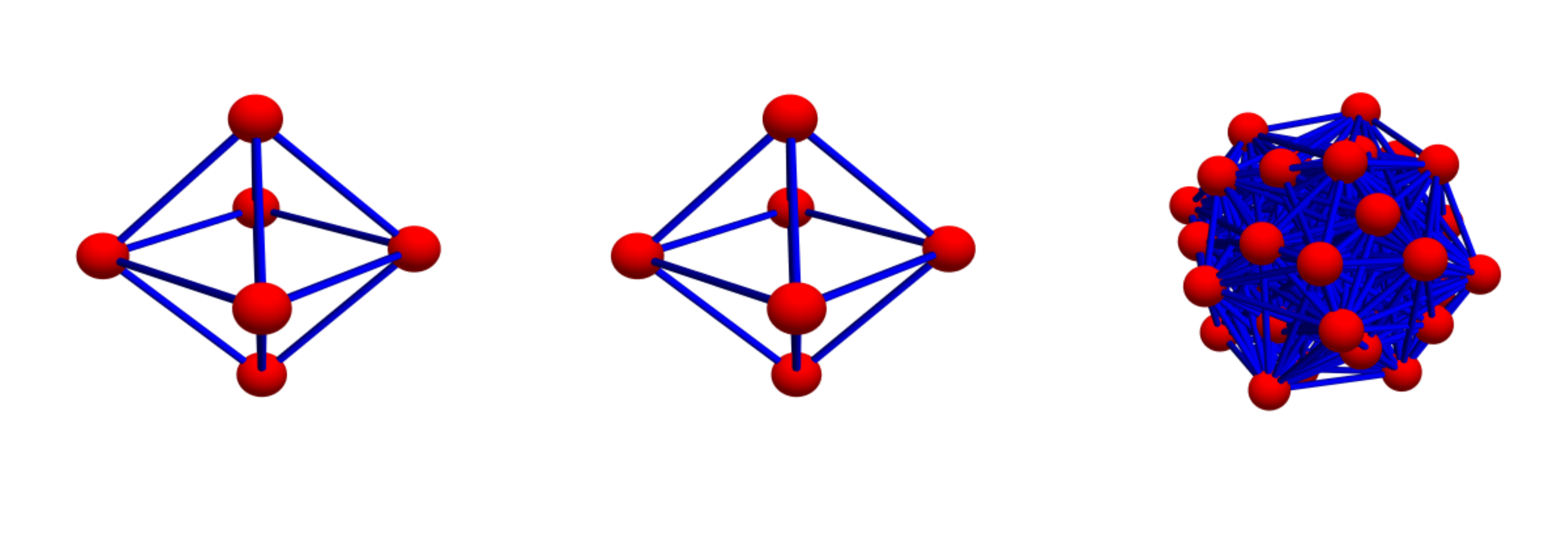}} 
\scalebox{0.2}{\includegraphics{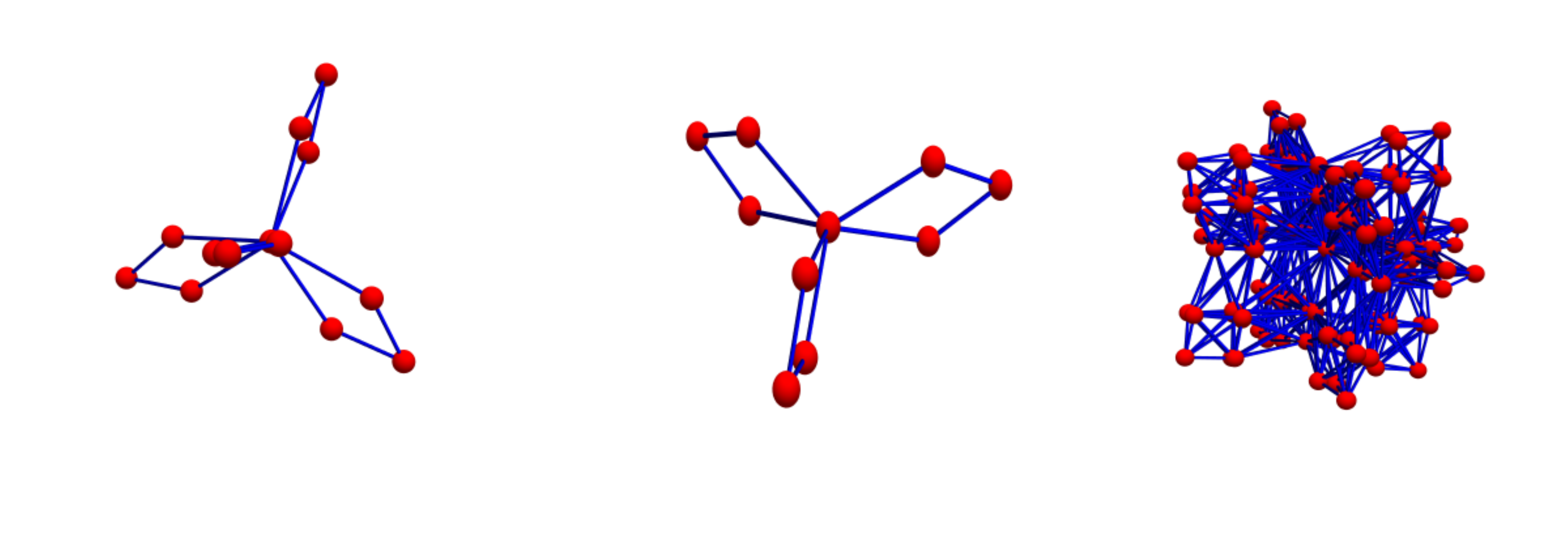}} 
\caption{
Let $B_k$ is a bouquet with $k$ flowers and O is the octahedron. 
We see first the strong product of $C_4,B_2$ where the Poincar\'e polynomial
identity is $(1+x)(1+2x) = 1+3x+2x^2$, then the strong product of 
$C_4,O$ with $(1+x)(1+x^2) = 1+x+x^2+x^3$, then the strong product of 
$O,O$ with $(1+x^2)(1+x^2) = 1+2x^2+x^4$ and  finally the strong product 
$B_2, B_3$ with $(1+3x)(1+4x) = 1+7x+12x^2$. In the 
last case, the f-vector is $(130, 700, 768, 192)$ and the 
Euler characteristic $130-700+768-192=6$ which matches $b_0-b_1+b_2=1-7+12=6$
as the general Euler-Poincar\'e formula shows. 
}
\end{figure}

\paragraph{}
Here is a summary about properties

\begin{tiny} \begin{tabular}{|l|l|l|l|l|} \hline
Operation       &     Clique number       &   Euler characteristic  & index           \\ \hline \hline
Union $\oplus$  &        -                &   additive              &                 \\ \hline
Join  +         &     additive            &                         & multiplicative  \\ \hline
\end{tabular} \end{tiny}

For the graph Cartesian product $\times$ or the tensor product $\otimes$
only the already mentioned additivity of Euler characteristic for disjoint union
or the trivial max-plus property of dimension for disjoint addition holds. 

\begin{tiny} \begin{tabular}{|l|l|l|l|l|l|} \hline
Operation       &    maximal dimension   &  Clique number       &   Euler characteristic  & index           \\ \hline \hline
$\cdot$         &      -                 &  multiplicative      &                         &                 \\ \hline 
$\times$        &    additive            &     -                &   multiplicative        &                 \\ \hline 
$\osquare$      &    additive            &     -                &   multiplicative        &                 \\ \hline 
$\otimes$       &      -                 &     -                &        -                &    -            \\ \hline 
$\square$       &      -                 &     -                &        -                &    -            \\ \hline 
\end{tabular} \end{tiny}

\begin{figure}[!htpb]
\scalebox{0.3}{\includegraphics{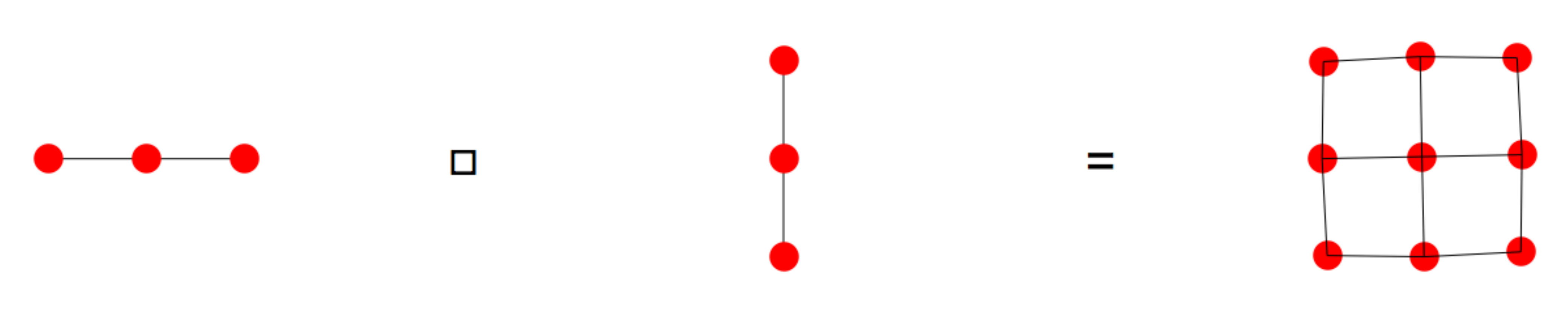}} 
\scalebox{0.3}{\includegraphics{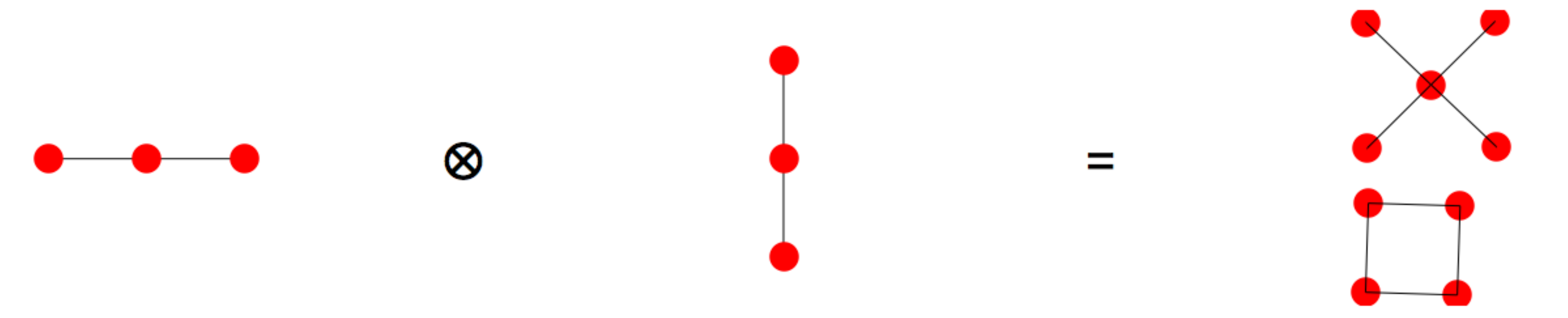}} 
\scalebox{0.3}{\includegraphics{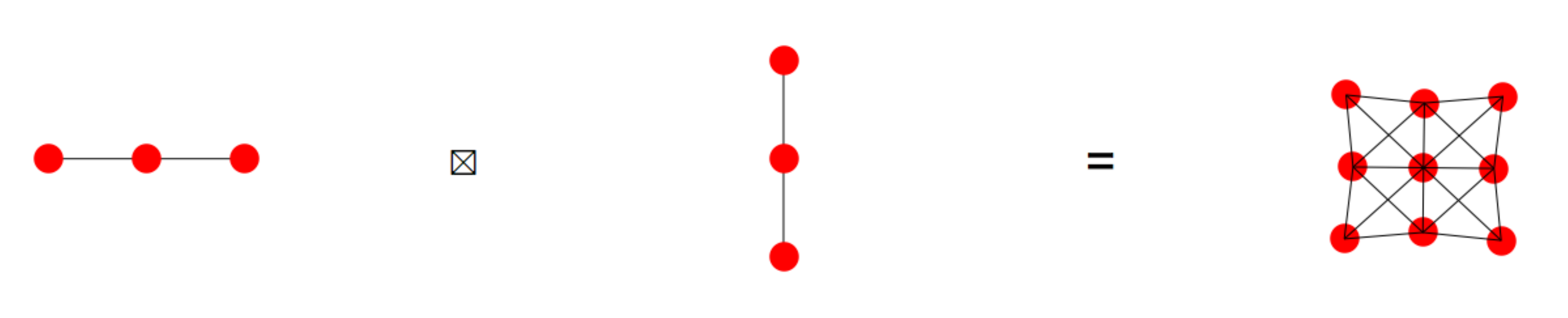}} 
\scalebox{0.3}{\includegraphics{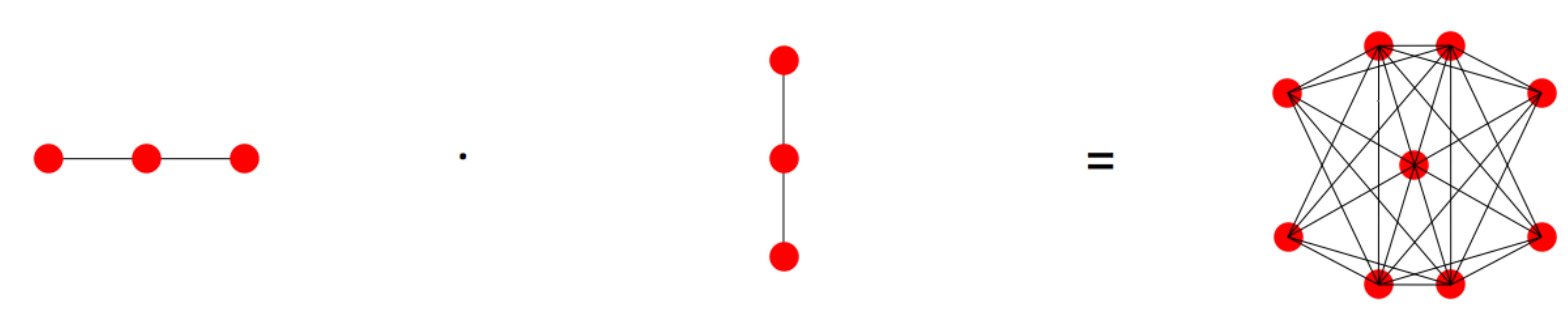}} 
\scalebox{0.3}{\includegraphics{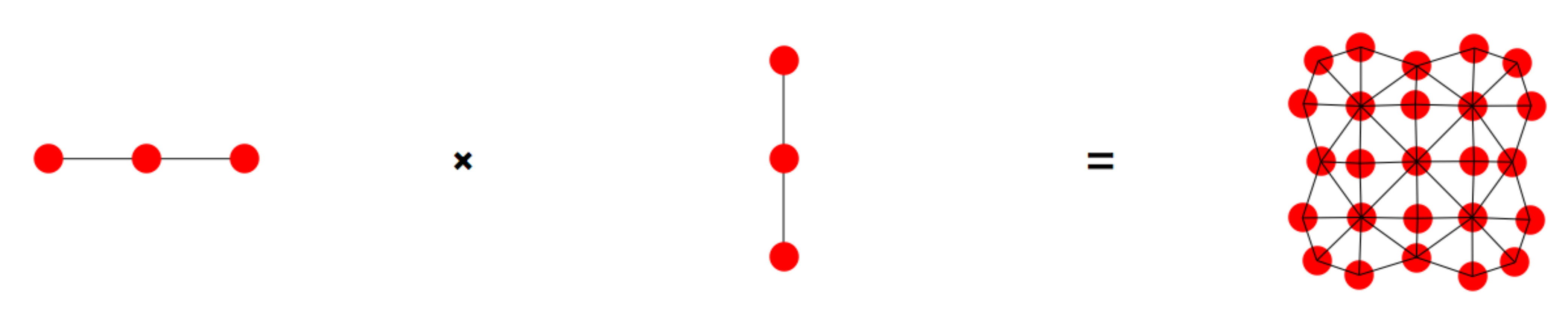}}
\caption{
The five multiplications of $L_3$ with $L_3$.
Only in the Cartesian ring case do we have a topological 
disk of dimension $2$ and Euler characteristic $1$. But as
$G \times H$ is a homotopic to Barycentric refinement of the
strong product, we have also Euler characteristic multiplicative. }
\end{figure}

\begin{figure}[!htpb]
\scalebox{0.3}{\includegraphics{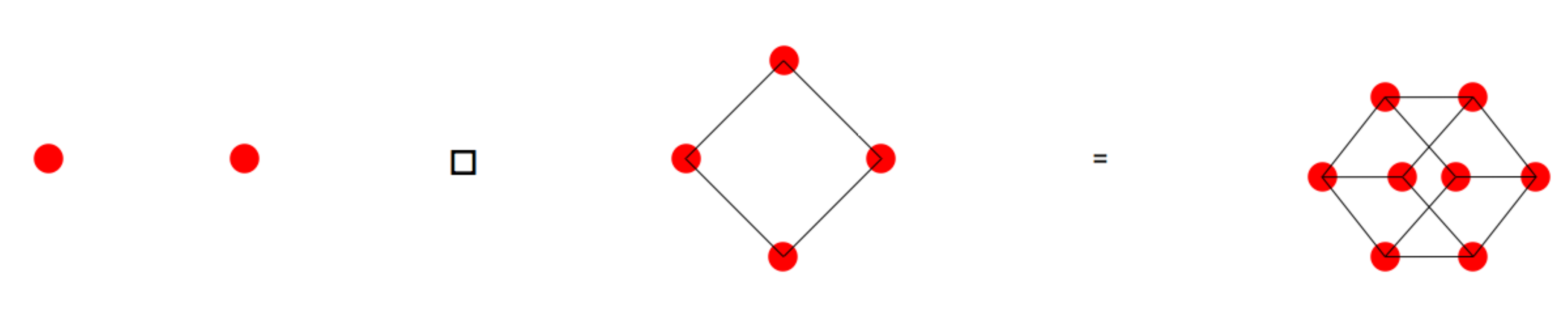}}
\scalebox{0.3}{\includegraphics{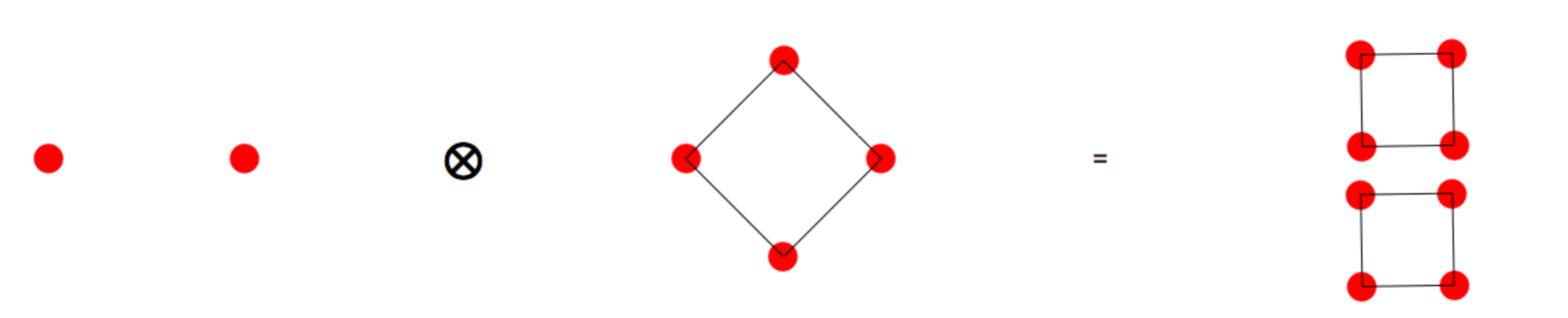}}
\scalebox{0.3}{\includegraphics{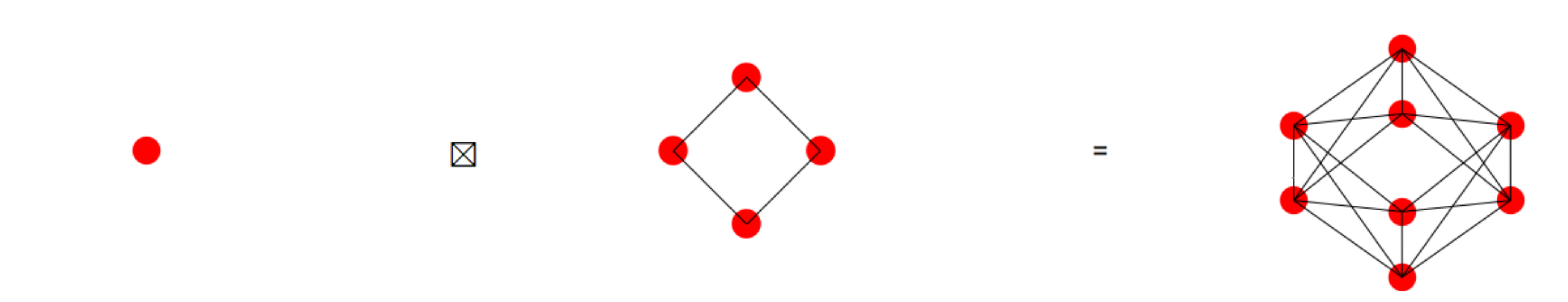}} 
\scalebox{0.3}{\includegraphics{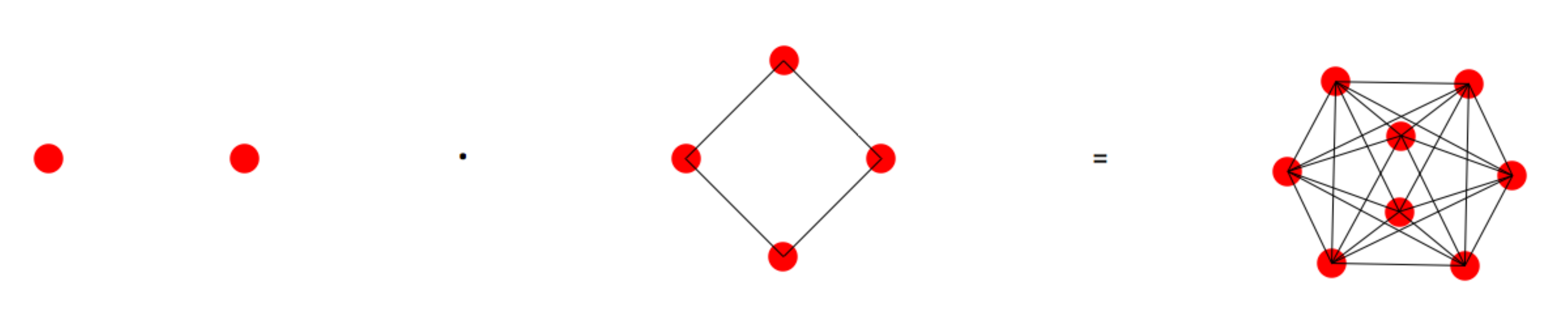}}
\scalebox{0.3}{\includegraphics{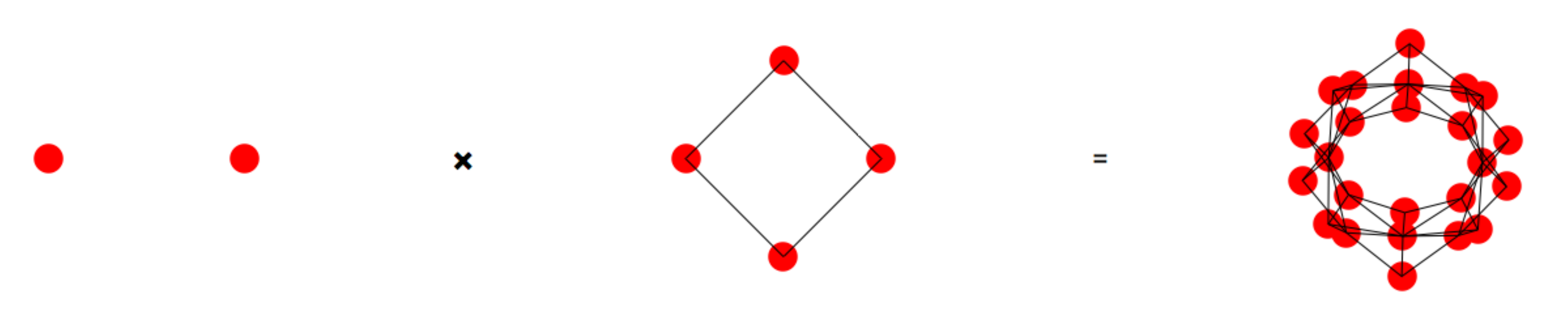}}
\caption{
The five multiplications of $K_2$ with $C_4$.
Only in the Cartesian product case, we see a cylinder of 
dimension $2$ and Euler characteristic $0$. The strong product is also
a cylinder but has larger dimension. 
}
\end{figure}

\begin{figure}[!htpb]
\scalebox{0.3}{\includegraphics{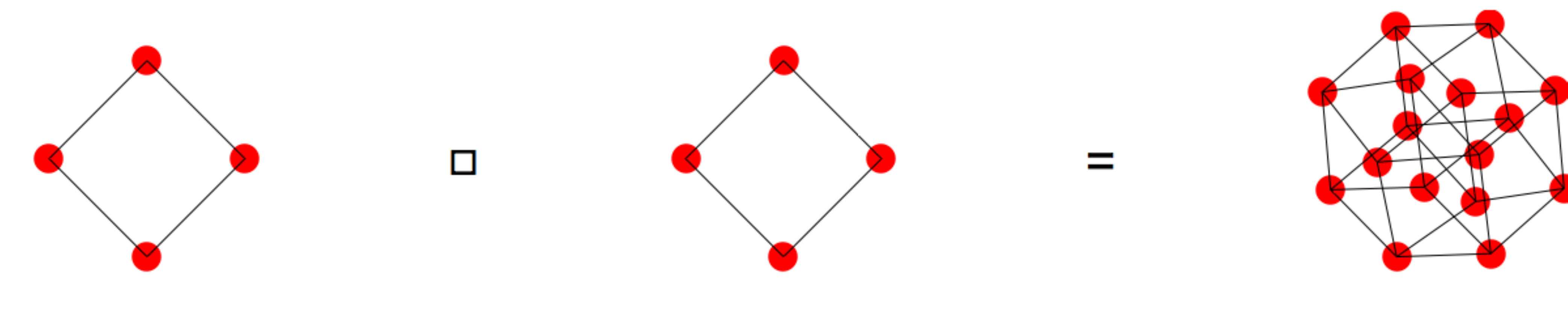}}
\scalebox{0.3}{\includegraphics{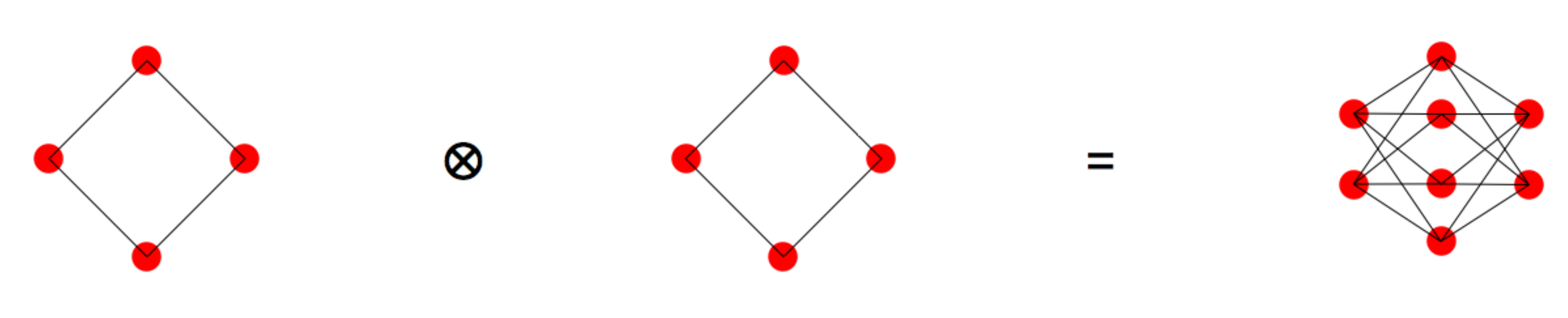}}
\scalebox{0.3}{\includegraphics{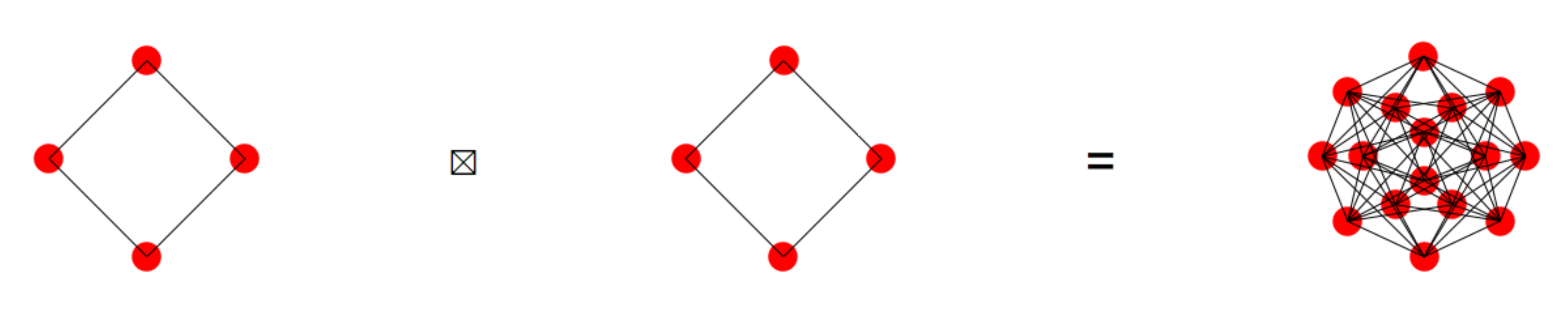}} 
\scalebox{0.3}{\includegraphics{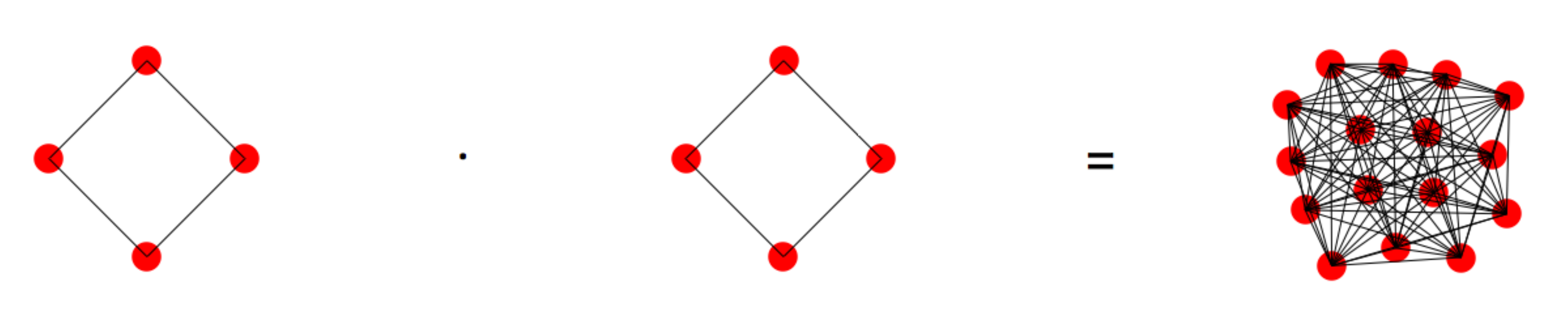}}
\scalebox{0.3}{\includegraphics{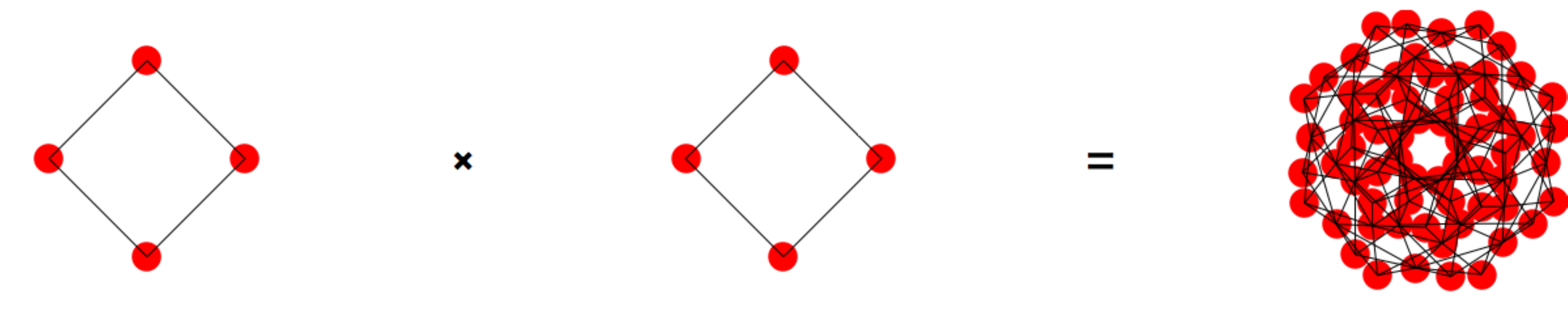}}
\caption{
The multiplication of $C_4$ with $C_4$ in the four rings.
Only in the Grothendieck case do we see a torus of dimension $2$
and Euler characteristic $0$. 
}
\end{figure}

\paragraph{}
{\bf Small examples:}  \\

{\bf 1)} $K_2 \oplus K_2 = P_2 \times P_2$  \\
{\bf 2)} $K_2 + K_2 = K_4$ \\
{\bf 3)} $K_2 \cdot K_2 = K_4$ \\
{\bf 4)} $K_2 \otimes K_2 = C_4$ \\
{\bf 5)} $K_2 \times K_2  = W_6$ \\
{\bf 6)} $K_2 \square K_2   = C_4$ \\
{\bf 7)} $K_2 \osquare K_2  = K_4$

{\bf 1)} $K_2 \oplus K_3$   \\
{\bf 2)} $K_2 + K_3 = K_5$ \\
{\bf 3)} $K_2 \cdot K_3 = C_6$ \\
{\bf 4)} $K_2 \otimes K_3$ \\
{\bf 5)} $K_2 \times K_3$ \\
{\bf 6)} $K_2 \square K_3$ \\
{\bf 7)} $K_2 \osquare K_3 = K_6$

\begin{figure}[!htpb]
\scalebox{0.1}{\includegraphics{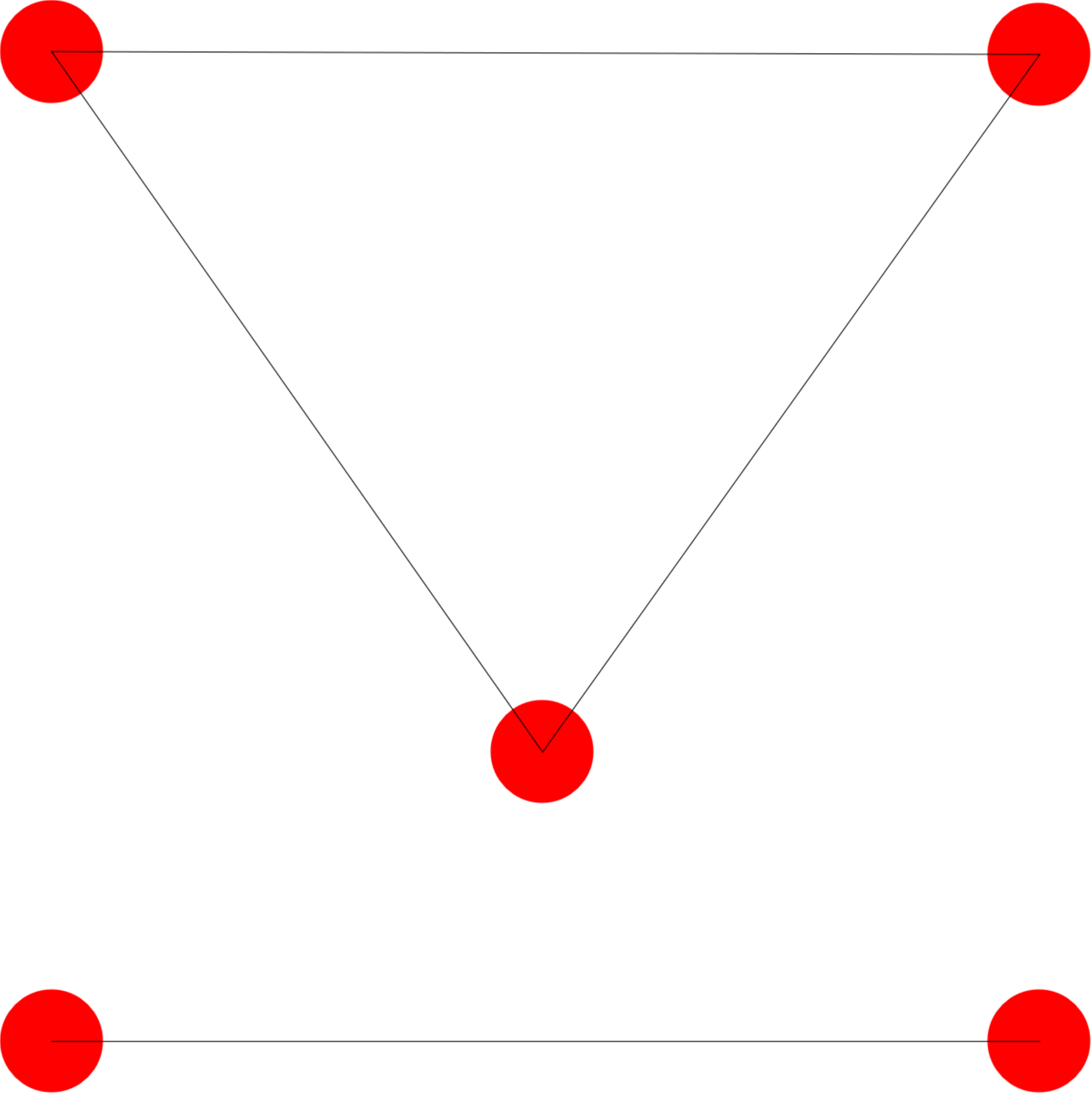}}
\scalebox{0.1}{\includegraphics{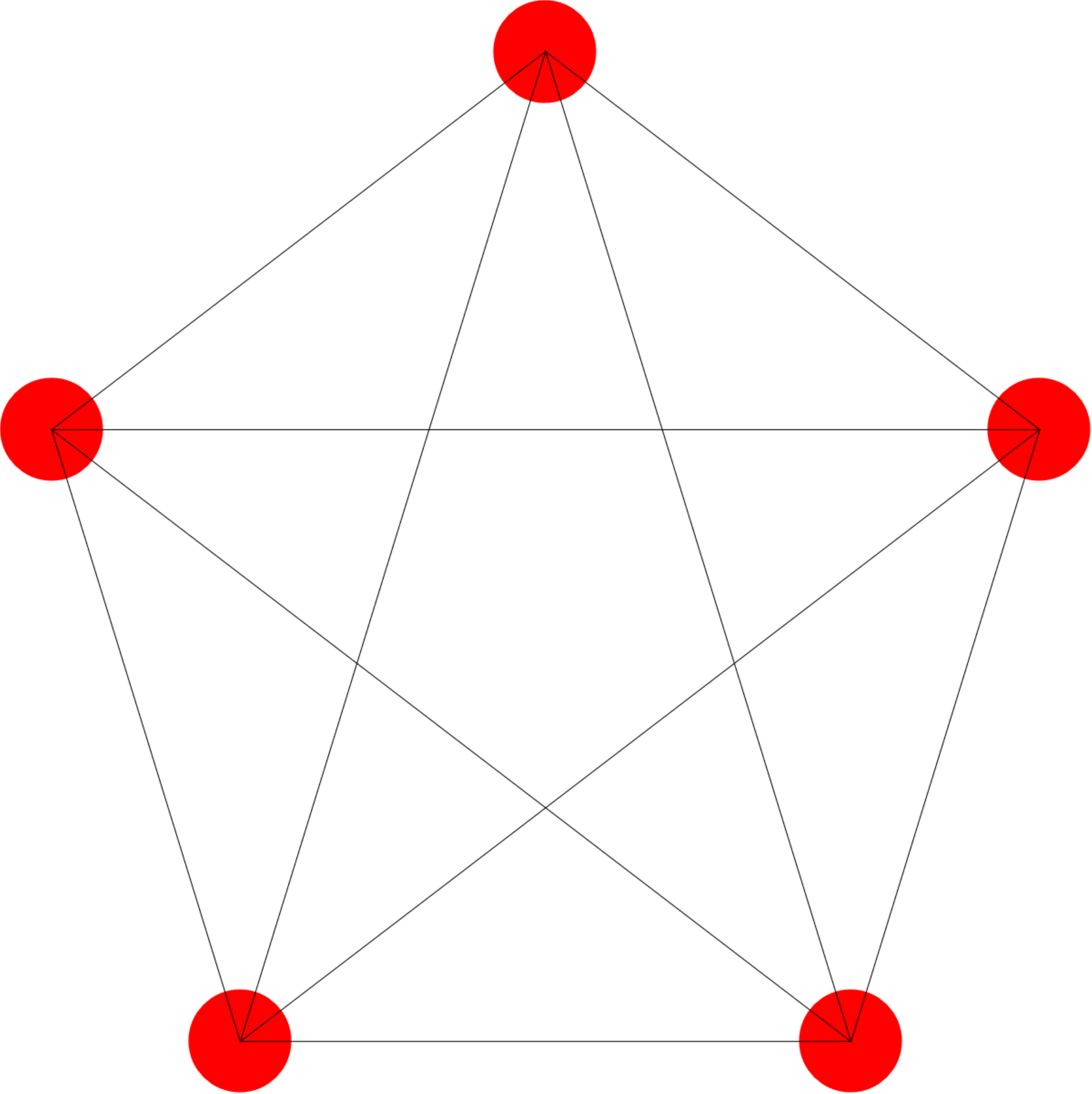}}
\caption{
The two graph additions combining $K_2$ with $K_3$.
}
\end{figure}

\begin{figure}[!htpb]
\scalebox{0.1}{\includegraphics{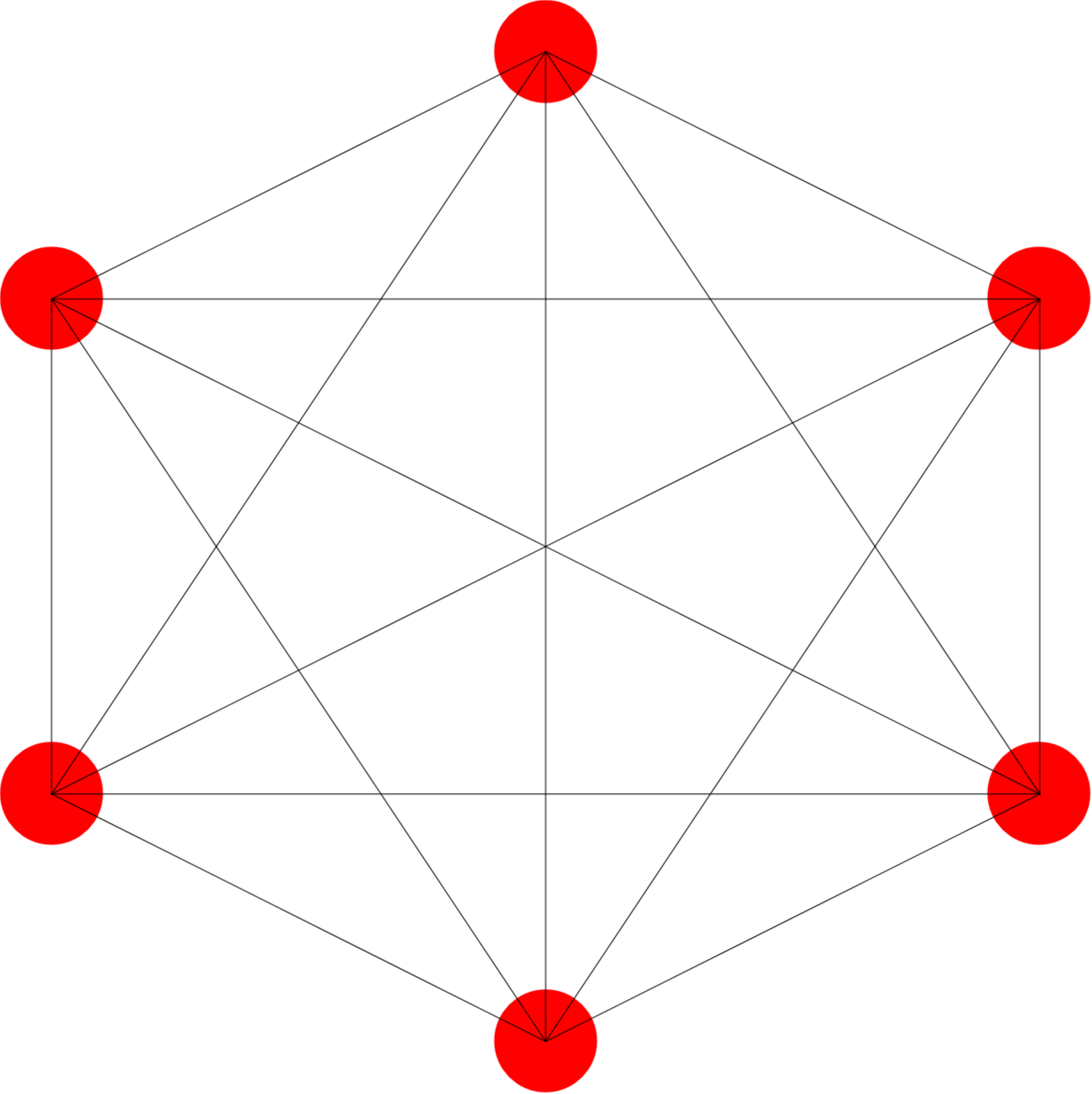}}
\scalebox{0.1}{\includegraphics{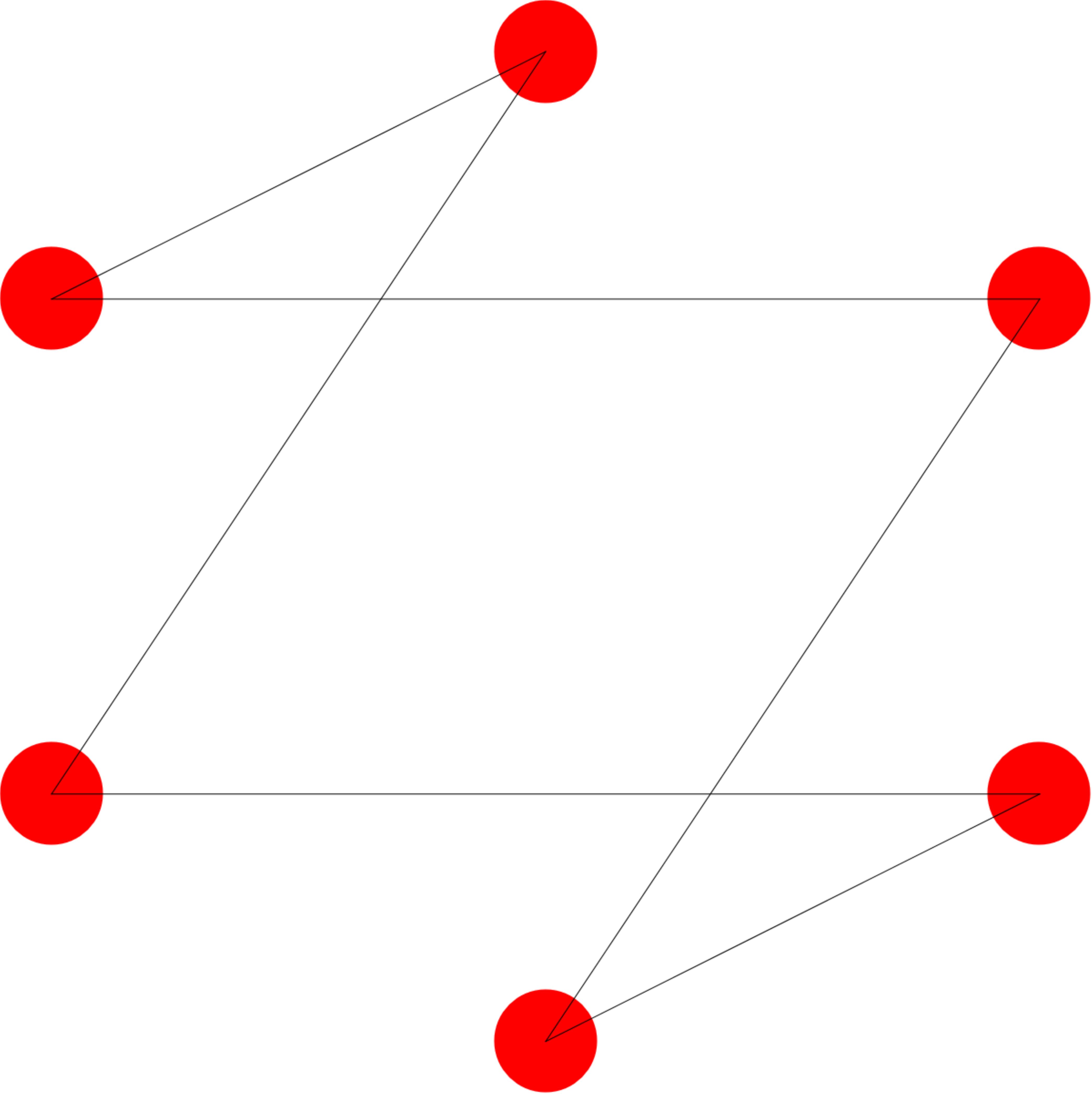}}
\scalebox{0.1}{\includegraphics{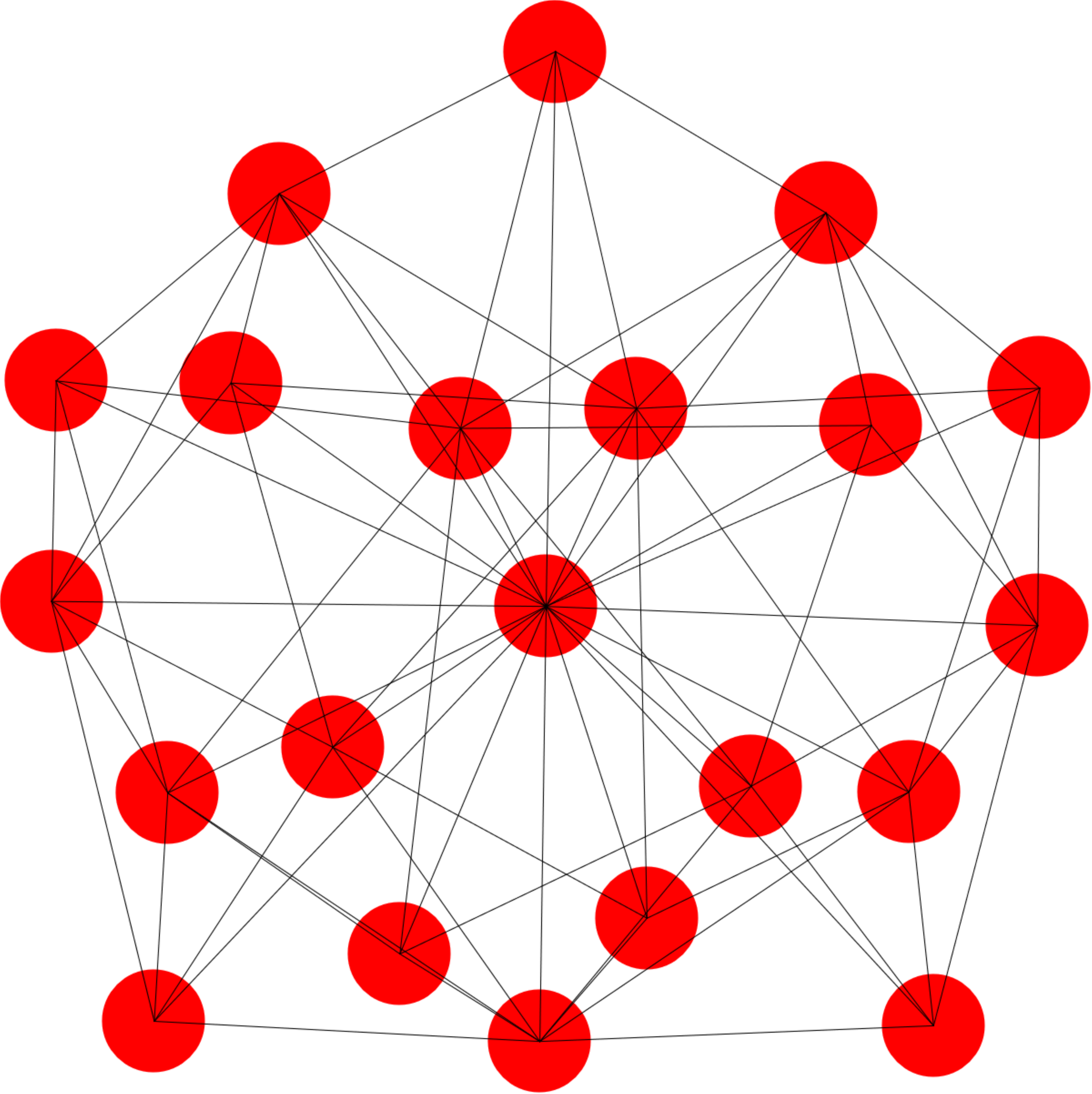}}
\scalebox{0.1}{\includegraphics{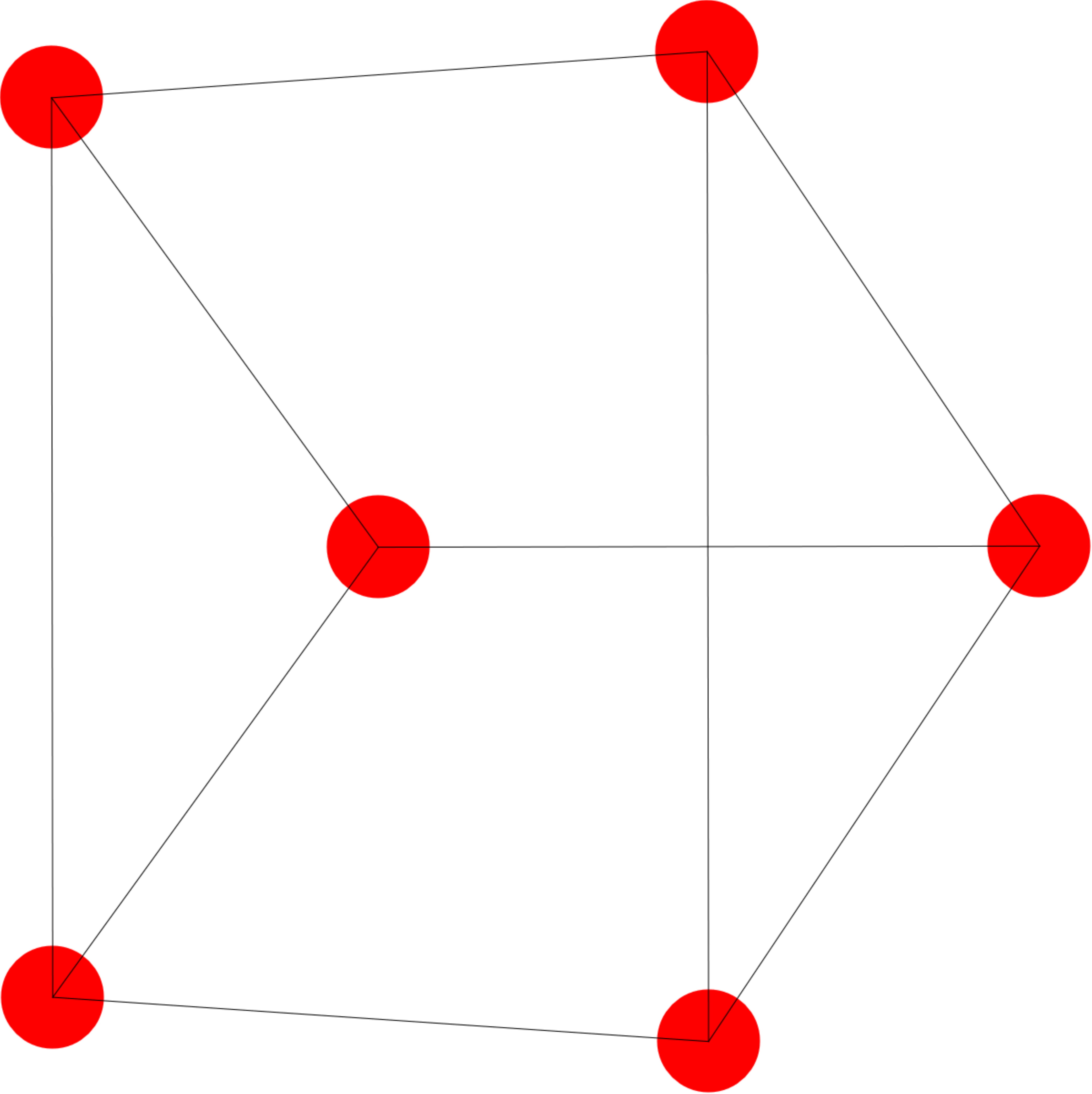}}
\scalebox{0.1}{\includegraphics{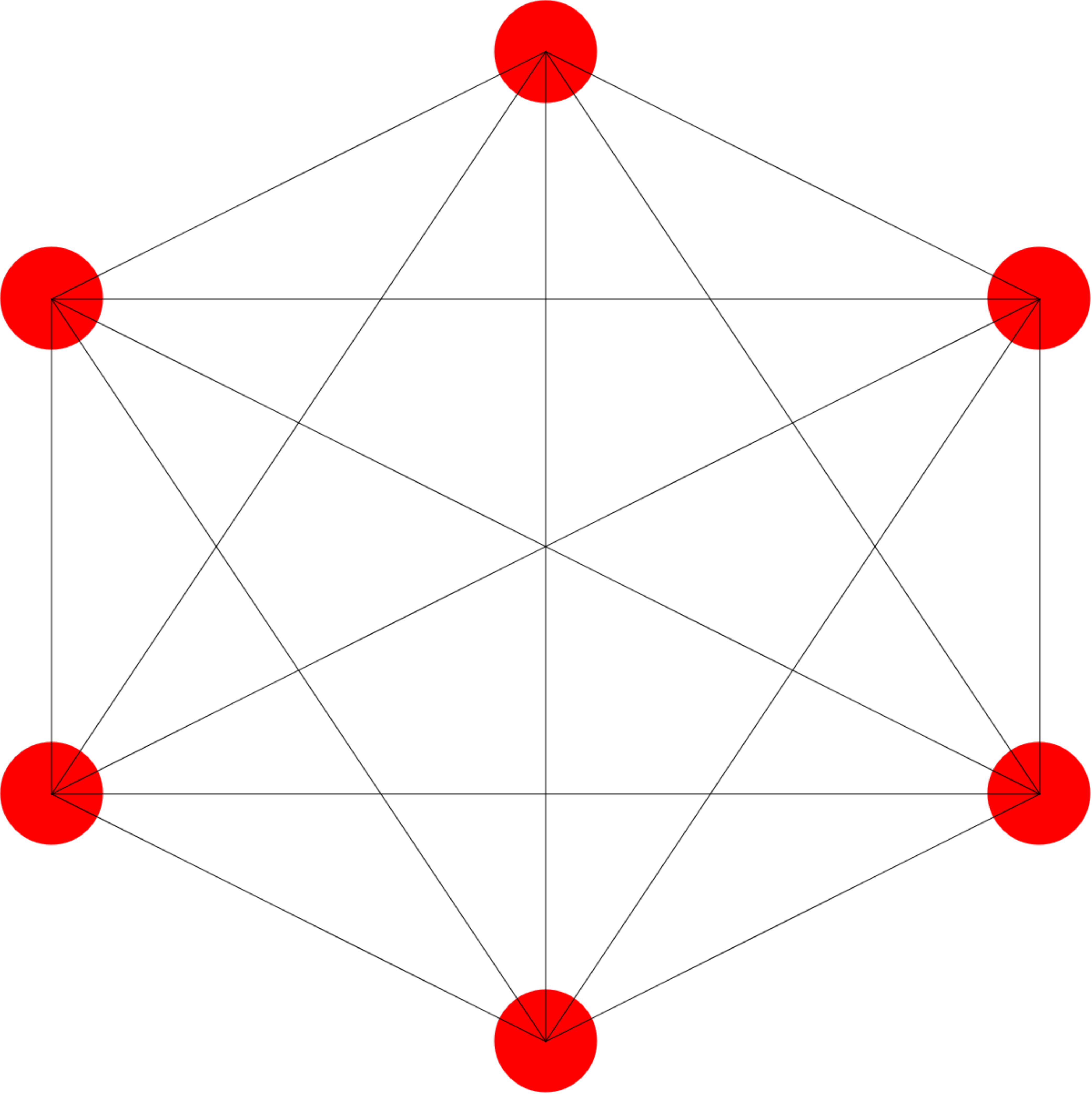}}
\caption{
The five graph multiplications $K_2$ with $K_3$.
}
\end{figure}

\paragraph{}
There are two classes of subgraphs which can play the role the integers. 
One is the class $P_n$ of graphs without edges if the addition is the 
disjoint union $\oplus$, the second
is the class $K_n$ of complete graphs if the addition is the join $+$. 
The ring homomorphisms are the obvious maps $n \to P_n$ or $n \to K_n$. 
Lets call the first class of graphs $\ZZ_P$ and the second class of graphs $\ZZ_K$. 

\begin{propo}
The weak, tensor and the strong rings contain the subring $\ZZ_P$. 
Their dual rings and in particular the Zykov ring contains the subring $\ZZ_K$. 
\end{propo}

\paragraph{}
The upshot is that the Zkyov ring has a compatible extension of the dimension
functional and that its dual has a compatible extension of Euler characteristic. 
The Zykov ring also has mixed additive-multiplicative compatibility
with Euler characteristic. The reduced Euler characteristic of the sum is the product
of the reduced Euler characteristics of the summands. 
The dual is compatible with cohomology but also has (when suitably deformed) a mixed
multiplicative-additive compatibility of dimension which is familiar from the continuum: the
dimension of the product is then the sum of the dimensions of the factors. 

\section{Zykov addition}

\paragraph{}
Given two finite simple graphs $G=(V,E), H=(W,F)$, the {\bf Zykov addition} is defined as
$$  G+H=(V \cup W, E \cup F \cup (a,b) \; a \in V, b \in W \}  \; . $$
As it is commutative and associative and the empty graph $0=(\emptyset,\emptyset)$ is
the zero element, it is a monoid. The set of equivalence classes of the form 
$A-B$ is defined as follows: define $A-B \sim C-D$ if there exists $K$ such that $A+D+K=B+C+K$.
These equivalence classes now form a commutative group. The addition $+$ corresponds to the {\bf join
operation} in topology. It has been introduced in 1949 by Zykov. The general 
construction of a group from a monoid has its roots in arithmetic but has
been formalized abstractly first by Grothendieck.

\begin{figure}[!htpb]
\scalebox{0.1}{\includegraphics{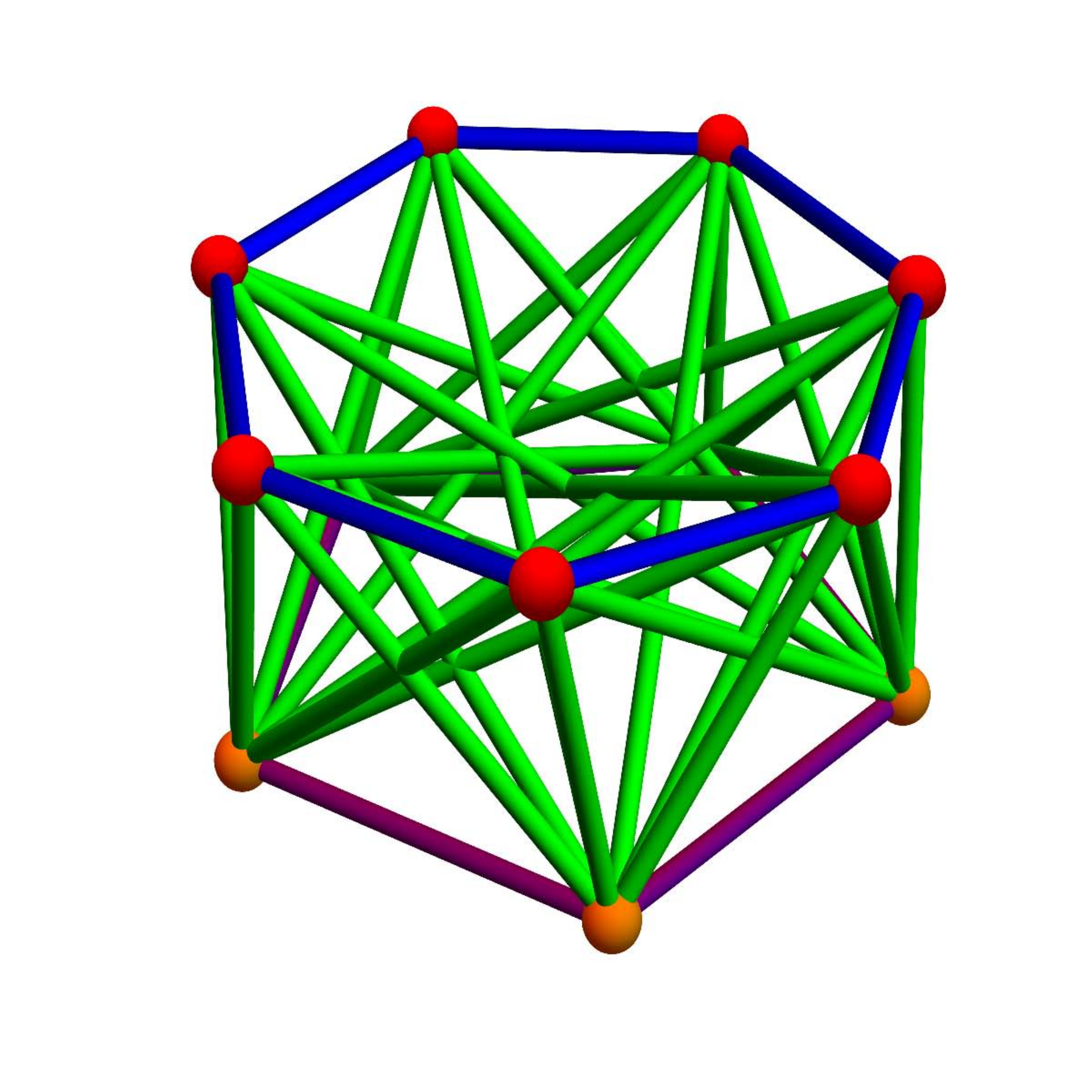}}
\scalebox{0.1}{\includegraphics{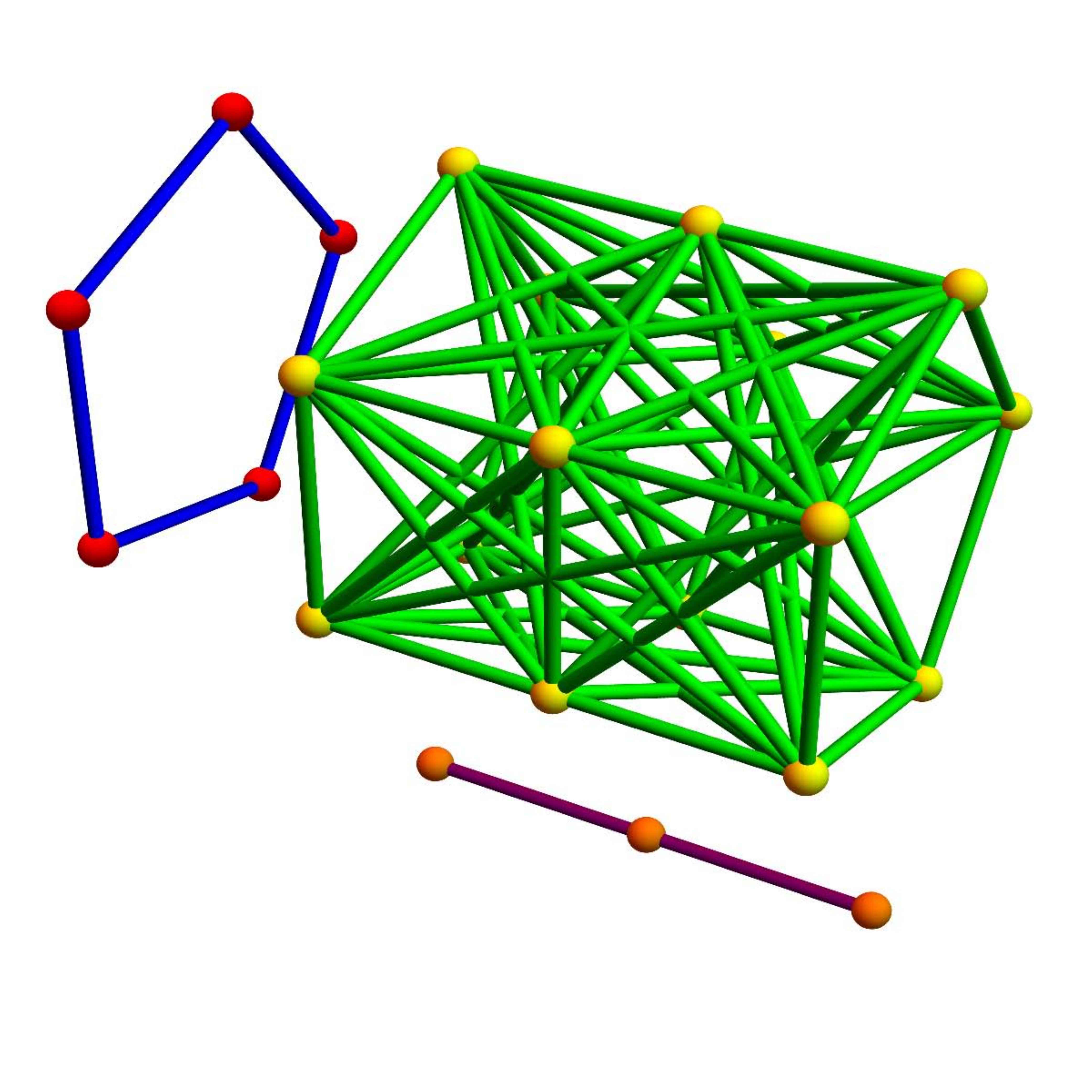}}
\caption{
Addition and multiplication in the Zykov ring. 
}
\end{figure}

\paragraph{}
For the additions $+$ and $\oplus$ it is no problem to extend the operation
to the larger class of simplicial complexes. 
A {\bf finite abstract simplicial complex} is a finite set of non-empty sets
closed under the operation of taking non-empty subsets. The disjoint union
of complexes produces a monoid. Also the Zykov addition 
can be extended to finite abstract simplicial complexes by defining
$$  G + H = G \cup H \cup \{ x \cup y \; | \; x \in G, y \in H \}  \; . $$
The empty set is the {\bf zero element}. The operation is obviously commutative and 
associative. One again gets a group containing elements of the form $A-B$ and identifies
$A-B \sim C-D$ if there exists $K$ such that $A+D+K=B+C+K$. 
If $G,H$ are finite simple graphs then the Whitney complex of the graph $G+H$
is the sum of the Whitney complexes of $G$ and $H$. The Zykov group of graphs 
therefore can be considered to be a sub-group of the group of simplicial complexes. 

\paragraph{}
In order to do computations in the additive Zykov group, lets look at 
the graphs $P_n$,the $n$-vertex graph with no edges, 
$K_n$ the {\bf complete graph} with $n$ vertices, $C_n$ the 
{\bf circular graph} with $n$ vertices and $W_n$ the {\bf wheel graph} with $n$
spikes so that for the central vertex $x$, the unit sphere $S(x)$ is $C_n$. 
Let $K_{n,m}$ denote the {\bf complete bipartite graph} of type $n,m$ and let
$S_n$ the {\bf star graph} with $n$ rays. Since $K_n$ plays the role of the
integer $n$ in $\mathbb{Z}$, we write also $n$ for $K_n$. \\

{\bf Examples.} \\
1) $P_n+P_m = K_{n,m}$ complete bipartite graph \\
2) $K_n+K_m = K_{n+m}$ integer addition \\
3) $K_1+G$ is the cone over $G$ \\
4) $P_2+G$ is the suspension over $G$ \\
5) $P_2+P_2+P_2$ octahedron  \\
6) $P_2+P_2+P_2+P_2$ 16-cell, three sphere \\
7) $K_1+C_n=W_n$ wheel graph \\
8) $K_1+P_n=S_n$ star graph \\
9) $P_3+C_n$ triple suspension of circle

\begin{figure}[!htpb]
\scalebox{0.3}{\includegraphics{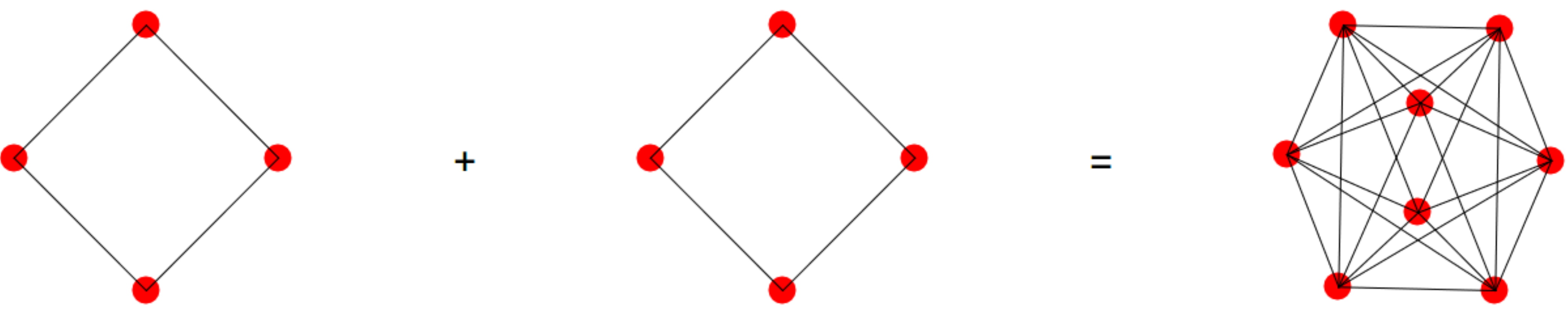}}
\scalebox{0.3}{\includegraphics{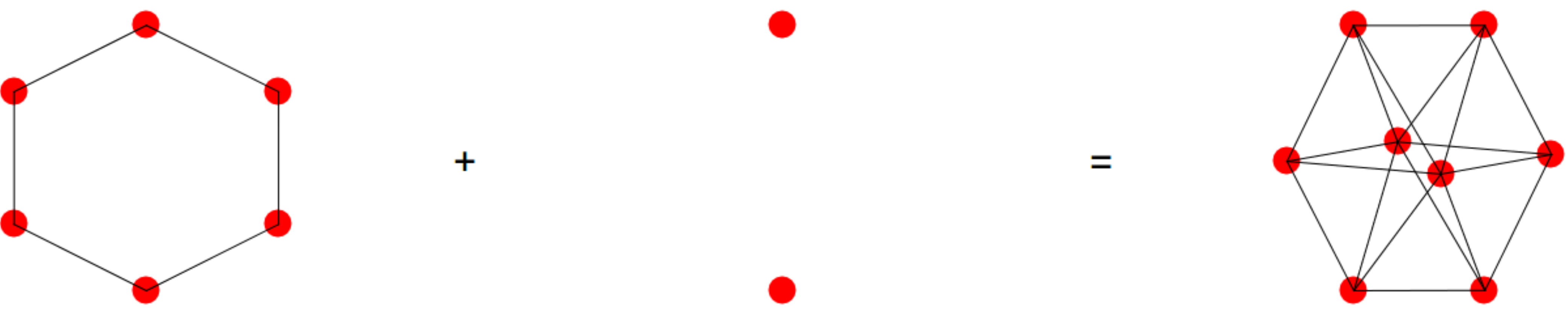}}
\scalebox{0.3}{\includegraphics{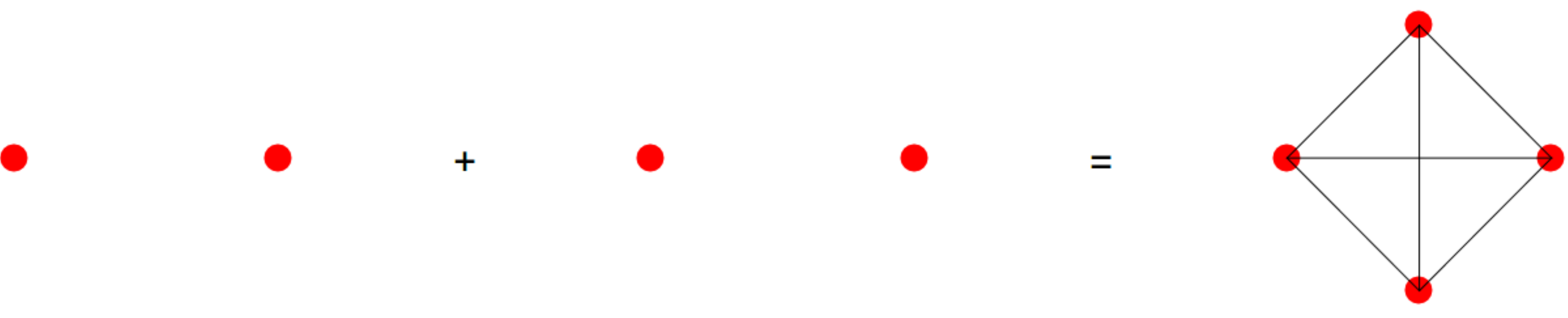}}
\scalebox{0.3}{\includegraphics{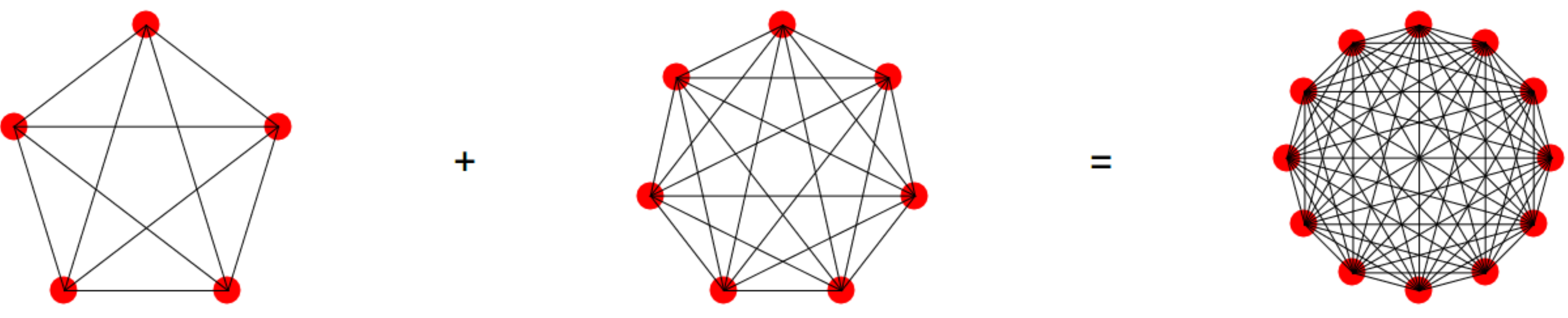}}
\scalebox{0.3}{\includegraphics{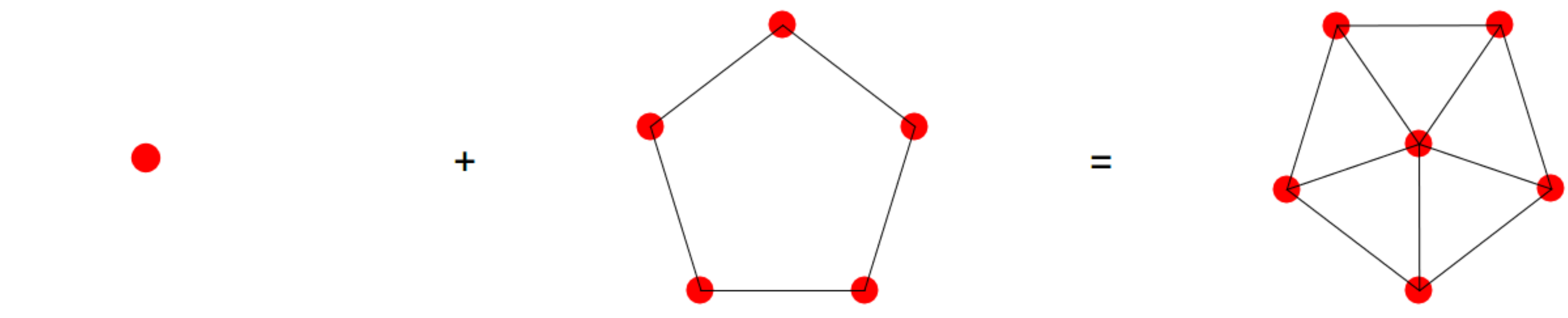}}
\caption{
Some examples of Zykov additions.
}
\end{figure}

\paragraph{}
Similarly as with numbers one can define classes of
graphs algebraically. For example, we can look at all graphs which can
be written as sums of $P_k$ graphs. Lets call them $P$-graphs.
If $G=P_{k_1} + \cdots + P_{k_d}$, then this graph has dimension 
$d-1$. Examples are the wheel graph $P_1+P_2+P_2$ of dimension $2$,
the complete bipartite graph $P_n + P_m$ of dimension $1$, 
the complete graph $P_1+P_1+ \cdots +P_1=d P_1$ of dimension $d-1$
or the {\bf cross polytopes} $P_2+P_2+ \cdots +P_2=d P_2$ which is a sphere 
of dimension $d-1$. 

\paragraph{}
We already know some thing about the addition: the functional $i(G) = 1-\chi(G)$ is
multiplicative. It can be interpreted as a genus or as a reduced Euler characteristic. 
If $f(G) = (v_0(G), v_1(G), \dots v_{d(G)}(G))$ is the $f$-vector of 
$G$ with dimension $d(G)$, then the clique number $c(G) = d(G)+1$ is additive. 
The {\bf volume} $V(G) = v_{d(G)}(G)$ is multiplicative. The Fermi functional 
$\prod_x (-1)^{{\rm dim}(x)}$ equal to the deteriminant of the connection Laplacian of $G$
which is the Fredholm determinant of the adjacency matrix of the connection graph $G'$ of $G$. 
We also have $f_{G + H} = f_G f_H$, where  $f_G(x)= 1+\sum_{k=0} v_k x^{k+1}$ is the 
$f$-generating function of $G$. This implies $\chi(G+H) = \chi(G) + \chi(H) - \chi(G) \chi(H)$
which is equivalent to the multiplicative property of $i(G)$. The Kirchhoff Laplacian
of $G+H$ has an eigenvalue $v_0(G) + v_0(H)$. We also know that the
second eigenvalue $\lambda_2(G)$ of the Kirchhoff Laplacian satisfies
$\lambda_2(G+H) = {\rm min}(v_0(G),v_0(H)) + {\rm min}(\lambda_2(G),\lambda_2(H)$. 
Finally, if $D=d+d^*$ be the Dirac operator of $G$ then the Hodge Laplacian 
$H=D^2 = (d+d^*)^2$ splits into blocks $H_k$ for which the nullity is the $k$'th Betti number
${\rm dim}({\rm ker})(H_k) = b_k(G)$ (\cite{DiracKnill,knillmckeansinger}). 
The spectrum $\sigma_{d(G)}(G)$ of $L_{d(G)}$ is the volume spectrum of $G$. 
We know that the volume eigenvalues of $G+H$ are of the form
$\lambda+\mu$, where $\lambda \in \sigma_{d(G)}(G)$ and $\mu \in \sigma_{d(H)}(H)$. 

\paragraph{}
Given a finite simple graph $G=(V,E)$, the {\bf graph complement}
$\overline{G}=(V,\overline{E})$ is defined by $\overline{E}$, the complement of $E$
in the edge set of the complete graph on $V$. This produces an involution on the set
of all finite simple graphs. The disjoint join operation $\oplus$ and the join $+$
are conjugated by this duality: 

\begin{lemma}
$G + H = \overline{\overline{G} \oplus \overline{H}}$. 
\end{lemma}
\begin{proof}
Let $e$ be first an edge in $G$. Then it is not in $\overline{E}$ and
also not in $\overline{G} \oplus \overline{H}$. So, it is in $\overline{\overline{G} \oplus \overline{H}}$. 
The same holds if $e$ is an edge in $H$. Let now $e$ be an edge in $G + H$ which connects vertices
from different graphs. Now since $e$ is not in $\overline{G} \oplus \overline{H}$, it is in 
$\overline{\overline{G} \oplus \overline{H}}$. 
\end{proof} 

We can derive again: 

\begin{coro}
The additive Zykov monoid has the unique factorization property. 
\end{coro}
\begin{proof} 
The monoid obtained by taking the disjoint union has the
unique factorization property. The duality functor carries
this to the join addition. 
\end{proof} 

{\bf Remark.} \\
The unique factorization property extends to the group. This is similar as in the
integers where $-3$ can be considered a prime in $\mathbb{Z}$ as it is only divisible by itself or
a unit. What are the units in the Zykov ring? These are the graphs which have a multiplicative
inverse. $G \cdot H = K_1$ however is only possible if the graph has one vertex. This means that
$G$ is either $K_1$ or $-K_1$. We have to check about the units as in the Gaussian ring for example,
more units have appeared $\{1,-1,i,-i\}$ which would have been possible also here. By the way, 
one can of course look at {\bf ring extensions} like $\G[i]$ for networks and there, the set
of units would be larger. \\

For $\oplus$, the primes are all the connected graphs. This means that the graph complement
of a connected graph in a complete graph is a prime for $+$. \\

{\bf Examples:} \\
 {\bf 1)} $C_4$ is not prime. But $\overline{C_4} = K_2 \oplus K_2$ is neither.  \\
 {\bf 2)} $C_5$ is prime. And $\overline{C_5}=C_5$ is also.  \\
 {\bf 3)} $C_6$ is prime. And $\overline{C_6}=K_3 \times K_2$ is too.  \\
 {\bf 4)} $C_7$ is prime. And $\overline{C_7}$ is a discrete Moebius strip.  \\
 {\bf 5)} $C_8$ is prime. And $\overline{C_8}$ is already a  three dimensional graph. \\
 {\bf 6)} $W_4 = P_2 + P_2 + P_1$. Indeed $\overline{W_4} =K_2 \oplus K_2 \oplus K_1$. \\
 {\bf 7)} $W_5 = C_5 + P_1$. And $\overline{W_5} = C_5 + K_1$. 

\section{Zykov Multiplication}

\paragraph{}
Having an addition on graphs, it is natural to look for multiplications
which satisfies the distributivity law. We were led to such a multiplication in the winter of 
2016 in \cite{Spheregeometry} after doing a systematic search.
The compatible multiplication has later turned out to be the dual of the strong 
graph multiplication. We call it the {\bf Zykov product}. 

\paragraph{}
Given two finite simple graphs $G=(V,E), H=(W,F)$, define the {\bf Zykov product}
$$ G \cdot H=(V \times W, E \cup F \cup 
           \{ ((a,b),(c,d))  \; | \; (a,c) \in E \; {\rm or} \; (b,d) \in F \; \}  \; . $$ 
It is commutative and associative and has as the $1$ element the graph $K_1$. Furthermore, we have
$0 \cdot G = 0$. We also have $G \cdot H=0$ if and only if one of them is the empty graph.
Also in this multiplicative monoid, we can extend the operation to a group. 

\paragraph{}
Here is an extension of the Zykov product $\cdot$ to the larger class of simplicial complexes. 
Assume $G$ is a set of subsets of $X$ and $H$ is a set of subsets of $Y$. Define the projections $\pi_k$
and take 
$$ G \cdot  H = \{ A  \subset X \times Y \; | \; \pi_1(A) \in G \; {\rm or} \;  \pi_2(A) \in H \} \; . $$
This means $G \cdot H = G \times P(Y) \cup P(X) \times H$ where $P(X)$ is the set of all subsets of $X$. 
Again the distributivity law $G \cdot (H+K) = G \cdot H + G \cdot K$ holds. 
One can also get this by defining a {\bf complementary simplicial complex} $\overline{G}$ of $G$. It can 
be defined as the set of subsets of $V=\bigcup_A A \in G$ which have the property that it does not contain
any $A \in G$ {\bf of positive dimension}. Now one has again $G + H = \overline{\overline{G} \oplus \overline{H}}$.

\paragraph{}
For any two graphs $G,H$, the {\bf strong product} $G \osquare H$ 
contains the {\bf tensor product} $G \otimes H$ and the {\bf weak Cartesian 
product} $G \square H$ as subgraphs. For the weak product $\otimes$ we 
know that if $A(G)$ is the adjacency matrix of $G$, 
then $A(G) \otimes A(H) = A(G \otimes H)$, where the former product 
is the tensor product of adjacency matrices.

\paragraph{}
Lets come back to the Zykov product: 

\begin{lemma}[Distributivity]
Multiplication is compatible with addition: $G \cdot (H+K) = G \cdot H + G \cdot K$. 
\end{lemma}
\begin{proof}
Both $G \cdot H + G \cdot K$ as well as $G \cdot (H+K)$ have as the vertex set 
the product sets of the vertices.
As edges in the sum $H+K$ consist of three types, connections within H,
connections within K and any possible connection between H and K, two points
$(a,b), (c,d)$ are connected if either $(a,c)$ is an edge in $H$, or
$(b,d)$ is an edge in K or then if either $a,c$ or $b,d$ are in different graphs.
\end{proof}

This lemma follows also by complementary duality. If we know that the graph tensor product
and graph Cartesian product both are compatible with the disjoint union operation $\oplus$, then 
also their union is.

\begin{figure}[!htpb]
\scalebox{0.3}{\includegraphics{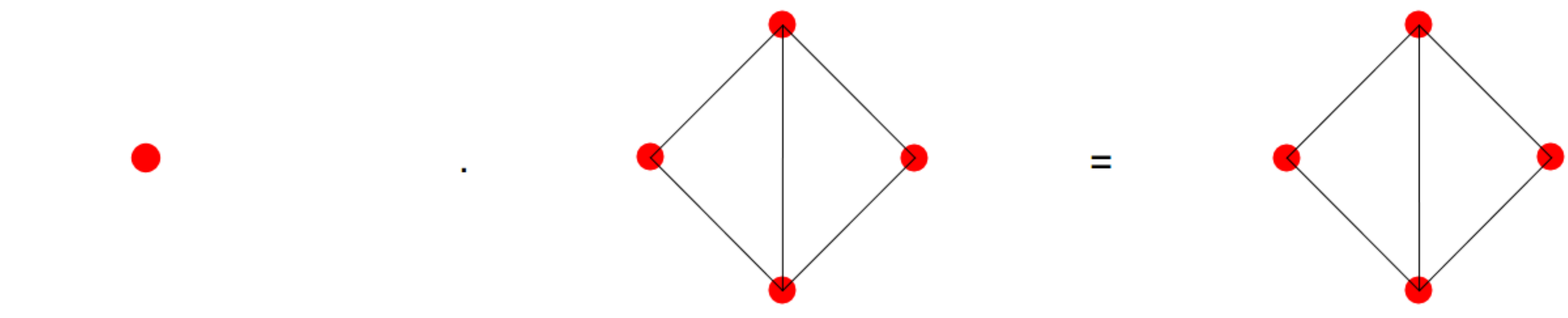}}
\scalebox{0.3}{\includegraphics{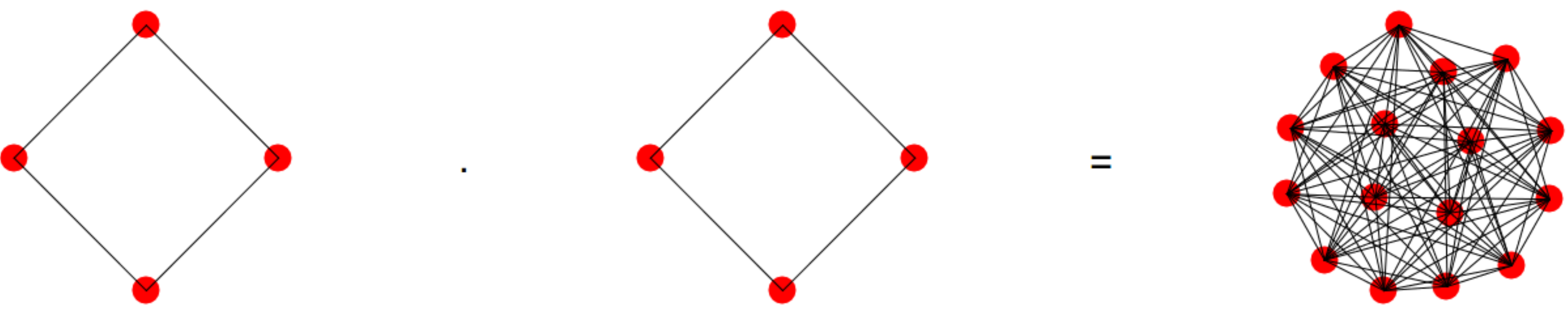}}
\scalebox{0.3}{\includegraphics{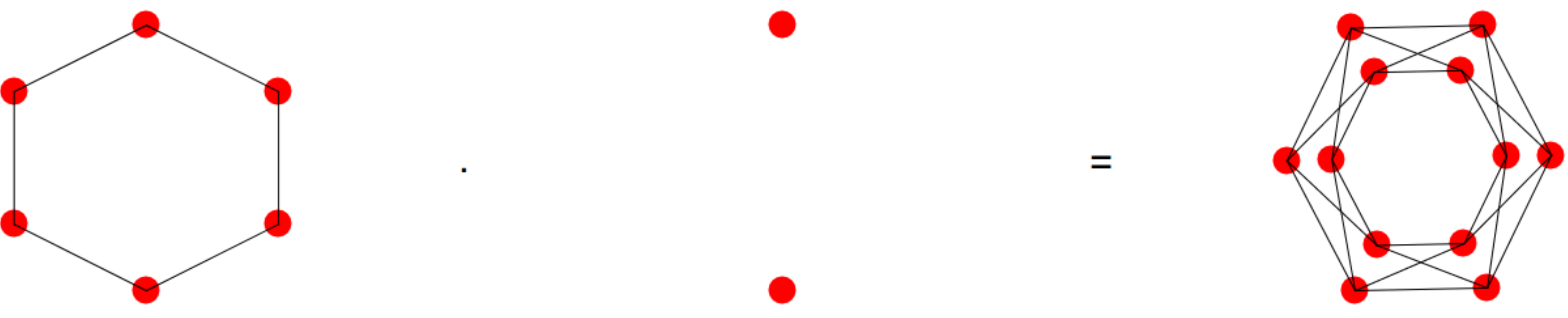}}
\scalebox{0.3}{\includegraphics{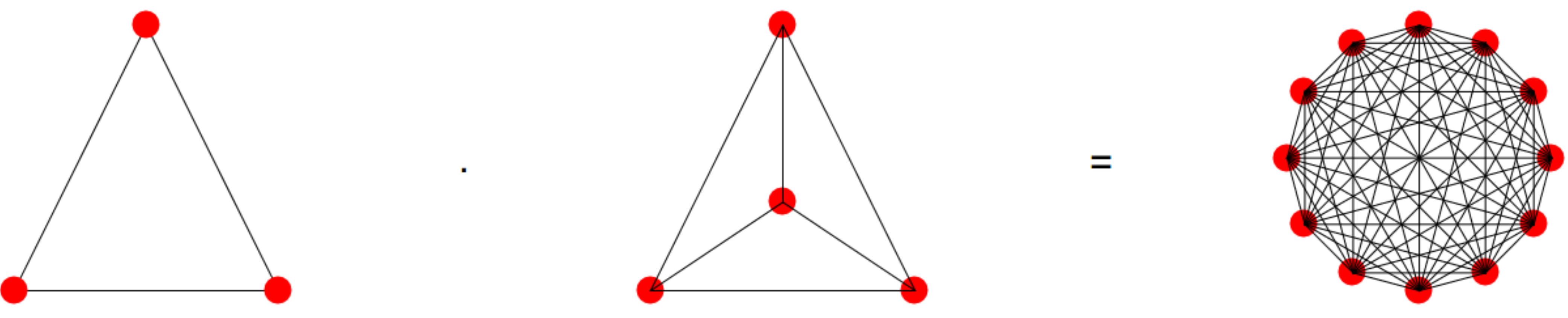}}
\scalebox{0.3}{\includegraphics{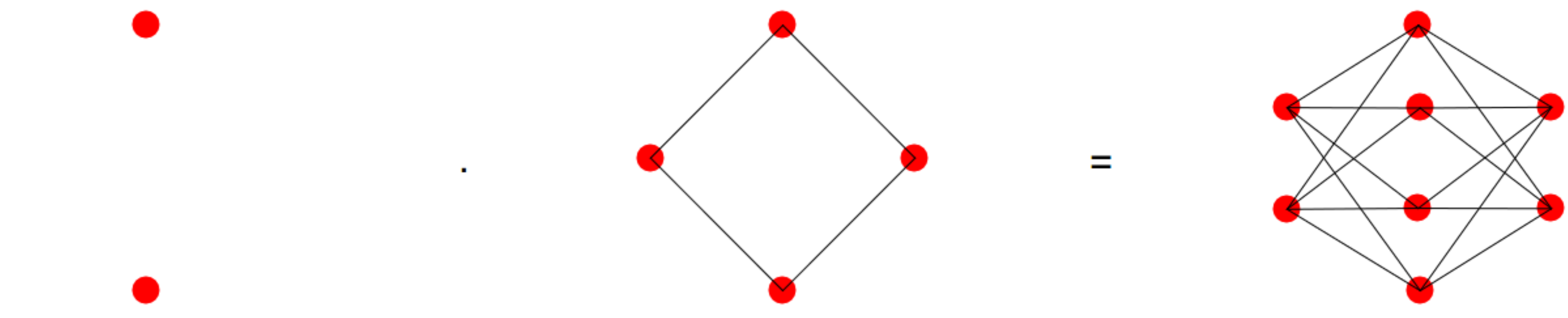}}
\caption{
Some examples of Zykov multiplications of graphs.
}
\end{figure}

\section{Computations in the Zykov ring}

\paragraph{}
Both the Zykov addition $+$ as well as the Zykov multiplication produce {\bf monoid structures} on 
the category of graphs or on the more general category of abstract simplicial complexes. 
For the addition, the empty graph or empty complex is the {\bf zero element}. 
For the multiplication the graph $K_1$ or one point simplicial complex is the {\bf one element}. 
A general construct of Grothendieck allows to produce a group from a monoid: in the additive case 
we look at all pairs $G-H$ of complexes and call $G-H = U-V$ if there exists $K$ such that $G+V+K = U+H+K$. 
In the multiplicative case, look at all pairs $G/H$ of complexes and call 
$G/H = U/V$ if there exists $K$ such that $G \cdot V \cdot K = U \cdot H \cdot K$. 

\paragraph{}
As we will see the multiplication $\cdot$ is not a unique factorization domain. 
Does the multiplicative monoid have the {\bf cancellation property}?
If the multiplicative monoid in a ring has the cancellation property, then the ring is called 
a {\bf domain}. If there is no possibility to write $x*y=0$ without one of the factors being zero,
the ring is called an {\bf integral domain}. Many rings are not integral domains:
like $Z_6$ in which $2*3=0$ or the product ring
$\mathbb{Z}^2$ where $(0,1)*(1,0)=(0,0)$. The ring of diagonal $3 \times 3$ matrices is an 
example of a ring which does not have the cancellation property as $G \cdot K=H \cdot K$ does not imply $G=H$ in general
as the case when $G,H,K$ be the projections on the x-axes, y-axes and z-axes shows. 

\paragraph{}
\begin{lemma}
While not unique factorization domains, the weak and strong rings have the cancellation property.
They are also integral domains. Consequently the Zykov ring is an integral domain but not a 
unique factorization domain. 
\end{lemma}
\begin{proof}
This is covered in section 6.5 of \cite{HammackImrichKlavzar}. 
The proof given there defines a ring homomorphism from $\G$ to 
an integral domain of polynomials. Having the property for the 
strong ring gives the property for the Zykov ring by the complement
duality. 
\end{proof}

\paragraph{}
Lets for a moment go back to the additive monoid. Before seeing the duality connection, we
searched for a direct proof of the unique prime factorization property for the Zykov 
monoid $(\G,+)$, where $+$ is the join.
The unique prime factorization property can be illustrated that we know about
$f$-generating functions $f(G) = 1+\sum_{k=0} v_k x^{k+1}$ of $G$ and 
$f(H) = 1+\sum_{k-0} w_k x^{k+1}$ of $H$ are the same, as $Z[x]$ is a unique
factorization domain. This means that the $f$-vectors of $G$ and $H$ agree. 
About the cancellation property: 
assume $G+K = H+K$, then $G=H$. Especially, if $G'$ is a suspension of $G$
and $H'$ is a suspension of $H$ and $G',H'$ are isomorphic graphs, then $G$ and
$H$ are isomorphic graphs.
To show that $G + K_1 = H + K_1$ implies $G = H$ one can use that every simplex 
containing the new point $x$ and every isomorphism $T$ from $G + K_1$ to $H + K_1$ induces a
permutation on the facets. Assume that the isomorphism $T$ maps $x$ to $y$, the
other case where the isomorphism maps $x$ to $x$ is similar. The isomorphism of dimension $k+1$
simplices induces after removing $x$ and $y$ an isomorphism of $k$-simplices from $G$ to $H$. 
Now, as all $k$, especially an isomorphism of $dim=1$ simplices which is a graph homomorphism.

\paragraph{}
We have now a ring of networks in which we can do some computations.
Lets look at some examples. Since $K_n \cdot K_m = K_{nm}$, one can abbreviate 
$n G$ for $K_n \cdot G$.  \\

a) $P_n \cdot P_m = P_{nm}$ \\
b) $K_n \cdot K_m = K_{nm}$ \\
c) $3 P_2$  octahedron \\
d) $P_2 \cdot G$ doubling  \\
e) $C_4 \cdot C_4 = (K_2+P_2) \cdot (K_2 + P_2) = K_4+ C_4 + P_4 $  \\
f) $K_2 \cdot C_n$ three sphere  \\
g) $K_2 \cdot P_2 = 2 P_2$ kite graph

\begin{figure}[!htpb]
\scalebox{1.0}{\includegraphics{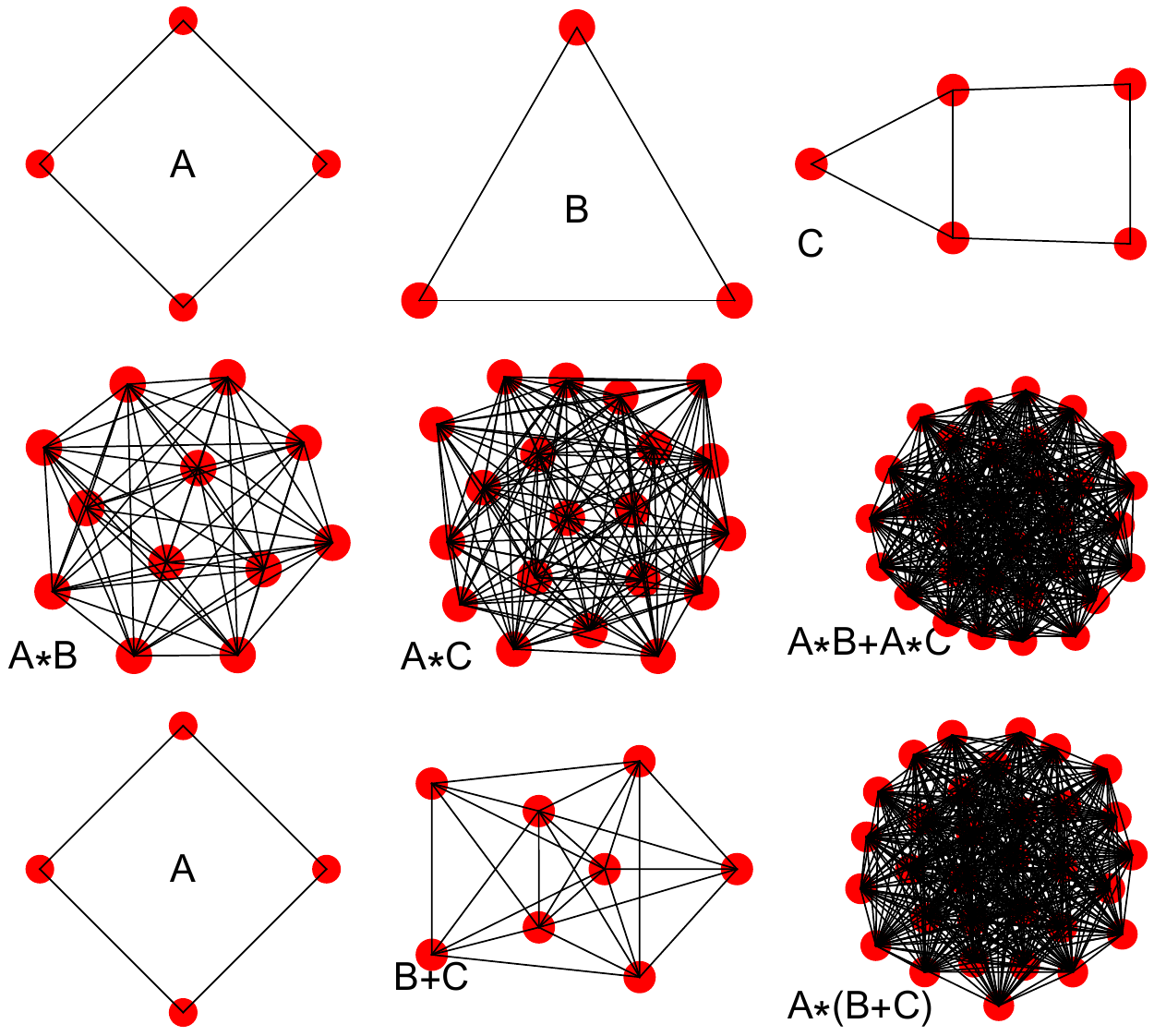}}
\caption{
Illustrating the distributivity in the Zykov ring. 
}
\end{figure}

\paragraph{}
With the disjoint union of simplicial complexes as addition, one can
look at the simplex Cartesian product $\times$ as a multiplication.
The set theoretical Cartesian product of two simplicial complexes is not a simplicial
complex in general and the Cartesian product defined in \cite{Kuenneth} is not associative as
$G_2=(G \times K_1) \times K_1 \neq G_1 \times (K_1 \times K_1) = G_1$. In order
to get a ring structure, one can define the product only on refinements of complexes 
$G_1 \times H_1 = G \times H$. The problem is that the product is then no more a refinement of a 
complex but when restricting to the set of graphs which are expressible as $G_1(1) \times G_1(2) \times \cdots \times G_1(k)$
we are fine. This is more convenient than building a data structure of CW complexes and essentially 
means to look at the structures in the Stanley-Reisner ring. We used that ring in \cite{Kuenneth}
without being aware of the Stanley-Reisner picture. 
All objects $G=(G_1 \times  \dots \times G_n)$ are equipped with
a CW structure so that the Euler characteristic is the product. There is a corresponding
connection graph $L'$ for which the energy is the Euler characteristic of $G$. When extending
Euler characteristic functional to the ring, it becomes a ring homomorphism to $Z$.
There is a relation between Cartesian and tensor product:
the connection graph of the Cartesian product is the tensor product of the connection graphs.

\paragraph{}
The tensor product has some relation to the Grothendieck product: the connection graph
Laplacians tensor if the product of the simplicial complexes is taken. This is an 
algebraic statement. \\
Its not true in general however that the Fredholm connection Laplacian 
of $G \times H$ is the tensor product of the Fredholm connection Laplacian of $G$ and $H$. 
It is only that the Fredholm connection matrices tensor.

\section{More examples} 

\paragraph{}
In this section, we perform a few example computations in the Zykov ring. 
The main building blocks are the {\bf point graphs} $P_n = (\{1,2,3, \dots, n\}, \emptyset)$, 
the {\bf complete graphs} $K_n = \overline{P_n}$, the {\bf cycle graphs} $C_n$ (assuming $n \geq 4$),
the {\bf star graphs} $S_n= P_n + 1$, (with the understanding that $S_2=P_2+1=L_2$ is the linear
graph of length $2$ and $S_1=P_1+1=L_2$). Then we have the 
{\bf wheel graphs} $C_n + 1$, the {\bf complete bipartite
graphs} $K_{n,m} = P_n + P_m$. Furthermore, we look at the
{\bf kite graph} $K=P_2 + K_2$, the {\bf windmill graph} $W=P_3 + K_2$ as well as 
the {\bf cross polytopes} $S^n = (n+1) \cdot P_2$ which are $n$-spheres, 
where especially $S^0=P_2, S^1=C_4$ and $S^2=O$, the {\bf octhahedron} graph are $0,1$
and $2$-dimensional spheres. (There should be no confusion as we do not use a graph with name $S$
as $S^d$ is not a power in the Zykov ring but a sum of zero dimensional spheres) Most of these
graphs are contractible with trivial cohomology $b(G)=(1,0,0,\dots)$.  The ones which are not, have the Betti numbers 
$b(P_n)=(n,0,0,..)$, $b(C_n)=(1,1,0,0,..)$, $b(S^n) = (b_0, \dots, b_n) = (1,0,0,\dots,1)$,
$b(K_{n,m}) = (1,(n-1)(m-1),0,\dots)$. 

\paragraph{}
{\bf Example A: the square of a circle}. 
Lets compute the square of the circular graph $C_4$ and show
\begin{center} \fbox{$C_4^2 = 2 K_{4,4}$. } \end{center} 

We use that $P_n \cdot P_m = P_{nm}$ and $P_n + P_m = K_{n,m}$ to get

\begin{eqnarray*}
C_4^2 &=& (2 P_2)^2 = 4 P_2^2 = 4 P_4  \\
      &=& 2 (P_4 + P_4) = 2 (K_{4,4}) \; . 
\end{eqnarray*}

It follows for example that $W_4^2=(1+C_4)^2 = 1+4 P_2 + 2 K_{4,4}$.

\begin{figure}[!htpb] \scalebox{0.3}{\includegraphics{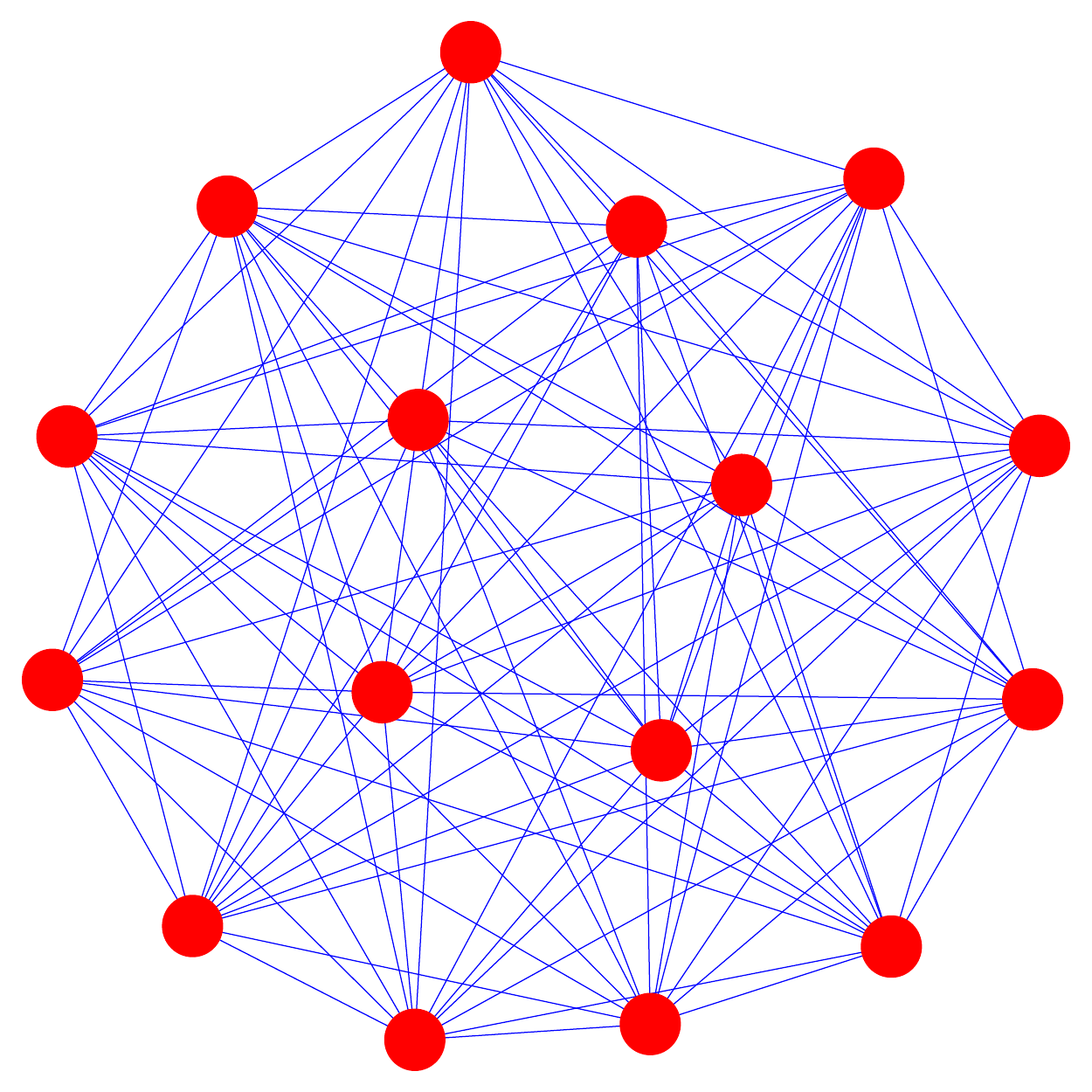}}
\caption{ The Zykov square of the circle $C_4$ is $K_{4,4}+K_{4,4}$.} \end{figure}

\paragraph{}
{\bf Example B: the square of a kite}.
Lets compute the square of the kite graph $K$ and show
\begin{center} \fbox{$K^2 = P_4 + K_4 + S^3$. } \end{center} 

We use
$P_n \cdot P_m = P_{nm}$ and $K_n \cdot K_m = K_{nm}$ and 
$2 C_4 = C_4 + C_4 = 4 P_2 = S^3$ as well as $K_2 \cdot P_2 = P_2 + P_2 = C_4$. 

\begin{eqnarray*}
K^2 &=& (P_2 + K_2)^2 = (P_2 + K_2) \cdot (P_2 + K_2) \\
    &=& (P_4 + K_4 + 2 K_2 \cdot P_2)  \\
    &=& P_4 + K_4 + 4 P_2 = P_4 + K_4 + S^3 \; . 
\end{eqnarray*}

\begin{figure}[!htpb] \scalebox{0.3}{\includegraphics{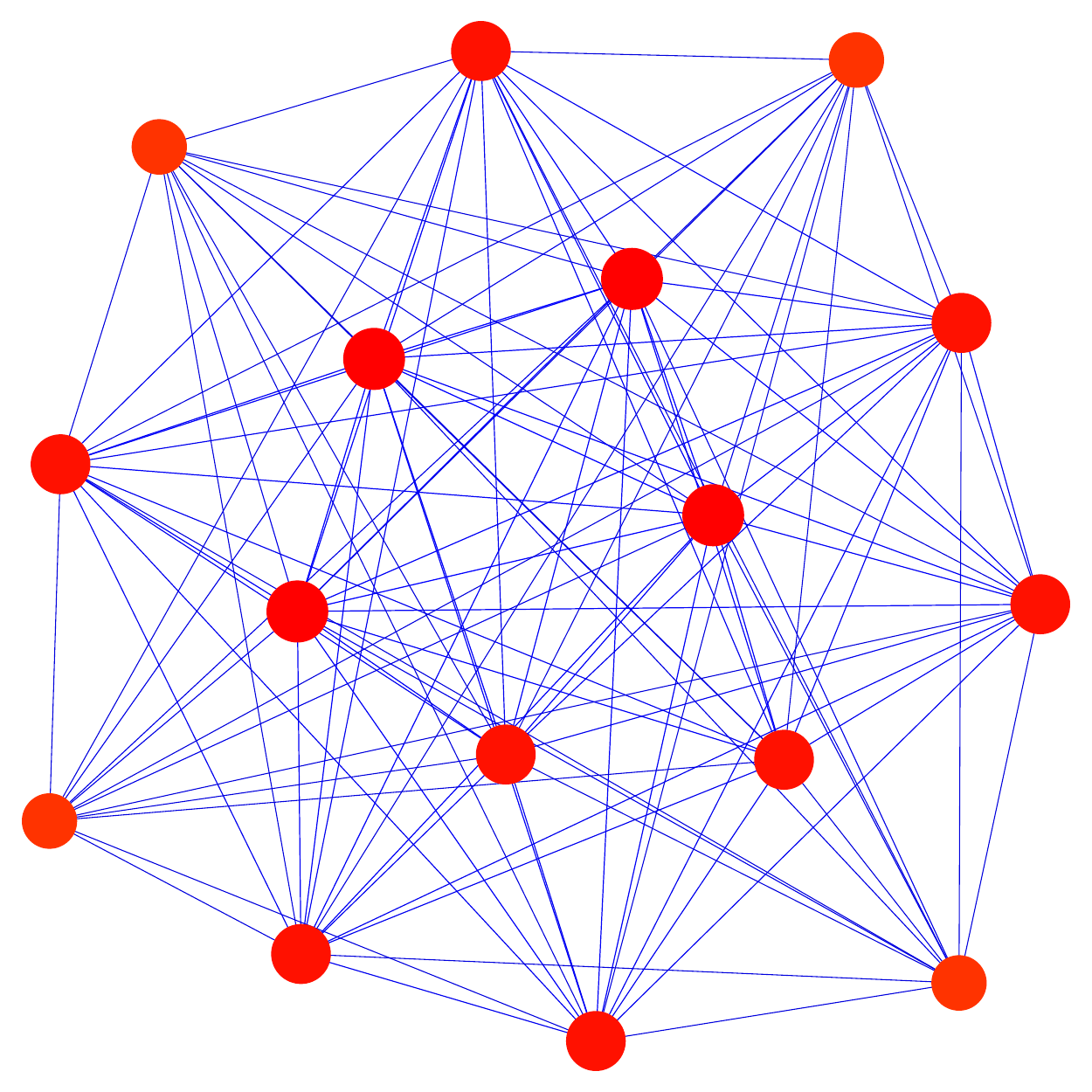}}
\caption{ The Zykov square of the kite graph is $P_4+K_4+S^3$.} \end{figure}

\paragraph{}
{\bf Example C: the square of a star.}  
\begin{center} \fbox{$S_n^2  = S_{n^2} + K_{n,n}$. } \end{center} 

We use that $S_n = P_n +P_1$ and $P_n \cdot P_m = P_{nm}$ and $P_n + P_m = K_{n,m}$ to 
get 

\begin{eqnarray*}
 S_n^2 &=& (P_1+P_n)^2 = P_1^2 + P_{n^2} + 2 P_1 \cdot P_n  \\
       &=& P_{n^2} + P_1 + 2 P_n = = S_{n^2} + (P_n + P_n) \\
       &=& S_{n^2} + K_{n,n} \; . 
\end{eqnarray*}

\begin{figure}[!htpb] \scalebox{0.3}{\includegraphics{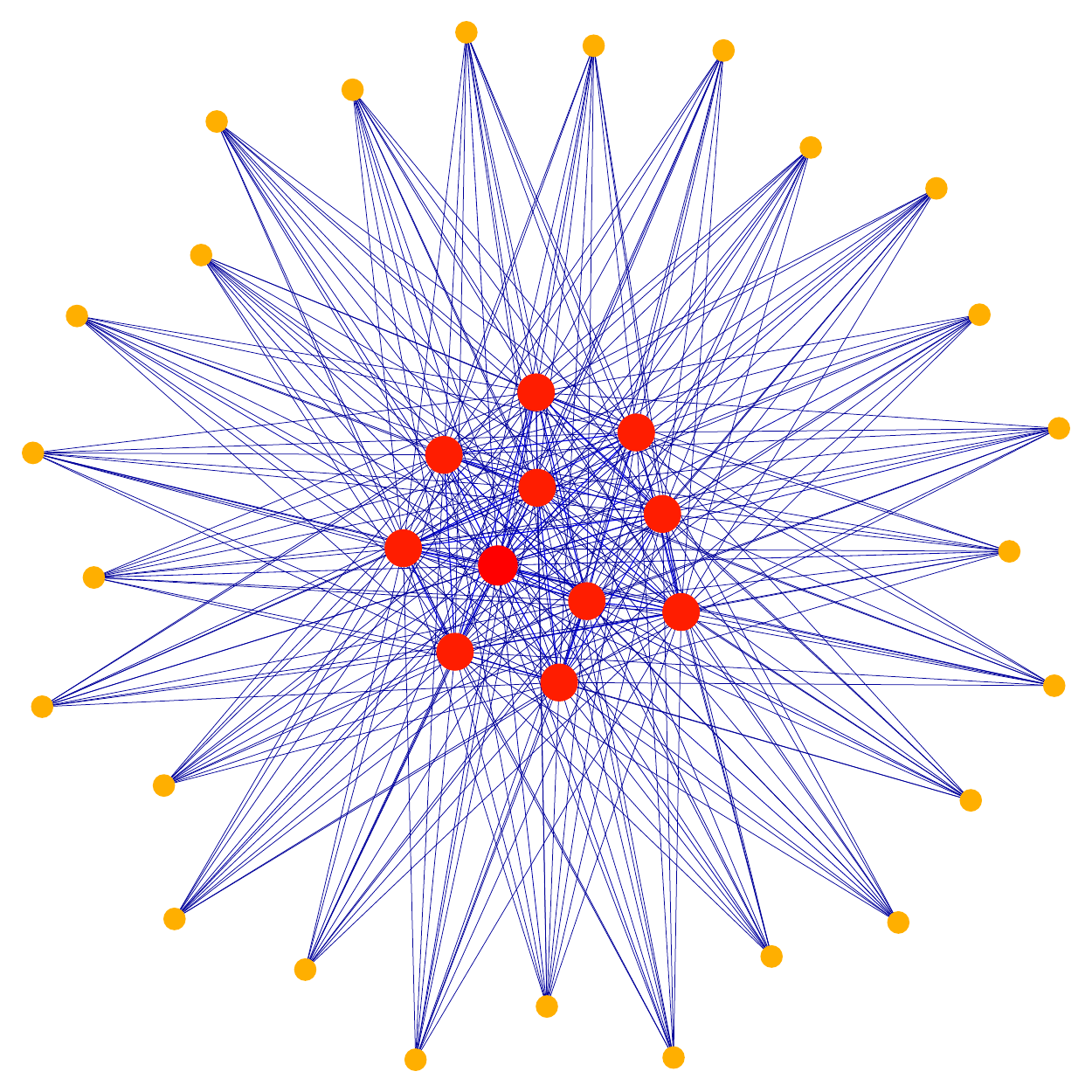}}
\caption{ The Zykov square of a star graph is the sum of a star graph and bipartite graph.} \end{figure}

\paragraph{}
{\bf Example D: the square of the windmill.}  \\
\begin{center} \fbox{$W^2  = P_9 + 4 S_3$. } \end{center} 

We use the definitions $W=P_3+K_2$ and $S_n=P_1+P_n$ and $P_n \cdot P_m = P_{nm}$ 
as well as $K_n \cdot K_m = K_{nm}$ to get

\begin{eqnarray*}
 W^2 &=& (P_3+K_2)^2 = P_9 + K_4 + 2 P_3 K_2  \\
     &=& P_9 + K_4 + 4 P_3 = P_9 + K_4 ( K_1 + P_3) \\
     &=& P_9 + 4 S_3 \; . 
\end{eqnarray*}

\begin{figure}[!htpb] \scalebox{0.3}{\includegraphics{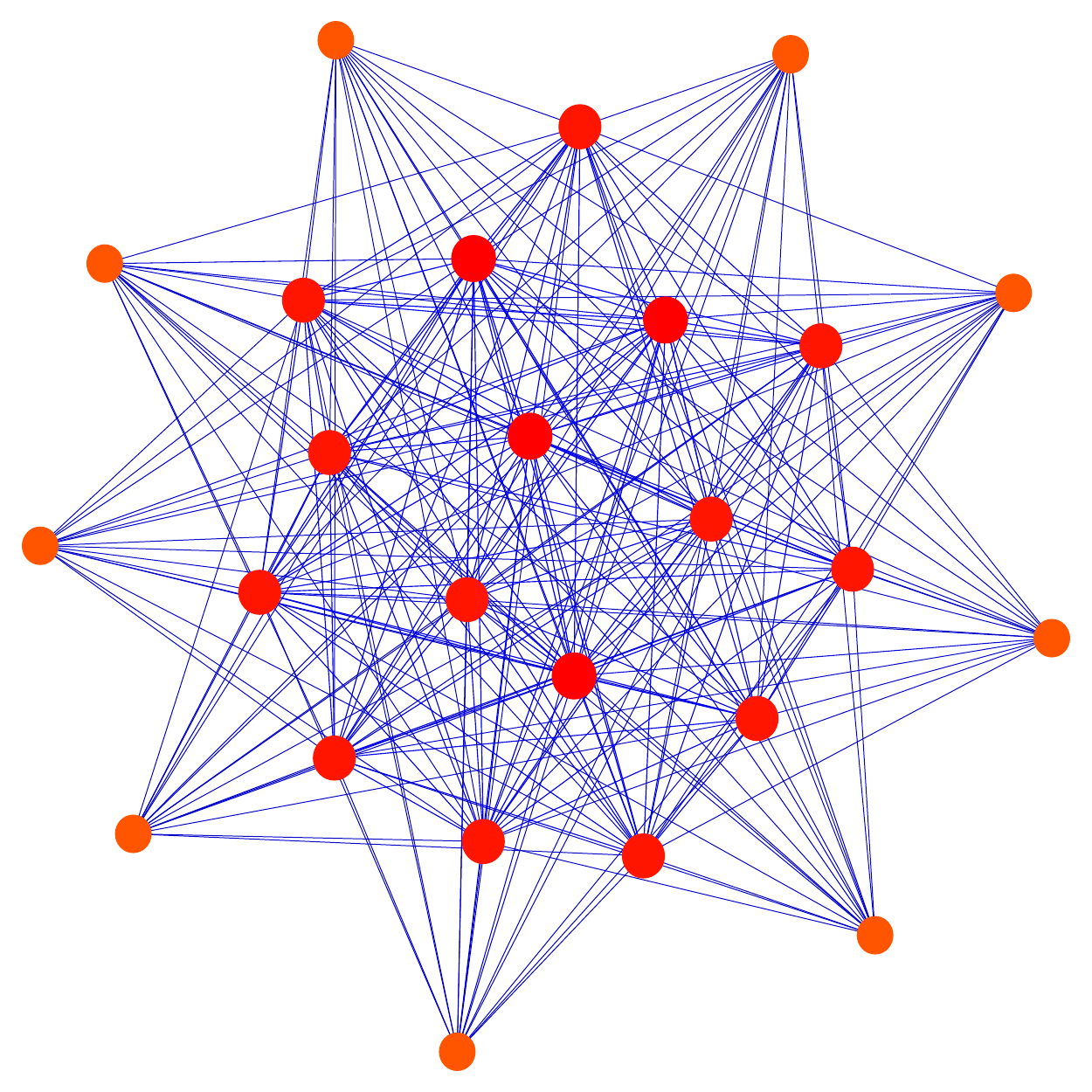}}
\caption{ The Zykov square of the windmill graph is $P_9+4 S_3$.} \end{figure}

\paragraph{}
{\bf Example E: the square of a complete bipartite graph.} 
\begin{center} \fbox{$K_{n,m}^2  = K_{n^2,m^2} + K_{nm,nm}$. } \end{center} 

By definition, we have $K_{n,m} = P_n + P_m$ and so $K_{n,m} + K_{n,m} = P_{2n} + P_{2m} = K_{2n,2m}$.
But now to the product: 

\begin{eqnarray*}
 K_{n,m}^2 &=& (P_n + P_m)^2 = P_n^2 + 2 P_n \cdot P_m + P_m^2 \\
           &=& P_{n^2} + P_{m^2} + (P_{nm} + P_{nm})  \\
           &=& P_{n^2} + P_{m^2} + K_{nm,nm} \\
           &=& K_{n^2,m^2} + K_{nm,nm} \; . 
\end{eqnarray*}

In the special case where $n=m$, we have 
\begin{center} \fbox{$K_{m,m}^k  = 2^k P_m^k$. } \end{center} 
Just write $(P_n+P_n)^k = (2 P_n)^k$. \\

We see that complete subgraphs are multiplicatively closed. This can also be seen 
by diagonalization: $\overline{K_{n,m}} = K_n \oplus K_m$. So that 
$\overline{K_{n,m}} \osquare \overline{K_{k,l}}  =  (K_n \oplus K_m) \osquare (K_k \oplus K_l)$
which can now be factored out and more generally identities like
$K_{n,m} \cdot K_{k,l}  = K_{nk,nl} + K_{mk,ml}$ hold.

\begin{figure}[!htpb] \scalebox{0.3}{\includegraphics{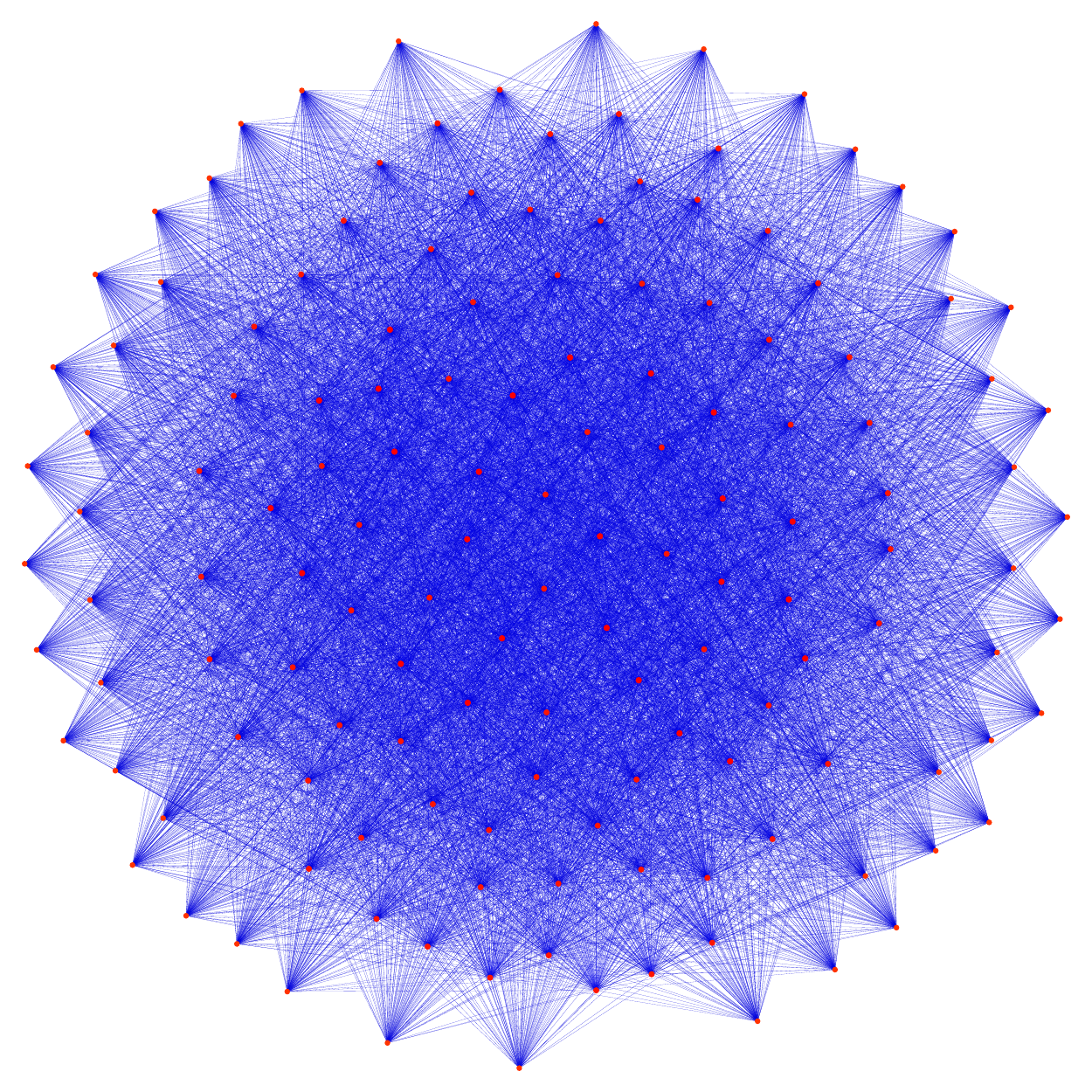}}
\caption{ The Zykov square of a bipartite graph $K_{n,m}$ is a sum of complete bipartite graphs.} \end{figure}

\paragraph{}
{\bf Example F: the square of the 3-sphere.} \\
\begin{center} \fbox{$(S^3)^2  = 4 P_4$. } \end{center} 

The definition is $S^d = (d+1) P_2 = P_2 + \cdots + P_2$ 
so that $(S^d)^2 = (d+1)^2 P_2^2 = (d+1)^2 P_4$
Especially 
$$ (S^3)^2  = K_{(3+1)^2} \cdot P_4 = 16 P_4 \; . $$

In the same way, we can compute the square of the octahedron, the 2-sphere, as 
\begin{center} \fbox{$O^2  = 9 P_4$. } \end{center} 

\begin{figure}[!htpb] \scalebox{0.3}{\includegraphics{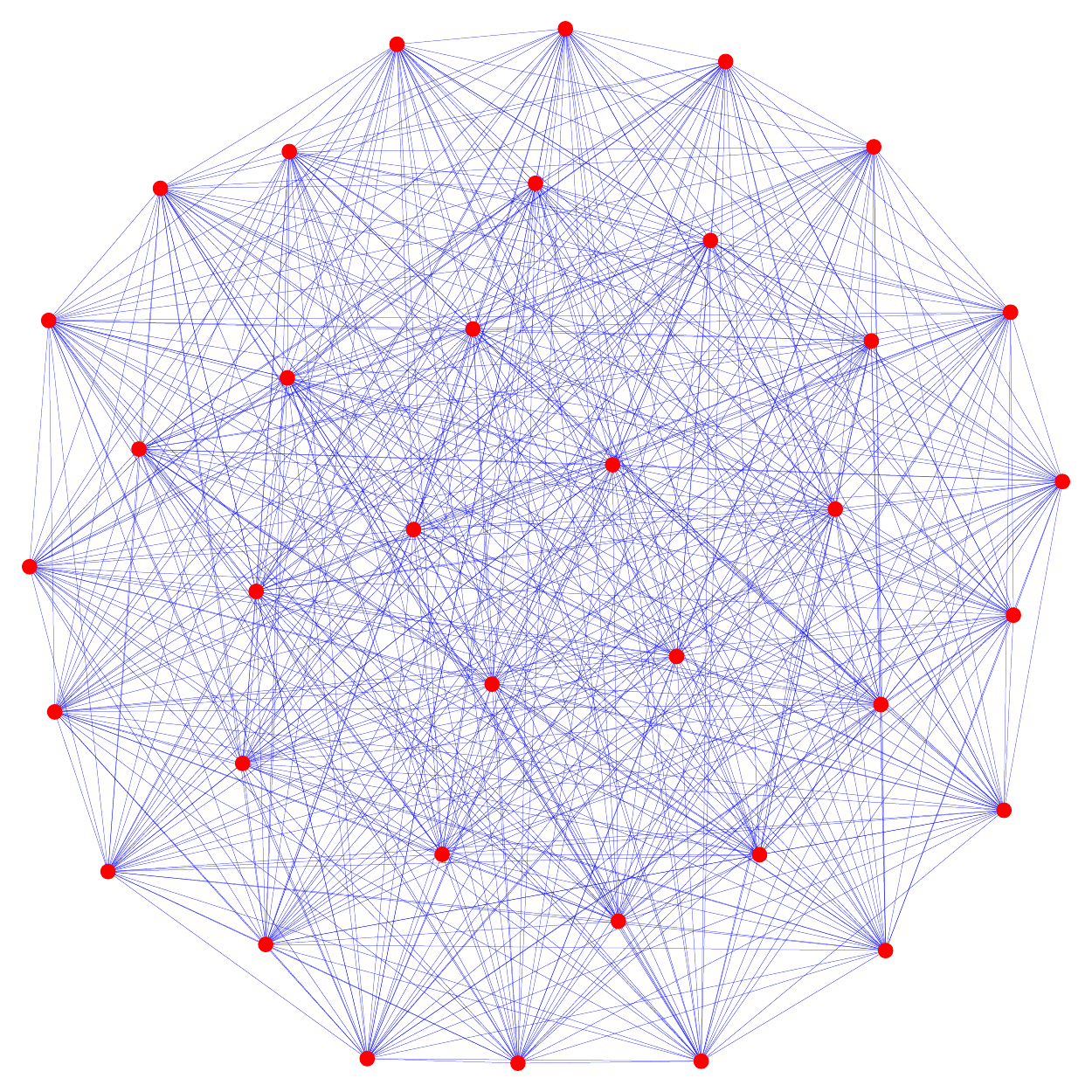}}
\caption{ The Zykov square of a three sphere.} \end{figure}

\paragraph{}
{\bf Example G: Subtracting a star from a sphere} \\
\begin{center} \fbox{$(C_4-S_4)\cdot (K_{4,4} - S_{16}$. } \end{center}

The computation
$$ (C_4-S_4) \cdot (C_4+S_4) = C_4^2-S_4^2 = 2 K_{4,4} - (S_{16} + K_{4,4} $$
follows from A) and B). We mention that as it is clear if we see somewhere an $A-A$, then
it can be reduced to $0$. \\

\paragraph{}
Most of these examples are {\bf $P$-graphs}, graphs generated by the point graphs $P_n$. 
Every $P$-graph has the form $\sum_{k=1}^n a_k P_k$, where $a_k \in \mathbb{Z}$.
Here is a summary of some small ``numbers": 

\begin{itemize}
\item $n=K_n$  complete graphs
\item $P_k P_n = P_{kn}$ point graphs
\item $P_k+P_n = K_{k,n}$ complete bipartite graphs
\item $(d+1) P_2 = S^d$ spheres
\item $2 P_2 = C_4$ circle graph
\item $1+2 P_2 = W_4$ wheel graph
\item $2+2 P_2$  three ball
\item $2+P_2$  kite graph
\item $1+P_2$  linear graph
\item $1+P_n$  star graphs
\item $1+2P_2$  wheel graph
\item $2+P_3$  windmill graph
\end{itemize}

\begin{figure}[!htpb]
\scalebox{0.08}{\includegraphics{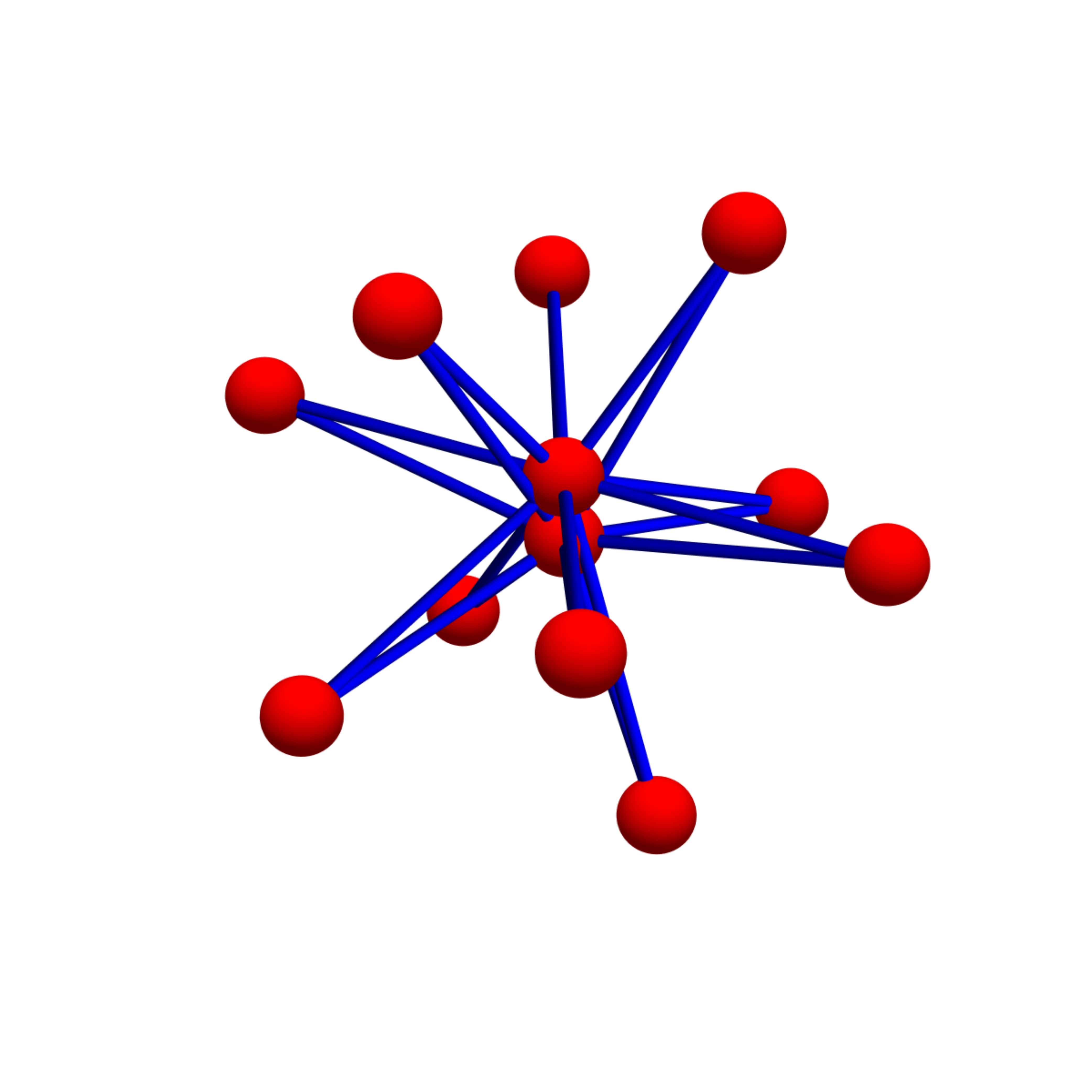}}
\scalebox{0.08}{\includegraphics{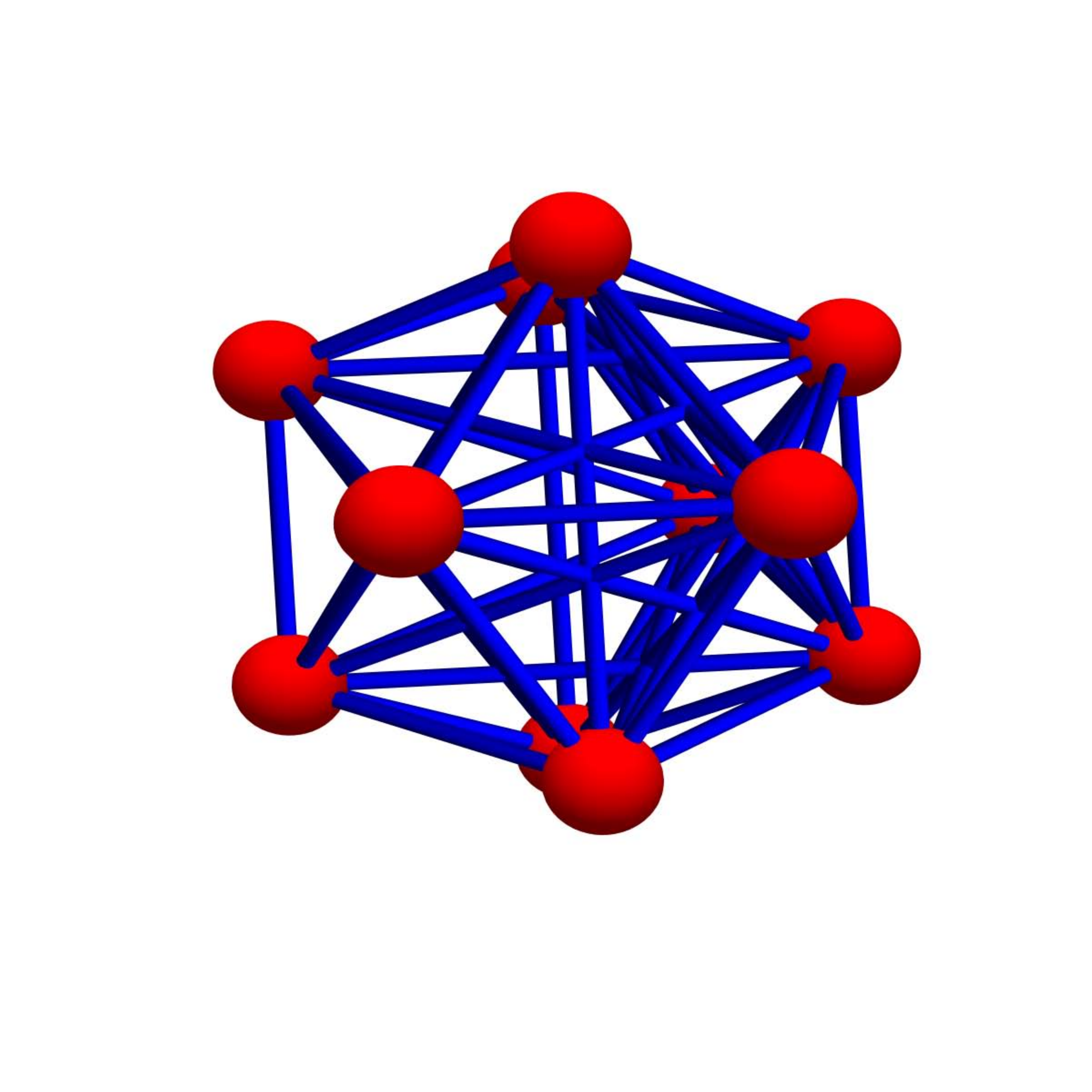}}
\scalebox{0.08}{\includegraphics{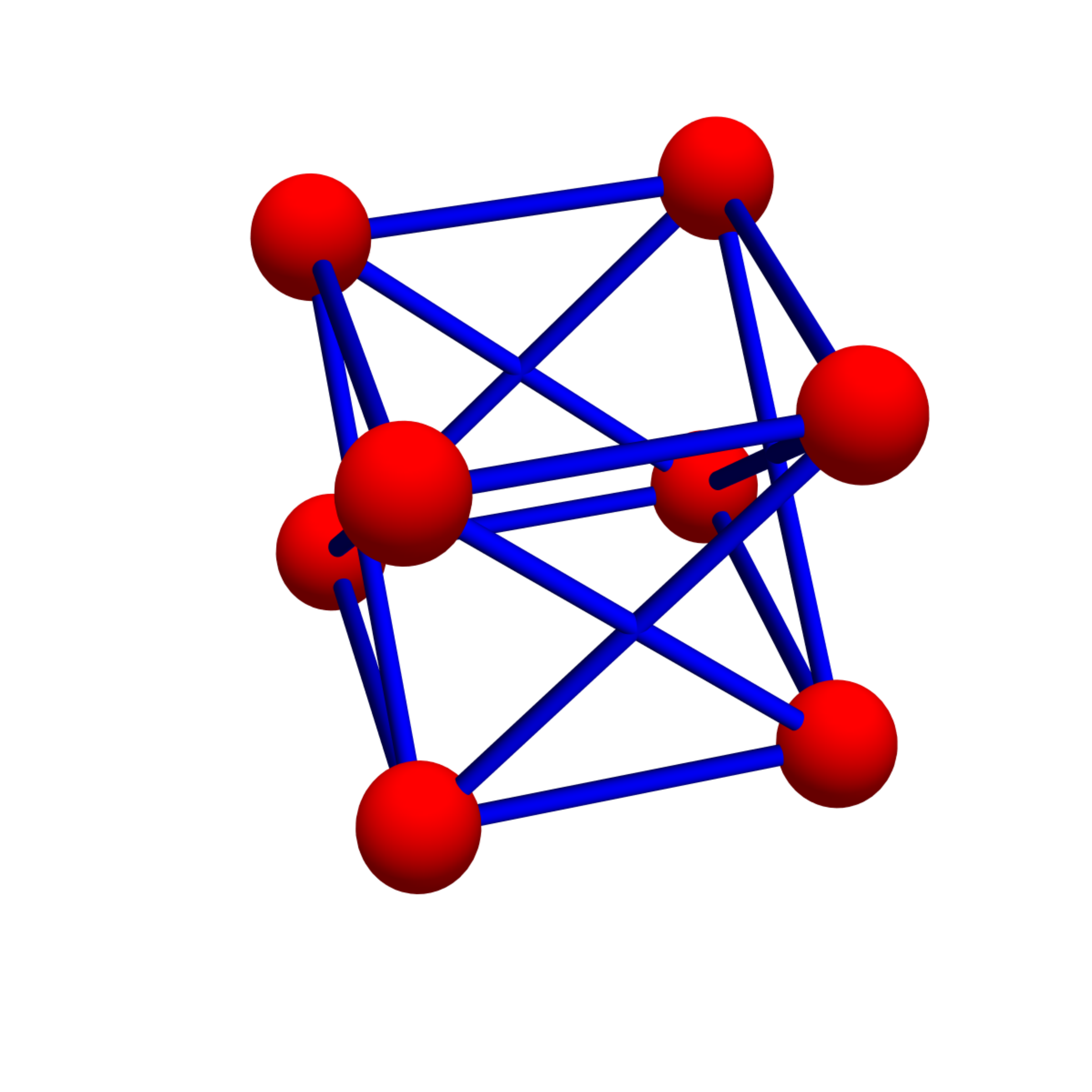}}
\scalebox{0.08}{\includegraphics{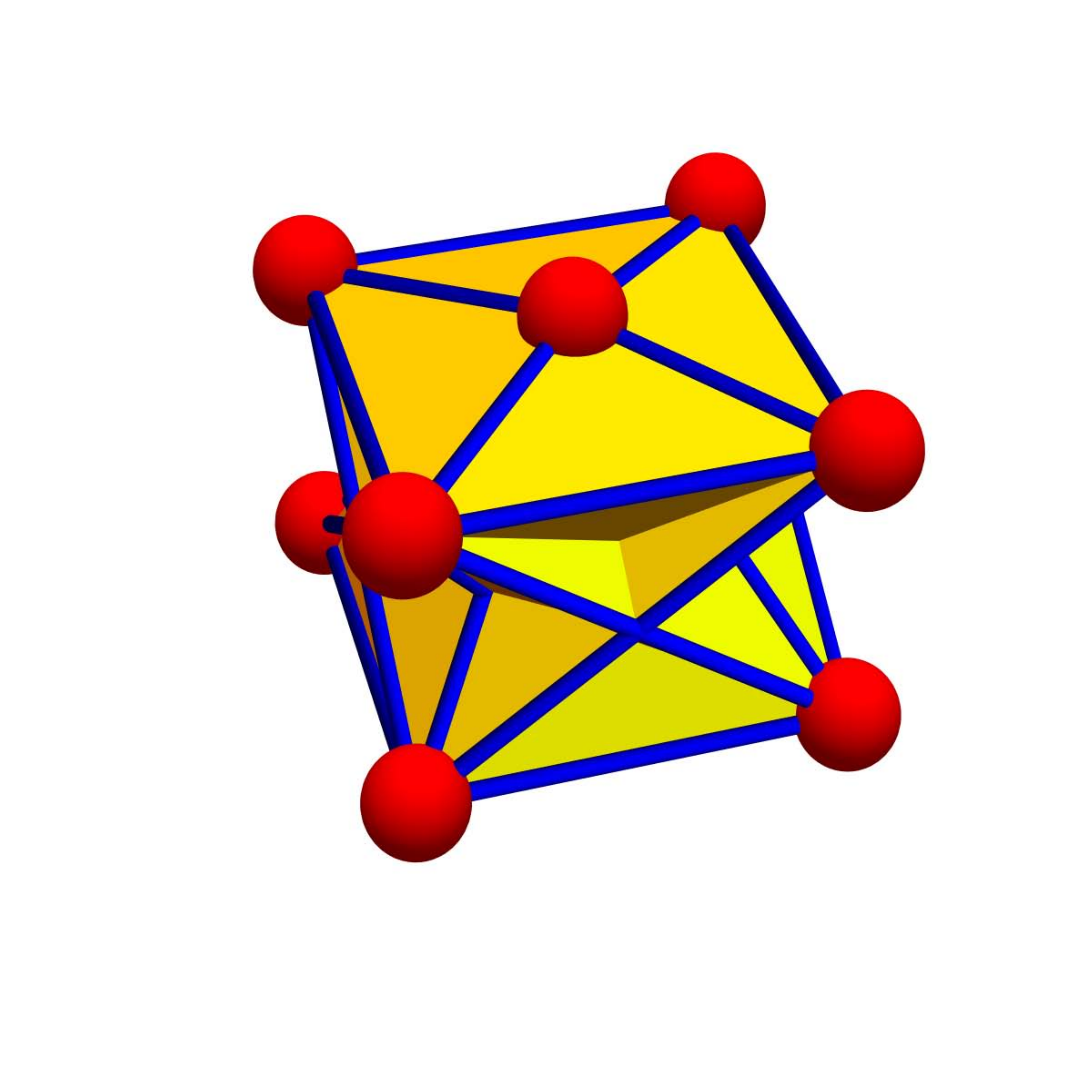}}
\scalebox{0.08}{\includegraphics{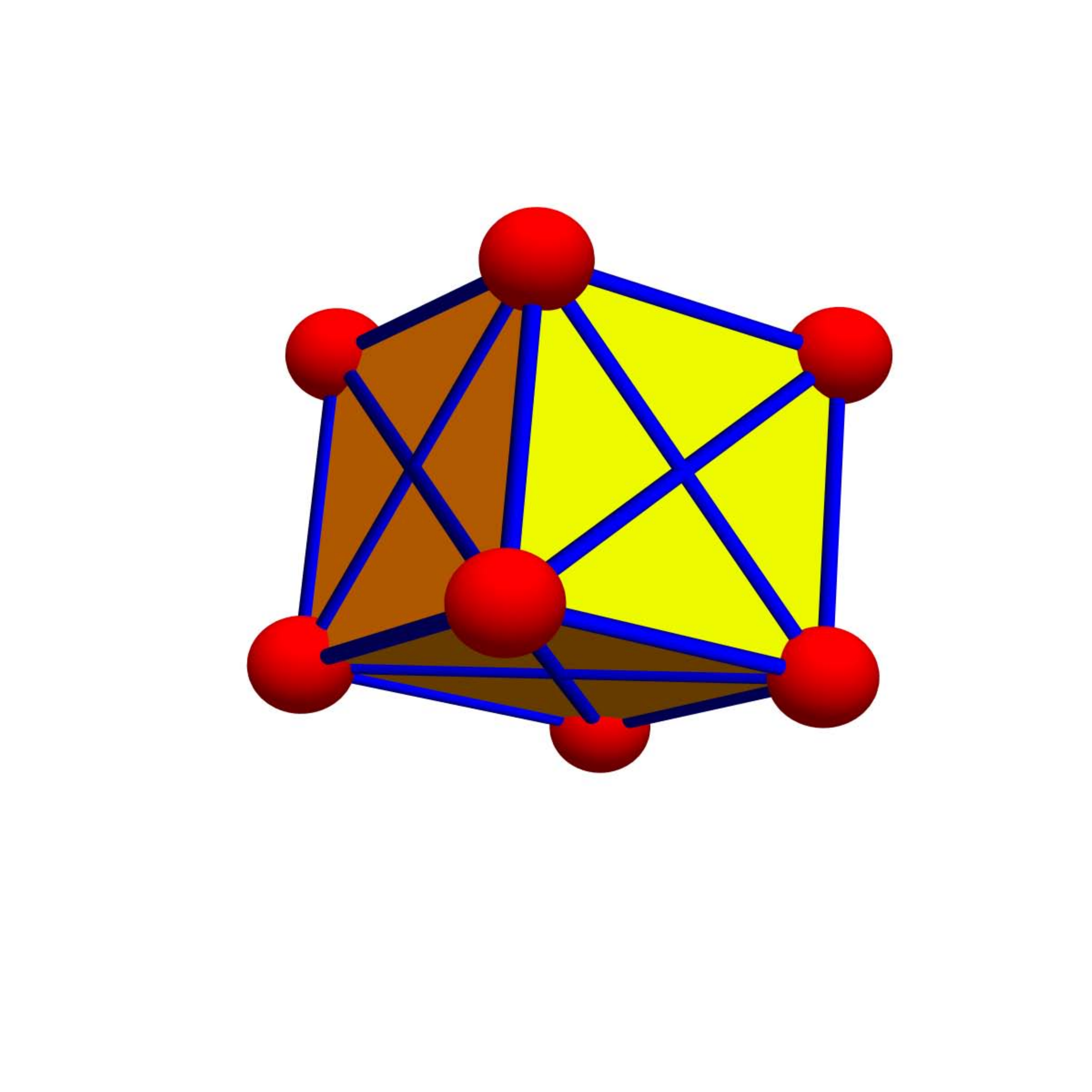}}
\scalebox{0.08}{\includegraphics{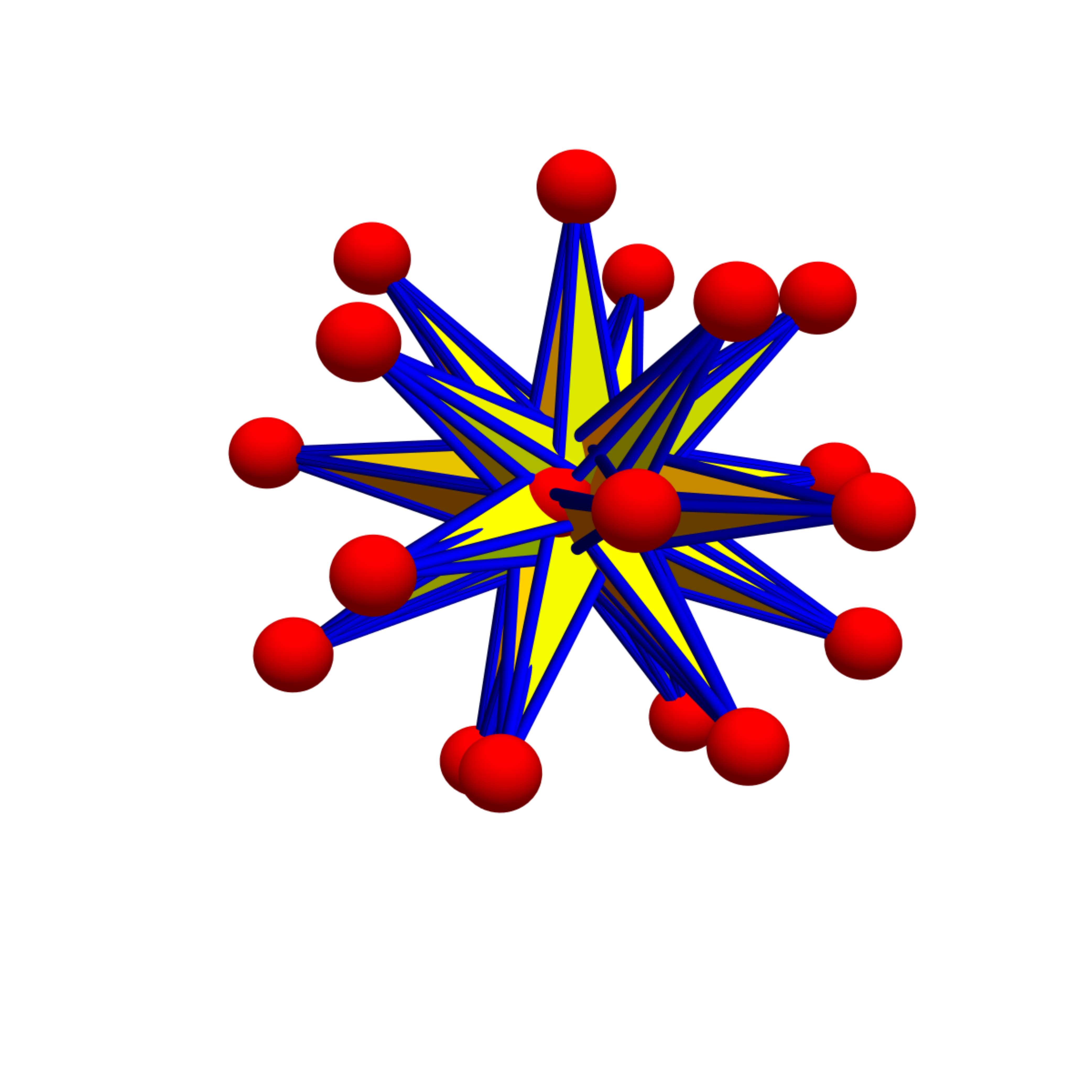}}
\caption{
Some graphs in the subring of $P$-graphs. 
We see first $S_6+P_2$, then $6+P_2$, then 
$2 P_2^2$, then $2 P_2^2 +P_2$, then
$4 P_2$ and finally $P_2+(1+P_6) P_4$. 
}
\end{figure}

\section{Additive Zykov primes}

\paragraph{}
There are two type of primes in the Zykov ring of complexes. The {\bf additive primes} for the 
monoid $(\G,+)$ and then the {\bf multiplicative primes} in the monoid $(\G, \cdot)$. 
For the integers $\mathbb{Z}$, the additive prime factorization is trivial as the only additive prime $1$. 
In the Zykov ring $(\G,+, \cdot)$, the additive primes are a bit more interesting but
easy to characterize. Here again, $\overline{G}$ is the complement of the graph $G$
so that $\overline{P_n}=K_n$ and $\overline{K_n}=P_n$. 

\begin{lemma} 
A graph $G$ is an additive Zykov prime if and only if $\overline{G}$ is connected. 
\end{lemma}
\begin{proof}
The complement operation is a ring isomorphism from the Zykov ring $(\G,+ \cdot)$
to the strong ring $(\G,\oplus,\osquare)$, mapping $0$ to $0$, $1$ to $1$
and satisfying $\overline{H \cdot G} = \overline{H} \osquare \overline{G}$ and
$\overline{H + G} = \overline{H} \oplus \overline{G}$, both additive and multiplicative
primes in one ring correspond to additive and multiplicative primes in the other ring. 
In the additive monoid $(\G,\oplus)$, the primes are the connected graphs. 
\end{proof} 

{\bf Remarks}. \\
{\bf 1)} The volume of a graph is the number of facets, complete subgraphs $K_{d+1}$,
where $d$ is the maximal dimension. We first thought that the primality of the volume
assures that a graph is an additive prime. This is not the case since there are primes of
volume $1$. An example is $G=\overline{L_3 \oplus C_4}$, where $L_3$ is the linear graph of length $3$ and
$C_4$ the cyclic graph of length $4$. Now $G$ has prime volume $2$ but it is obviously not prime. Indeed,
one of the factors $\overline{L_3} = K_2$ has volume $1$. 

{\bf 2)} The property of being prime is not a topological one. There are spheres like the octahedron
which can be factored $O=P_2 + P_2 + P_2$, and then there circles like $C_5$ which are prime
and can not be factored. But as in the additive integer case, the "fundamental theorem of 
additive network arithmetic" is easy, also with a direct proof:

\thm{ The additive Zykov monoid has a unique additive prime factorization.  }

\begin{proof}
The disjoint union has a unique additive prime factorization. Now dualize. \\
{\bf Direct proof.}
Assume $G=A+B = C+D$, where $A$ is prime. Following the Euclid type lemma, we prove
that either $A=C$ or $A=D$. \\
If that is not true then $A \cap C$ and $A \cap D$ are
both non-empty. Lets look at the four intersections $A\cap C$, $A \cap D$, $B \cap C$ and
$B \cap D$. Because every element in $A \cap C$ is connected to $A \cap D$, we have $A=A \cap C+A \cap D$.
But this contradicts that $A$ is prime.
\end{proof}

\begin{coro}
Every element in the additive Zykov group of networks can be written as 
$G=U-V$, where $U,V$ are two graphs. 
\end{coro}
\proof{
The reason is that the cancellation property holds. $U+K=V+K$ implies $U=V$.
}

For integers we can always write an integer as either $n$ or $-n$. The reason
is that the unit $1$ is the only additive prime. But for
networks, we do not have an overlap in general. For example $C_5-C_7$ can not 
be written as a single network similarly as $5/7$ can not be simplified. Sometimes
we can like $K_5-K_7=-K_2$ as in the ring we have $K_5 = K_1 + K_1 + K_1 + K_1=5 K_1$
and $K_5-K_7=5 K_1 - 7 K_1 = -2 K_1$. 

\section{Multiplicative primes}

\paragraph{}
Proving the fundamental theorem of algebra for rational integers  $\mathbb{Z}$
was historically a bit more convoluted. 
As Andr\'e Weyl pointed out, there is a subtlety which was not covered by Euclid
who proved however an important lemma, which comes close. 
Examples of number fields like $\mathbb{Z}[\sqrt{5}]$ show that factorization is not obvious.
We pondered the problem of unique prime factorization for the Zykov product on our own 
without much luck. We finally consulted the handbook of graph products and got relieved as the answer
is known for the strong ring and realizing that this Sabidussi ring is dual to the Zykov ring. 
Lets look at the primes

\begin{lemma}
Every graph $G=(V,E)$ for which $|V|$ is prime is a multiplicative prime
in any of the three rings $(\G,\oplus,\square),(\G,\oplus,\otimes), (\G,\oplus,\osquare)$
as well as dual rings $(\G,+,\diamond),(\G,\oplus,*),(\G,+,\cdot)$. 
\end{lemma}

\proof{
For all ring multiplications, the cardinalities of the vertices multiplies. 
If one of the factors has $1$ vertex only, then this factor is $1=K_1$. 
The reason is that there is only one graph for which the vertex 
cardinality is $1$. }

{\bf Remark.} It is the last part of the proof which fails if we look at the volume
in the Zykov addition, for which volume is multiplicative. In that case, there are 
graphs with volume $1$ which are not equal to the unit $1=K_1$.  

\paragraph{}
We have seen that the Zykov sum $G+H$ of two graphs $G,H$ is always connected. 
Also the Zykov product has strong connectivity properties but the 
example $P_n \cdot P_m = P_{mn}$ shows that we don't necessarily have
connectivity in the product. 

\begin{lemma}
If either $G$ or $H$ is connected, then $G \cdot H$ is connected. 
\end{lemma}
\begin{proof} 
Given two points $(a,b),(c,d)$ and assume $G$ is connected. We
can connect $a_0=a,...,a_n=c$ with a path in $G$. Now, for 
any choice $b_1,...,b_n$ in $H$, we have a connection
connect $(a_0,b_0) \dots (a_n,b_n))$.
\end{proof}

\paragraph{}
As we will see below that there is a unique prime factorization for multiplication $\cdot$
as long as the graph $\overline{G}$ is connected. Here is a lemma which tells that for complete graphs, 
we have a unique prime factorization. This is not so clear as we had to extend the arithmetic. 
We have learned for example for Gaussian primes that primes like $p=5$ $Z$ do no
more stay primes in number field and that we can have non-unique prime factorization 
like $2 * 3 = (1+i \sqrt{5})(1-i \sqrt{5})$. 

\begin{lemma}
The graph $K_n$ uniquely factors into $K_{p_1} \cdot \cdots \cdot K_{p_m}$
where $n=p_1 \cdot p_m$ is the prime factorization of $n$. 
\end{lemma}    
\begin{proof}
{\bf Using duality.} The graphs $P_n$ in the strong ring $(\G,\oplus,\osquare)$ 
are multiplicative primes if and only if $n$ is a prime. By duality, this is inherited
by the ring $(\G,+ \cdot)$. \\
{\bf Direct.} If not, then $U \times V = x=K_n$ is a factorization of $K_n$, where
one of the $U$ is not a complete graph. We can assume without loss
of generality that it is $U$. There are then two vertices $v,w$ for which 
the edge $(v,w)$ is missing in $U$. But now also the vertex $((v,a),(v,a))$ 
is missing in the product. 
\end{proof}

\paragraph{}
{\bf Remark.} We could 
prove like this uniqueness for graphs of the form $K_p \cdot H$, where $H$ is prime: the reason
is that $K_p \cdot H = H+H+ \cdot +H$. Assume we can write it as $A B$, where $A$ has $p$ vertices, then
$A$ is a subgraph of $K_p$ and so also of the form $K_d$. 

\paragraph{}
Can we give a formula for the $f$-vector of the product and so get criteria for primality.
We have $v1 v2 = v$ and $e = e1 v2^2 + e2 v1^2 - 2 e1 e2$. Now this is a Diophantine problem for 
the unknowns $e1,e2$.  

\paragraph{}  
The following example is in \cite{ImrichKlavzar,HammackImrichKlavzar} for the 
strong ring. By duality we get:

\begin{propo}
The Zykov multiplication $\cdot$ does not feature unique prime factorization.
Complements of connected graphs have a unique prime factorization for the Zykov 
product $\cdot$. 
\end{propo}
\begin{proof}
This follows from Sabidussi's factorization result and 
by duality. The example of \cite{ImrichKlavzar} seen below works. 
\end{proof}

We had not been able to get counter example by random search. The smallest
example might be small. But an example given in \cite{ImrichKlavzar} 
for the Cartesian graph product works also for the Zykov product. 

\begin{figure}[!htpb]
\scalebox{0.3}{\includegraphics{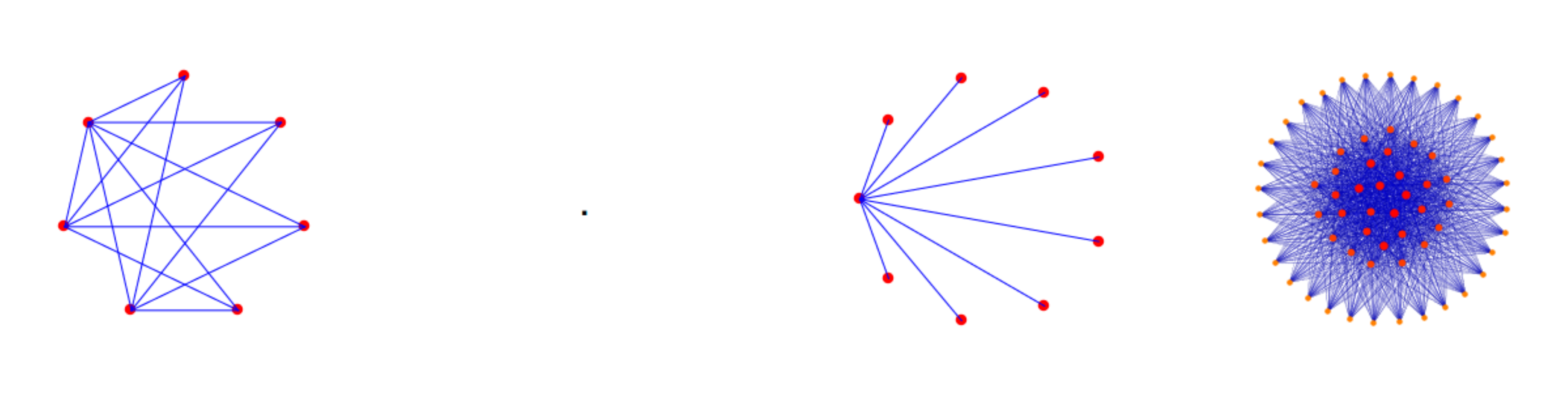}}
\scalebox{0.3}{\includegraphics{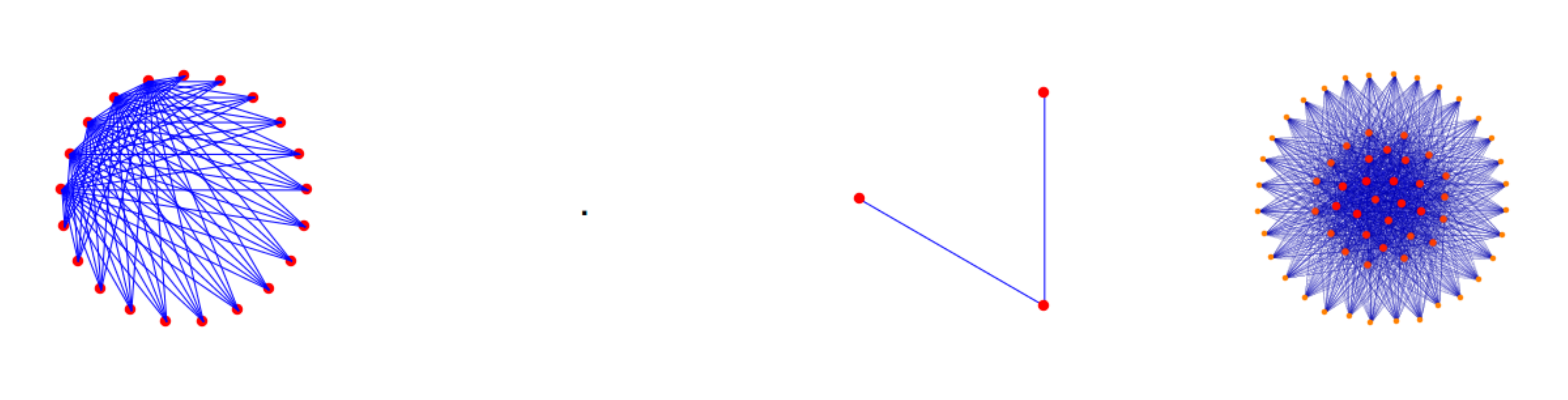}}
\caption{
Non-uniquness of factorization in the Zykov ring.
The graph which gives no uniqueness is 5 dimensional and
has $f$-vector $(63, 1302, 11160, 41664, 64512, 32768)$.
}
\end{figure}

{\bf Remarks.} \\
{\bf 1)} As pointed out in \cite{ImrichKlavzar}, the example is based on the fact that $\mathbb{N}[x]$ has no unique 
prime factorization: $(1+x+x^2)(1+x^3) =(1+x^2+x^4)(1+x)$. \\
{\bf 2)} Geometric realizations of the corresponding Whitney complexes produce non-unique prime factorizations 
of manifolds in the Grothendieck ring of manifolds, where the disjoint union is the addition and the Cartesian
product is the multiplication. 

\paragraph{}
By looking at geometric realizations, this
immediately shows that the Grothendieck ring of simplicial complexes 
with disjoint union and Cartesian (topological product) has no unique prime factorization. 
And since the examples produce manifolds with boundary: 

\begin{coro}
The Grothendieck ring of manifolds with boundary with disjoint union as addition and Cartesian product as multiplication
has no unique prime factorization.
\end{coro}

For related but much more subtle examples of the Grothendieck ring of varieties see \cite{Poonen2002}.

\paragraph{}
A Theorem of Sabidussi tells that
connected graphs have a unique prime factorization with respect to the Cartesian product
\cite{Sabidussi}. An other proof of this Sabidussi's theorem is given in \cite{Vizing}.
The tensor product has a non-unique prime factorization even in the 
connected case as shown in \cite{HammackImrichKlavzar}. 

\begin{figure}[!htpb]
\scalebox{0.3}{\includegraphics{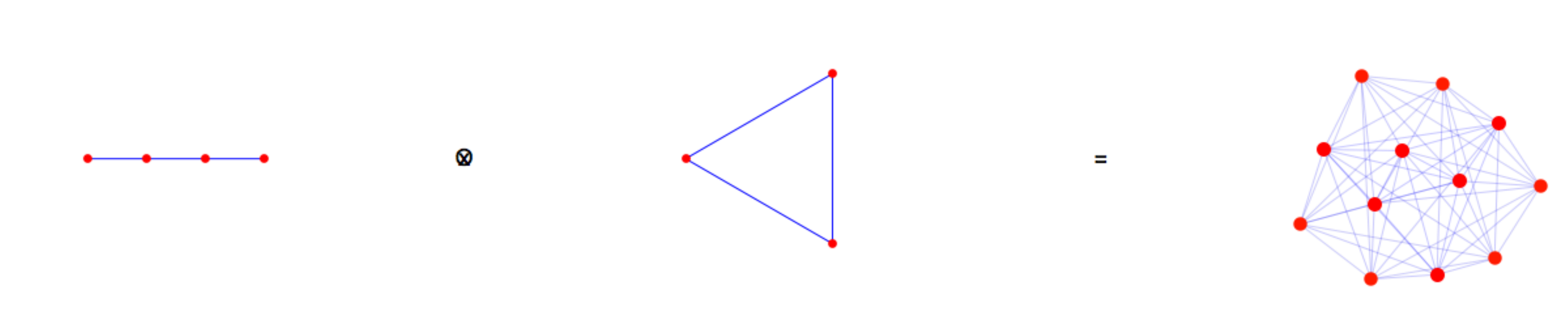}}
\scalebox{0.3}{\includegraphics{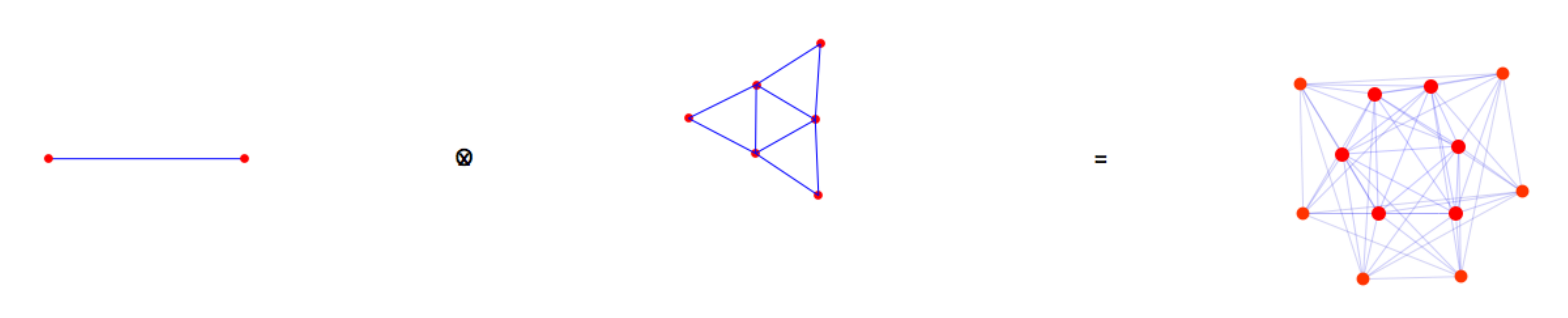}}
\caption{
Non-uniqueness of factorization in the tensor ring. 
The example is given in \cite{HammackImrichKlavzar}. 
}
\end{figure}

\section{Energy}

\paragraph{}
A finite abstract simplicial complex $G$ defines a {\bf connection matrix} $L$.
If $G$ contains $n$ sets, then $L$ is a $n \times n$ matrix with $L_{xy} = L(x,y)=1$ if the 
faces $x,y$ intersect $x \cap y \neq \emptyset$ and $L(x,y)=L_{xy}=0$ else. 
The connection matrix $L$ is always unimodular \cite{Unimodularity} so that its inverse $g$
integer valued. We also know the energy theorem $\sum_{x,y} g(x,y) = \chi(G)$ \cite{Helmholtz}. 
In the context of arithmetic, one can look at the product $G \times H$ of simplicial complexes. 
It is the set of all pairs $(A,B)$ with $A \in G, B \in H$. The product $G \times H$ is not a simplicial complex
as it is not closed under the operation of taking finite subsets. Still, one can look at the
{\bf connection matrix} of $G \times H$. Define $L(X,Y)=1$ if the two sets $X=(x,y)$ and $Y=(a,b)$ 
intersect and $L(X,Y)=0$ else. Let here $\G$ denote the ring of all simplicial complexes generated
by simplicial complexes, disjoint union as addition and with the just defined Cartesian product
as multiplication. The Euler characteristic of an element in this ring is still defined as
$\chi(G)=\sum_x (-1)^{{\rm dim}(x)}$. It follows almost by definition that 
$\chi(G \times H) = (\sum_{x \in G} (-1)^{{\rm dim}(x)}) (\sum_{y \in H} (-1)^{{\rm dim}(y)}
=\chi(G) \chi(H)$. The adjacency matrices tensor if the weak product of graphs
is taken. There is an analogue for connection Laplacians:

\begin{lemma}[Tensor connection lemma]
$L_{G \times H} = L_G \otimes L_H$. 
\end{lemma}
\begin{proof}
In the natural basis $e_i \times e_j$, the 
connection Laplacian of the product is a matrix
containing the matrix $L_H$ at the places where $L_G$ 
has entries $1$ and $0$ matrices else. 
\end{proof} 

Linear algebra gives

\begin{coro}
$\sigma(L(G \times H)) = \sigma(L(G)) \sigma(L(H))$.
\end{coro}

The energy theorem follows

\begin{coro}
$\sum_{x,y} g_{G \times H}(x,y) = \chi(G \times H)$. 
\end{coro}

\begin{proof}
As $L$ tensor also its inverse $g$ tensor. The energy 
as a sum over all matrix entries is therefore multiplicative. 
Also the Euler characteristic is multiplicative.
\end{proof}

{\bf Remarks.}  \\
{\bf 1)} Note that classically, for the {\bf Laplace Beltrami operator} $L$
of a manifold $G$, we have $\sigma(L(G \times H)) = \sigma(L(G)) + \sigma(L(H))$. 
Already in the discrete, the connection Laplacian has special and different features 
than the Hodge Laplacian $H=D^2=(d+d^*)^2$. First of all, it behaves more like the
Dirac operator $D=d+d^*$ in that there is negative spectrum, but it does not have the
symmetry $\sigma(D)=\sigma(-D)$ of the Dirac operator. \\
{\bf 2)} Already the Kirchhoff Laplacian $L_0=A-B = d_0^* d_0$ with adjacency matrix $A$
and diagonal vertex degree matrix $B$, which is the scalar part
$L_0$ of the Hodge Laplacian $H=(d+d^*)^2 = \oplus_{k=0}^d L_k$ has different properties
with respect to the product. The best already known compatibility is for the 
adjacency matrix $A(G)$ which tensors under the multiplication of the weak product 
(graph Cartesian product).  So, also there, the eigenvalues multiply. \\
{\bf 3)} One still has to explore the relevance of the connection Laplacian in a physics context. 
Both the Hodge Laplacian $H=(d+d^*)$ as well as the connection Laplacian $L$ are operators
on the same Hilbert space. The block entry $H_1$ obviously has relations to electromagnetism
as $L A = j$ is equivalent to the Maxwell equation $dA=F, dF=0, d^*F=j$ in a Coulomb gauge. 
It might well be that the connection Laplacian define some internal gravitational energy.
The invertibility of $L$ is too remarkable to not be taken seriously. Furthermore, if we let
the Dirac operator evolve freely in an isospectral way, this nonlinear dynamics also produces
a unitary deformation of the connection Laplacian, of course preserving its spectrum. 
But the energy values, the Green function entries $g(x,y)$ change under the evolution
preserving the total energy, the Euler characteristic. Of course, after the deformation, the
connection Laplacian is no more integer valued so that also its inverse is no more 
integer valued in general. 

\section{A Field of networks} 

\paragraph{}
For any of the rings under consideration we can now also look at the 
smallest field generated by this ring. 
As the cancellation property holds for the addition, we can represent elements 
in that field by $(a-b)/(c-d)$. As we have no unique prime factorization for the
multiplication, we have to keep the equivalence relation. \\

This is not that strange as we know that rational numbers can be equivalent
even so it is not obvious. Especially if we take $pq/(r p)$ for large primes $p,q,u,v$, 
where we can not factor the products easily. 

\paragraph{}
There has been some work on the complexity of factorization in multiplicative network
rings \cite{HammackImrichKlavzar}. For the weak product as well as for the direct 
product, one can get the factors fast. What about for the Zykov product? 
Can we define a graph version for modular arithmetic analogous how $\mathbb{Z}_p$ is obtained 
from $\mathbb{Z}$? Maybe just identify vertices? 
This would be an other way to get a field structure and could have cryptological applications. 

\paragraph{}
Here are some illustrations of how one can represent elements in the 
field generated by the Zykov ring:

\begin{figure}[!htpb]
\scalebox{1.0}{\includegraphics{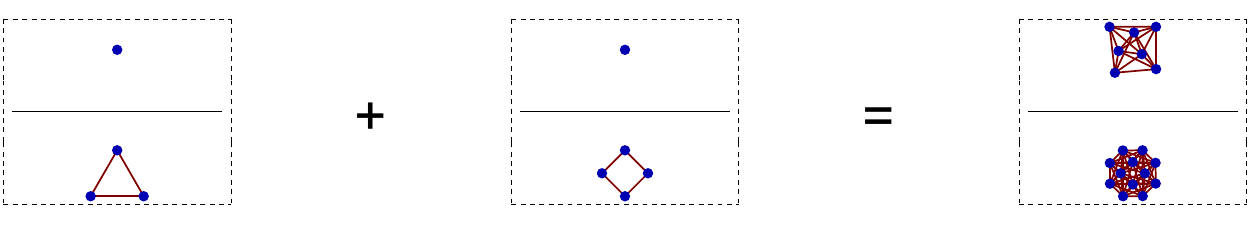}}
\caption{
A computation $1/K_3 + 1/C_4 = (C_4+K_3)/(3 C_4)$ in network arithmetic.
\label{fraction1}
}
\end{figure}

\begin{figure}[!htpb]
\scalebox{1.0}{\includegraphics{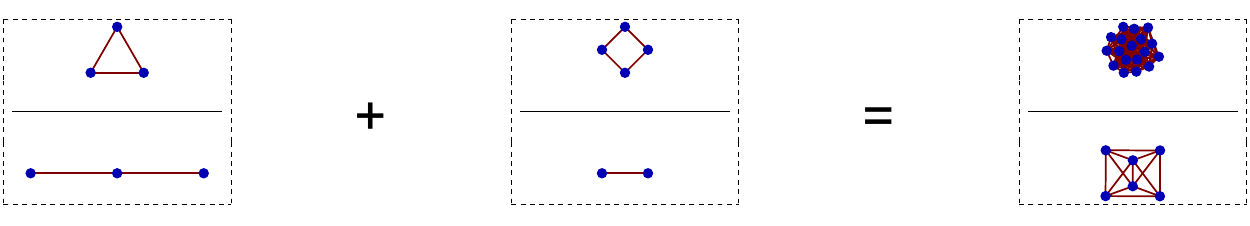}}
\caption{
A computation $K_3/S_3 + C_4/K_2 = (5+ S_3 C_4)/(2(1+P_3))$ in network arithmetic.
As $S_3 C_4 = (1+P_3) 2 P_2 = 2 P_2 + 2 P_6$ this is the $P$-fraction 
$(5+2P_2 +2 P_6)/(2+2 P_3)$. 
\label{fraction1}
}
\end{figure}

\paragraph{}
Lets just prove an Euclid like lemma. Note that this is not entirely trivial
as we have no unique factorization domain: 

\begin{lemma}
The square root of $2$ is irrational in the Zykov field.
\end{lemma}
\begin{proof}
Assume $\sqrt{2}=p/q$, then $2 q^2 = p^2$. This means $q^2 + q^2 = p^2$.
Since $q^2 + q^2$ is a sum, it is connected and features a unique prime
factorization by Sabidussi's theorem. Now the proof is the same as in 
the classical case as $2$ is a multiplicative prime and the number of prime
factors on the left and right are not the same modulo $2$. 
\end{proof}

\paragraph{}
More generally we can prove that for connected primes the square root
is irrational:

\begin{lemma}
For any connected graph $G$ which is a multiplicative prime in the Zykov ring,
the square root of $G$ is irrational.
\end{lemma}
\begin{proof}
It is the same classical argument. Writing $G \cdot Q^2 = P^2$ 
and noticing that for connected $G$, the product $G \cdot Q^2$ is 
connected, we have a unique prime factorization on both sides. 
Again the number of prime factors of $G$ modulo $2$ is different on the
left and right hand side. 
\end{proof}

{\rm Ramark.} \\
We currently do not know of any disconnected multiplicative prime $P$ 
for which the square root is rational.
\vfill

\pagebreak

\section{Code}

\paragraph{}
Here is Mathematica code and example computations as shown in
examples A-F in this text. \\

\begin{tiny}
\lstset{language=Mathematica} \lstset{frameround=fttt}
\begin{lstlisting}[frame=single]
NormalizeGraph[s_]:=Module[{r,v=VertexList[s],e=EdgeRules[s]},
 r=Table[v[[k]]->k,{k,Length[v]}];UndirectedGraph[Graph[v /.r,e /.r]]];

ZykovAdd[s1_,s2_]:=Module[{v,w,o,p,f,g,V,e,q,r,L=Length},
 v=VertexList[s1];   o=L[v]; f=EdgeList[s1]; q=L[f];
 w=VertexList[s2]+o; p=L[w]; g=EdgeList[s2]; r=L[g]; V=Union[v,w];
 f=If[L[f]==0,{},Table[f[[k,1]]-> f[[k,2]],{k,q}]];
 g=If[L[g]==0,{},Table[g[[k,1]]+o->g[[k,2]]+o,{k,r}]];
 e=Flatten[Union[{f,g,Flatten[Table[v[[k]]->w[[l]],{k,o},{l,p}]]}]];
 NormalizeGraph[UndirectedGraph[Graph[V,e]]]];        ZA=ZykovAdd;

ZykovProduct[s1_,s2_] :=Module[{v,w,f,g,n,o,p,V,e,q,r,A,L=Length},
 v=VertexList[s1]; n=L[v]; f=Union[EdgeList[s1]]; q=L[f];
 w=VertexList[s2]; o=L[w]; g=Union[EdgeList[s2]]; r=L[g];
 V=Partition[Flatten[Table[{v[[k]],w[[l]]},{k,n},{l,o}]],2];
 f=Table[Sort[{f[[k,1]],f[[k,2]]}],{k,q}];
 g=Table[Sort[{g[[k,1]],g[[k,2]]}],{k,r}]; e={}; A=Append;
 Do[e=A[e,{f[[k,1]],w[[m]]}->{f[[k,2]],w[[l]]}],{k,q},{m,o},{l,o}];
 Do[e=A[e,{v[[m]],g[[k,1]]}->{v[[l]],g[[k,2]]}],{k,r},{m,n},{l,n}];
 NormalizeGraph[UndirectedGraph[Graph[V,e]]]];        ZP=ZykovProduct; 

s1=RandomGraph[{9,20}];s2=RandomGraph[{9,20}];s3=RandomGraph[{9,20}];
A1=ZykovProduct[s1,s2];A2=ZykovProduct[s1,s3];B2=ZykovAdd[s2,s3];
A=ZykovAdd[A1,A2];B=ZykovProduct[s1,B2];Print[IsomorphicGraphQ[A,B]];

KK[n_]:=CompleteGraph[n];KK[n_,m_]:=CompleteGraph[{n,m}];CC=CycleGraph;
SS[n_]:=StarGraph[n+1];WW[n_]:=WheelGraph[n+1]; KI=ZA[KK[2],PP[2]];  
LL[n_]:=LinearGraph[n+1]; PP[n_]:=UndirectedGraph[Graph[Range[n],{}]]; 
OO=ZP[KK[3],PP[2]]; S3=ZP[KK[4],PP[2]]; WM=ZA[PP[3],KK[2]]; 
ZP[s_1,s2_,s3_]:=ZP[ZP[s1,s2],s3]; ZA[s1_,s2_,s3_]:=ZA[s1,ZA[s2,s3]]; 

(*  Example A                                                        *)
Print[IsomorphicGraphQ[ZP[CC[4],CC[4]] , ZA[KK[4,4],KK[4,4]]]];
Print[IsomorphicGraphQ[ZP[CC[4],CC[4]] , ZP[KK[2],KK[4,4]]]];
(*  Example B                                                        *)
Print[IsomorphicGraphQ[ZP[KK[2],CC[4]]  , S3]];
Print[IsomorphicGraphQ[ZP[KI,KI] ,   ZA[PP[4],KK[4],S3]] ];
(*  Example C                                                        *)
Print[IsomorphicGraphQ[ZP[SS[5], SS[5]] , ZA[ SS[25], KK[5,5]] ]];
(*  Example D                                                        *)
Print[IsomorphicGraphQ[ZP[WM, WM ] , ZA[ PP[9] , ZP[KK[4],SS[3]] ]]];
(*  Example E                                                        *)
Print[IsomorphicGraphQ[ZA[PP[5], PP[7]] , KK[5,7] ]];
Print[IsomorphicGraphQ[ZP[KK[2,3],KK[2,3]],ZA[PP[4],PP[9],KK[6,6]]]];
(* Example  F                                                        *)
Print[IsomorphicGraphQ[ZP[S3,S3 ] , ZP[ KK[16], PP[4] ] ]];
Print[IsomorphicGraphQ[ZP[OO,OO ] , ZP[ KK[9],  PP[4] ] ]];
\end{lstlisting}
\end{tiny}

\pagebreak

\paragraph{}
And here is example code illustrating that the Euler characteristic is
multiplicative on the strong ring. The Euler characteristic computation
uses the Poincar\'e-Hopf theorem allowing to reduce it to Euler characteristic
computations of part of unit spheres. 

\begin{tiny}
\lstset{language=Mathematica} \lstset{frameround=fttt}
\begin{lstlisting}[frame=single]
UnitSphere[s_,a_]:=Module[{b=NeighborhoodGraph[s,a]},
  If[Length[VertexList[b]]<2,Graph[{}],VertexDelete[b,a]]];
EulerChi[s_]:=Module[{vl,n,sp,u,g,sm,ff,a,k,el,m,q,A},
  vl=VertexList[s]; n=Length[vl];
  ff=Range[n]; el=EdgeList[s]; m=Length[el];
  g[b_]:=ff[[Position[vl,b][[1,1]]]];
  If[n==0,0,If[n==1 || m==Binomial[n,2],1,If[m==0,n,
  u=Table[ A=g[vl[[a]]]; sp=UnitSphere[s,vl[[a]]];
  q=VertexList[sp]; sm={};
  Do[If[g[q[[k]]]<A,sm=Append[sm,q[[k]]]],{k,Length[q]}];
  If[Length[sm]==0,1,(1-EulerChi[Subgraph[sp,sm]])],{a,n}];
  Sum[u[[k]],{k,n}]]]]];

NormalizeGraph[s_]:=Module[{r,v=VertexList[s],e=EdgeRules[s]},
 r=Table[v[[k]]->k,{k,Length[v]}];UndirectedGraph[Graph[v /.r,e /.r]]];
ZykovProduct[s1_,s2_] :=Module[{v,w,f,g,n,o,p,V,e,q,r,A,L=Length},
 v=VertexList[s1]; n=L[v]; f=Union[EdgeList[s1]]; q=L[f];
 w=VertexList[s2]; o=L[w]; g=Union[EdgeList[s2]]; r=L[g];
 V=Partition[Flatten[Table[{v[[k]],w[[l]]},{k,n},{l,o}]],2];
 f=Table[Sort[{f[[k,1]],f[[k,2]]}],{k,q}];
 g=Table[Sort[{g[[k,1]],g[[k,2]]}],{k,r}]; e={}; A=Append;
 Do[e=A[e,{f[[k,1]],w[[m]]}->{f[[k,2]],w[[l]]}],{k,q},{m,o},{l,o}];
 Do[e=A[e,{v[[m]],g[[k,1]]}->{v[[l]],g[[k,2]]}],{k,r},{m,n},{l,n}];
 NormalizeGraph[UndirectedGraph[Graph[V,e]]]];        ZP=ZykovProduct;
StrongProduct[s1_,s2_]:=Module[{t1,t2,t},
  t1=GraphComplement[s1]; t2=GraphComplement[s2]; t=ZykovProduct[t1,t2];
  NormalizeGraph[GraphComplement[t]]];

Do[
s1=RandomGraph[{10,13}];s2=RandomGraph[{12,9}]; ss=StrongProduct[s1,s2];
{a,b,c}={EulerChi[s1],EulerChi[s2],EulerChi[ss]}; 
Print["a=",a,",b=",b,",c=",c,"   a*b=c is ",a*b==c],{10}]
\end{lstlisting}
\end{tiny}

\pagebreak

\paragraph{}
Finally, here is example code illustrating that the connection Laplacian 
tensors when taking the Cartesian product and that the energy is the 
Euler characteristic. Also this code block is self contained and can be
grabbed by looking at the LaTeX source on the ArXiv.

\begin{tiny}
\lstset{language=Mathematica} \lstset{frameround=fttt}
\begin{lstlisting}[frame=single]
CliqueNumber[s_]:=Length[First[FindClique[s]]];
ListCliques[s_,k_]:=Module[{n,t,m,u,r,V,W,U,l={},L},L=Length;
  VL=VertexList;EL=EdgeList;V=VL[s];W=EL[s]; m=L[W]; n=L[V];
  r=Subsets[V,{k,k}];U=Table[{W[[j,1]],W[[j,2]]},{j,L[W]}];
  If[k==1,l=V,If[k==2,l=U,Do[t=Subgraph[s,r[[j]]];
  If[L[EL[t]]==k(k-1)/2,l=Append[l,VL[t]]],{j,L[r]}]]];l];
W[s_]:=Module[{F,a,u,v,d,V,LC,L=Length},V=VertexList[s];
  d=If[L[V]==0,-1,CliqueNumber[s]];LC=ListCliques;
  If[d>=0,a[x_]:=Table[{x[[k]]},{k,L[x]}];
  F[t_,l_]:=If[l==1,a[LC[t,1]],If[l==0,{},LC[t,l]]];
  u=Delete[Union[Table[F[s,l],{l,0,d}]],1]; v={};
  Do[Do[v=Append[v,u[[m,l]]],{l,L[u[[m]]]}],{m,L[u]}],v={}];v];
Fvector[s_]:=Delete[BinCounts[Length /@ W[s]],1];dim[x_]:=Length[x]-1;
Euler[s_]:=Module[{w=W[s],n},n=Length[w];Sum[(-1)^dim[w[[k]]],{k,n}]];
F[A_,z_]:=A-z IdentityMatrix[Length[A]]; F[A_]:=F[A,-1];

ConnectionLaplacian[s1_,s2_]:=Module[{c1=W[s1],c2=W[s2],n,A,c},
  n = Length[c1]*Length[c2]; c = {}; A = Table[1,{n},{n}];
  Do[c=Append[c,{c1[[k]],c2[[l]]}],{k,Length[c1]},{l,Length[c2]}];
  Do[ If[(DisjointQ[c[[k,1]],c[[l,1]]]||DisjointQ[c[[k,2]],c[[l,2]]]),
    A[[k,l]] = 0],{k,n},{l,n}]; A];
ConnectionGraph[s_] := Module[{c=W[s],n,A},n=Length[c];
  A=Table[1,{n},{n}];Do[If[DisjointQ[c[[k]],c[[l]]]||
  c[[k]]==c[[l]],A[[k,l]]=0],{k,n},{l,n}];AdjacencyGraph[A]];
ConnectionLaplacian[s_]:=F[AdjacencyMatrix[ConnectionGraph[s]]];

MyTensorProduct[A_,B_]:=Module[{n1=Length[B],n2=Length[A],n},
   n = n1*n2; Table[B[[1+Mod[k-1,n1],1+Mod[l-1,n1]]]*
   A[[1+Floor[(k-1)/n1],1+Floor[(l-1)/n1]]],{k,n},{l,n}]];

s=RandomGraph[{7,10}];t=RandomGraph[{7,10}];
A=ConnectionLaplacian[s]; B=ConnectionLaplacian[t];
L=ConnectionLaplacian[s,t];
M=MyTensorProduct[A,B]; Print[L==M];
Energy=Total[Flatten[Inverse[L]]]; Energy==Euler[s]*Euler[t]
a=Eigenvalues[1.0*A]; b=Eigenvalues[1.0*B];
c=Sort[Eigenvalues[1.0*M]];
d=Sort[Flatten[Table[a[[k]]*b[[l]],{k,Length[A]},{l,Length[B]}]]];
Print[Total[Abs[c-d]]] 
\end{lstlisting}
\end{tiny}

A Mathematica demonstration project, featuring the three graph products
$\square,\otimes,\osquare$ can be seen in \cite{Shemmer}.

\bibliographystyle{plain}

\begin{thebibliography}{10}

\bibitem{HammackImrichKlavzar}
R.~Hammack, W.~Imrich, and S.~Klav\v zar.
\newblock {\em Handbook of product graphs}.
\newblock Discrete Mathematics and its Applications (Boca Raton). CRC Press,
  Boca Raton, FL, second edition, 2011.
\newblock With a foreword by Peter Winkler.

\bibitem{HararyGraphTheory}
F.~Harary.
\newblock {\em Graph Theory}.
\newblock Addison-Wesley Publishing Company, 1969.

\bibitem{Hatcher}
A.~Hatcher.
\newblock {\em Algebraic Topology}.
\newblock Cambridge University Press, 2002.

\bibitem{ImrichKlavzar}
W.~Imrich and S.~Klavzar.
\newblock {\em Product graphs, Structure and recognition}.
\newblock John Wiley and Sons, Inc. New York, 2000.

\bibitem{KlainRota}
D.A. Klain and G-C. Rota.
\newblock {\em Introduction to geometric probability}.
\newblock Lezioni Lincee. Accademia nazionale dei lincei, 1997.

\bibitem{knillmckeansinger}
O.~Knill.
\newblock {The McKean-Singer Formula in Graph Theory}.
\newblock {\\}http://arxiv.org/abs/1301.1408, 2012.

\bibitem{DiracKnill}
O.~Knill.
\newblock The {D}irac operator of a graph.
\newblock {{\\}http://http://arxiv.org/abs/1306.2166}, 2013.

\bibitem{KnillKuenneth}
O.~Knill.
\newblock The {K\"u}nneth formula for graphs.
\newblock {{\\}http://arxiv.org/abs/1505.07518}, 2015.

\bibitem{valuation}
O.~Knill.
\newblock Gauss-{B}onnet for multi-linear valuations.
\newblock {\\}http://arxiv.org/abs/1601.04533, 2016.

\bibitem{Unimodularity}
O.~Knill.
\newblock On {F}redholm determinants in topology.
\newblock {\\}https://arxiv.org/abs/1612.08229, 2016.

\bibitem{Helmholtz}
O.~Knill.
\newblock On {H}elmholtz free energy for finite abstract simplicial complexes.
\newblock {\\}https://arxiv.org/abs/1703.06549, 2017.

\bibitem{Spheregeometry}
O.~Knill.
\newblock Sphere geometry and invariants.
\newblock {\\}https://arxiv.org/abs/1702.03606, 2017.

\bibitem{Kuenneth}
H.~K{\"u}nneth.
\newblock \"{U}ber die {B}ettischen {Z}ahlen einer {P}roduktmannigfaltigkeit.
\newblock {\em Math. Ann.}, 90(1-2):65--85, 1923.

\bibitem{Poonen2002}
B.~Poonen.
\newblock The {G}rothendieck ring of varieties is not a domain.
\newblock {\em Math. Res. Letters}, 9:493--498, 2002.

\bibitem{RourkeSanderson}
C.P. Rourke and B.J. Sanderson.
\newblock {\em Introduction to Piecewise-Linear Topology}.
\newblock Springer Verlag, Berlin, 1982.

\bibitem{Sabidussi}
G.~Sabidussi.
\newblock Graph multiplication.
\newblock {\em Math. Z.}, 72:446--457, 1959/1960.

\bibitem{Shemmer}
B.~Shemmer.
\newblock Graph products.
\newblock http://demonstrations.wolfram.com/GraphProducts, 2013.

\bibitem{VisualizationGraphProducts}
S.{J\"anicke}, C.~Heine, M.~Hellmuth, P.~Stadler, and G.~Scheuermann.
\newblock Visualization of graph products.
\newblock www.informatik.uni-leipzig.de, retrieved, Jun 13, 2017.

\bibitem{Vizing}
V.~G. Vizing.
\newblock The cartesian product of graphs.
\newblock {\em Vy\v cisl. Sistemy No.}, 9:30--43, 1963.

\bibitem{Wu1953}
Wu~W-T.
\newblock Topological invariants of new type of finite polyhedrons.
\newblock {\em Acta Math. Sinica}, 3:261--290, 1953.

\bibitem{Zykov}
A.A. Zykov.
\newblock On some properties of linear complexes. (russian).
\newblock {\em Mat. Sbornik N.S.}, 24(66):163--188, 1949.

\end{thebibliography}

\end{document}